\documentclass[11pt, a4paper]{article}
\usepackage{amsmath,amssymb, amsthm, amsfonts, hyperref, ifthen}

\usepackage{graphicx}
\newtheorem{theorem}{Theorem}[section]
\newtheorem{lemma}[theorem]{Lemma}
\newtheorem{proposition}[theorem]{Proposition}

\newtheorem{corollary}[theorem]{Corollary}
\theoremstyle{definition}

\newtheorem{example}[theorem]{Example}
\theoremstyle{remark}
\newtheorem{remark}[theorem]{Remark}
\newcommand{\be}{\begin{equation}}
\newcommand{\ee}{\end{equation}}
\newcommand{\beq}{\begin{equation}}
\newcommand{\enq}{\end{equation}}

\newcommand{\R}{{\mathbb{R}}}
\newcommand{\N}{{\mathbb{N}}}
\newcommand{\Z}{{\mathbb{Z}}}
\newcommand{\Q}{{\mathbb{Q}}}

\newcommand{\C}{{\mathbb{C}}}

\def\tg{{\widetilde{\gamma}}}

\def\({\left(}
\def\){\right)}

\newcommand{\supp}{{\mbox{\rm supp}}}

\renewcommand{\epsilon}{\varepsilon}

\renewcommand{\Im}{\operatorname{Im}}
\renewcommand{\Re}{\operatorname{Re}}
\newcommand{\Drc}[1]{{\operatorname{T}_{#1}}}
\newcommand{\spec}{\operatorname{spec}}
\newcommand{\gspec}[1]{{\Sigma_{#1}}}
\newcommand{\ir}{{\mathrm{i}}}
\newcommand{\dr}{{\mathrm{\,d}}}
\newcommand{\er}{{\mathrm{e}}}
\newcommand{\psib}{{\boldsymbol\psi}}
\newcommand{\phib}{{\boldsymbol\varphi}}
\newcommand{\Vclass}{{\mathbb{V}}}

\newcommand{\sob}[1][1]{H^{#1}}
\newcommand{\abs}[1]{\lvert{#1}\rvert}
\newcommand{\bigabs}[1]{\bigl\lvert{#1}\bigr\rvert}
\newcommand{\lrabs}[1]{\left\lvert{#1}\right\rvert}
\newcommand{\norm}[1]{\|#1\|}
\newcommand{\ipd}[2]{\langle{#1},{#2}\rangle}
\newcommand{\dom}{\operatorname{dom}}
\newcommand{\loc}{{\mathrm{loc}}}

\newcommand{\cir}{\mathbb{S}^1}
\newcommand{\sgn}{\operatorname{sgn}}
\newcommand{\CVclass}{{\widetilde{\mathbb{V}}}}
\newcommand{\BV}{BV_0}
\newcommand{\Tess}{{\Lambda_k}}

\newcommand{\asol}[2]{\phib^{#1}_{#2}}
\newcommand{\co}{\operatorname{co}}

\newcommand{\var}{\operatorname{var}}
\newcommand{\Ran}{\operatorname{Ran}}
\newcommand{\dist}{\operatorname{dist}}
\newcommand{\Nind}{\mathcal{N}_{\epsilon}}
\newcommand{\Ien}{I_{\epsilon,n}}
\newcommand{\Jen}{J_{\epsilon,n}}
\newcommand{\Ee}[1][\epsilon]{E_{#1}}
\newcommand{\Fe}{F_\epsilon}

\renewcommand{\o}{\overline}
\newcommand{\f}[1][]{\ifthenelse{\equal{#1}{}}{f}{f_{#1}}}
\newcommand{\Eng}[1]{E_{#1}}

\newcommand{\numz}[1][\phi]{\mathsf{N}_{#1}}
\newcommand{\numt}{\mathsf{m}}
\newcommand{\numzj}[1][n]{\mathsf{n}_{#1}}
\newcommand{\nstp}[1][m,M]{\mathsf{k}_{#1}}

\newcommand{\mcf}[1]{\widetilde{\chi}_{#1}}

\newcommand{\tp}{\mathcal{T}}
\newcommand{\stp}{\widetilde{\tp}}
\newcommand{\Itp}[1]{J_{#1}}
\newcommand{\rItp}[1]{\widetilde{J}_{#1}}
\newcommand{\Istp}[1]{K_{#1}}
\newcommand{\sEng}{U_{\kappa}}

\title{Eigenvalues of a one-dimensional Dirac operator pencil\thanks{MSC 34L40, 35P20}}
\author{
Daniel M.\ Elton\thanks{DME: Department of Mathematics and Statistics, Fylde College, Lancaster University, Lancaster LA1 4YF, United Kingdom; d.m.elton@lancaster.ac.uk;  \url{http://www.maths.lancs.ac.uk/\~elton/}}
\and 
Michael Levitin\thanks{ML: Department of Mathematics and Statistics, University of Reading, Whiteknights, PO Box 220, Reading RG6 6AX, United Kingdom; m.levitin@reading.ac.uk; \url{http://www.personal.reading.ac.uk/\~ny901965/}}
\and 
Iosif Polterovich\thanks{IP: D\'{e}partement de math\'{e}matiques et de statistique, Universit\'{e} de Montr\'{e}al, CP 6128, succ. Centre-ville, Montr\'{e}al, Qubec H3C 3J7, Canada; iossif@dms.umontreal.ca; 
\url{http://www.dms.umontreal.ca/\~iossif/}}
}
\renewcommand\footnotemark{}
\date{version September 28, 2013; \LaTeX ed \today}

\begin{document}

\maketitle
\begin{abstract}

We study the spectrum of a one-dimensional Dirac operator pencil, with a coupling constant in front of the potential considered as the spectral parameter. Motivated by recent investigations of graphene waveguides, we focus on the values of the coupling constant for which the kernel of the Dirac operator contains a square integrable function. In physics literature such a function is called a confined zero mode. Several results on the asymptotic distribution of coupling constants giving rise to zero modes are obtained. In particular, we show  that  this distribution depends in a subtle way on the sign variation and the presence of gaps in the  potential. Surprisingly, it also depends on the arithmetic properties of certain quantities determined by the potential.  We further observe that variable sign potentials may produce complex eigenvalues of the operator pencil. Some examples and numerical calculations illustrating these phenomena are presented.

\end{abstract}
\section{Introduction and main results}
\label{sec:intro}
\subsection{Statement of the problem} Consider the system of differential equations
\begin{equation}
\label{eq:basiceqv1}
\begin{aligned}
\bigl[V(x)-(\lambda-k)\bigr]\psi_1-\dfrac{\dr\psi_2}{\dr x}&=0,\\
\dfrac{\dr\psi_1}{\dr x}+\bigl[V(x)-(\lambda+k)\bigr]\psi_2&=0,
\end{aligned}
\end{equation}
on $\R$, where $k,\lambda$ are parameters and $V$ is a potential. 
Equivalently one may define a self-adjoint operator by
\[
\Drc{V}=\begin{pmatrix}
V+k&-\nabla\\
\nabla&V-k
\end{pmatrix}
=-i\sigma_2\nabla+k\sigma_3+V,
\]
where $\nabla=\dfrac{\dr}{\dr x}$  and $\sigma_2,\sigma_3$ are Pauli matrices.  Then \eqref{eq:basiceqv1} becomes the eigenvalue equation
$\Drc{V}\psib=\lambda\psib$,
where  $\psib=\begin{pmatrix}\psi_1\\\psi_2\end{pmatrix}$.

For a given potential $V$ let us set $\lambda=0$ and introduce the \emph{$\gamma$-spectrum} associated with $V$:
\[
\gspec{V}=\bigl\{\gamma\in\C: 0\in\spec(\Drc{\gamma V})\bigr\}.
\]
Equivalently $\gspec{V}$ is the spectrum of the linear operator pencil $\gamma\mapsto\Drc{0}+\gamma V$. 
Our goal is to understand the properties of $\gspec{V}$, such as symmetries, existence of real and complex (non-real) eigenvalues, eigenvalue estimates and asymptotics. Similar problems, as well as some other related questions,  have been studied in a variety of situations in mathematical literature --- see, for instance, \cite{BL}, \cite{GGHKSSV}, \cite{K}, \cite{Sa}, \cite{Schm}.

Whilst the general asymptotic behaviour and estimates in our case are generally in line with earlier results (see Theorems \ref{thm:lowerasym} and \ref{thm:singsigngvalasym}; we should note that our methods allow the widest class of potentials),  some unexpectedly subtle phenomena occur depending on the properties of $V$. In particular, $\gspec{V}$ may have a totally different structure for single-sign and variable-sign potentials (compare Theorems \ref{thm:singsignreal} and \ref{thm:oddV0}), as well as for potentials having gaps (that is, whose support is not connected) and for no-gap potentials 
(see Examples \ref{ex:ex3} and \ref{ex:ex4}). Also, variable-sign potentials can produce some (or even all) non-real eigenvalues, which have not been studied previously (see Theorem~\ref{thm:oddV0} and Example~\ref{ex: ex2}).

In physical literature this problem appears in the study of  electron waveguides in graphene (see \cite{HRP}, \cite{SDP} and references therein). Note that the electron dynamics in graphene is governed by the two-dimensional massless Dirac operator, and  the one-dimensional system \eqref{eq:basiceqv1} is obtained as a result of the separation of variables: the parameter $k$ corresponds to the frequency of the wavefunction $\psib$ in the direction parallel to the waveguide. 
From the physical viewpoint solutions  $\psi_1,\psi_2\in L^2$ are of particular interest; these are called {\it confined modes}.   
Among them, especially important in the study of conductivity properties of graphene are {\it zero modes}: $L^2$-solutions corresponding to $\lambda=0$. (See Section \ref{sec:nzmodes} for discussion of modes corresponding to $\lambda\ne 0$.)
Zero-energy states in graphene have also been studied for potentials of  other types --- see, for instance, \cite{BT}, \cite{BF}.
It was shown in \cite{HRP} that for the potential $V_{\text{HRP}}(x)=-1/{\cosh(x)}$  the solutions of the system \eqref{eq:basiceqv1}
can be found explicitly in terms of special functions. Moreover, there exists an infinite sequence of coupling constants $\gamma$ such that $0$ is an
eigenvalue of the operator $\Drc{\gamma V_{\text{HRP}}}$. 
An attempt to formulate and prove precise mathematical statements confirming and generalising the results of \cite{HRP} was the starting point of our research.

\subsection{Basic results} 
\label{sec:basic}
To state precise results we need to make some basic restrictions on the local regularity and global decay of the potential $V$. We shall assume all potentials are real valued and locally $L^2$. Let $\Vclass_0$ denote the class of such potentials which additionally satisfy
\[
\text{$\norm{V}_{L^2(x-1,x+1)}\to0$ as $\abs{x}\to\infty$;}
\]
roughly, $V\in\Vclass_0$ if it decays at infinity. In the literature $\Vclass_0$ is sometimes denoted as $c_0(L^2)$.

We can define the constant coefficient operator $\Drc{0}$ as a multiplication operator in Fourier space. If $V\in\Vclass_0$ we show that $V$ is a relatively compact perturbation of $\Drc{0}$, allowing us to define $\Drc{V}$ as an unbounded self-adjoint operator on $L^2$ (see Section \ref{subsec:argugen} for more details). The same construction can be used for complex-valued potentials (although, of course, the resulting operator will no longer be self-adjoint); this allows us to consider $\Drc{\gamma V}$ for any $\gamma\in\C$. Further use of the relative compactness of $V$ leads to the following:

\begin{theorem}
\label{thm:disgspec}
If $V\in\Vclass_0$ then $\gspec{V}$ is a discrete subset of $\C$.
\end{theorem}

\begin{remark}
\label{rem:lambdaspec}
\emph{Standard spectrum.}
The (usual) spectrum of the self-adjoint operator $\Drc{0}$ can be computed easily by considering it as a multiplication operator in Fourier space; we get 
\[
\spec(\Drc{0})=\R\setminus(-\abs{k},\abs{k})=:\Tess,
\]
while this spectrum is purely absolutely continuous. Since $V\in\Vclass_0$ is a relatively compact perturbation of $\Drc{0}$ the operators $\Drc{V}$ and $\Drc{0}$ must have the same essential spectrum (see \cite[section XIII.4]{RSI}); thus
\begin{equation}
\label{eq:TVessspec}
\spec_{\rm ess}(\Drc{V})=\spec_{\rm ess}(\Drc{0})=\Tess.
\end{equation}
The operator $\Drc{V}$ may have eigenvalues outside $\Tess$ but these must be isolated and of finite multiplicity (using the fact that we're dealing with a $1$-dimensional problem it is not hard to show that these eigenvalues must in fact be simple; a somewhat restricted form of this result is given in Lemma \ref{lem:simeval}).
\end{remark}

In common with other Dirac operators, $\Drc{V}$ possesses a number of elementary symmetries which lead to symmetries for the set $\gspec{V}$. In particular, if $V\in\Vclass_0$ then $-\gspec{V}=\gspec{V}=\overline{\gspec{V}}$, while $\gspec{V}$ is unchanged if we replace $k$ with $-k$ in the definition of $\Drc{0}$. With this last symmetry in mind we shall henceforth assume $k>0$; this will enable us to simplify the statement of some results.

To obtain estimates for the distribution of points in $\gspec{V}$ we impose extra global decay conditions on the potential $V$. Let $\Vclass_1$ denote the class of real valued locally $L^2$ potentials which satisfy
\[
\text{$\displaystyle\int_\R\,\abs{V(x)}\dr x<+\infty$;}
\]
that is, we require $V$ to be integrable. Equivalently we can define $\Vclass_1=\Vclass_0\cap L^1$. The class $\Vclass_1$ is sometimes denoted as $\ell^1(L^2)$.

Firstly we consider the number of points of $\gspec{V}$ lying inside the disc $\{z\in\C:\abs{z}\le R\}$ of radius $R\ge0$.

\begin{theorem}
\label{thm:genVL1upbnd}
Suppose $V\in\Vclass_1$. Then
\[
\#\bigl(\gspec{V}\cap\{z\in\C:\abs{z}\le R\}\bigr)\le C\,\norm{V}_{L^1}R
\]
for any $R\ge0$, where $C$ is a universal constant (we can take $C=4\er/\pi$).
\end{theorem}

This result can be generalised (using a rather different approach) to deal with potentials $V\in\Vclass_0$ which have weaker decay than is required to be $L^1$; see Theorem \ref{thm:cfnupperbnd}.

Lower bounds which complement the upper bounds given by Theorem \ref{thm:genVL1upbnd} can also be obtained. Restricting our attention to real points we have the following:

\begin{theorem}
\label{thm:lowerasym}
Suppose $V\in\Vclass_1$. Then
\[
\#(\gspec{V}\cap[0,R])\ge\frac{R}{\pi}\left\lvert\int_{\R}V(x)\dr x\right\rvert +o(R)
\]
as $R\to\infty$, while the same estimate holds for $\#(\gspec{V}\cap[-R,0])$ (by symmetry).
In particular, $\gspec{V}\cap\R$ contains infinitely many points if $\int_{\R}V(x)\dr x\neq0$.
\end{theorem}

\subsection{Single-signed potentials} 
\label{sec:signdef}
In general the set $\gspec{V}$ may contain complex eigenvalues (see Section \ref{sec:egs} for some examples of explicit potentials which illustrate various possible behaviours for complex points in $\gspec{V}$). Note that, even though the operator $\Drc{V}$ is self-adjoint (recall that $V$ is real valued), it does not follow in general that the corresponding operator pencil should have a purely real spectrum.  However, if $V$ does not change sign (as in the example considered in \cite{HRP}) all eigenvalues of the operator pencil are real:

\begin{theorem}
\label{thm:singsignreal}
If $V\in\Vclass_0$ is single-signed then $\gspec{V}\subset\R$.
\end{theorem}

By symmetry we can write $\gspec{V}=\{\pm\gamma_n:n\in\N\}$ where $0<\gamma_1<\gamma_2<\dots$ denotes the sequence of positive points in $\gspec{V}$, arranged in order of increasing size. The bound in Theorem \ref{thm:lowerasym} can be turned into an asymptotics: 

\begin{theorem}
\label{thm:singsigngvalasym}
Suppose $V\in\Vclass_1$ is single-signed. Then
\[
\#(\gspec{V}\cap[0,R])=\frac{R}{\pi}\left\lvert\int_{\R}V(x)\dr x\right\rvert +o(R)
=\frac{\norm{V}_{L^1}}{\pi}\,R+o(R)
\]
as $R\to\infty$. If $V$ is non-zero we can equivalently write
\[
\gamma_n=\frac{\pi}{\norm{V}_{L^1}}\,n+o(n)
\]
as $n\to\infty$.
\end{theorem}

\subsection{Anti-symmetric potentials} 
\label{sec:odd}

For potentials of variable sign the behaviour of the $\gamma$-spectrum may be different, in some cases quite drastically so. For anti-symmetric potentials we have the following:
\begin{theorem}
\label{thm:oddV0}
If $V\in\Vclass_0$ is anti-symmetric then $\gspec{V}\cap\R=\emptyset$.
\end{theorem}

Note that, the $\gamma$-spectrum may still contain an infinite number of complex eigenvalues; see Example \ref{ex: ex2} below.

The absence of real points in the $\gamma$-spectrum together with Theorem \ref{thm:singsigngvalasym} shows that the lower bound obtained in
Theorem \ref{thm:lowerasym} is quite sharp.

\begin{remark}
\label{rem:translateV}
It is easy to see that translating a potential $V$ changes the operator $\Drc{\gamma V}$ to something which is unitarily equivalent. In particular, 
Theorem \ref{thm:oddV0} also applies to potentials $V$ satisfying the condition $V(a+x)=-V(a-x)$ for some $a\in\R$ and all $x\in\R$. The translation invariance of our problem will also be used to simplify the presentation of some arguments in Section \ref{sec:argu}.
\end{remark}

\subsection{Potentials without gaps} 
\label{sec:gaps}
Let $\BV$ denote the class of compactly supported real valued functions of (totally) bounded variation. Clearly $\BV\subset\Vclass_1$ while $\BV$ contains compactly supported piecewise constant potentials with a finite number of pieces, as well as compactly supported functions in $C^1$. We say that a potential \emph{$V\in\BV$ has no gaps} if
\[
\bigabs{\co(\supp(V))\cap V^{-1}(0)}=0,
\]
where $\abs{S}$ and $\co(S)$ denote the Lebesgue measure and convex hull of a set $S\subseteq\R$ respectively.

\begin{theorem}
\label{thm:Vnogapaysm}
Suppose $V\in\BV$ has no gaps. Then
\[
\#(\gspec{V}\cap[0,R])=\frac{R}{\pi}\left\lvert\int_\R V(x)\dr x\right\rvert+O(1)
\]
as $R\to\infty$. The same estimate holds for $\#(\gspec{V}\cap[-R,0])$ (by symmetry).
\end{theorem}

\begin{remark}
\label{rem:nogapint0}
When $\int_{\R}V(x)\dr x=0$ this result simply states that $\gspec{V}\cap\R$ is finite.
\end{remark}

\subsection{Discussion}
\label{sec:discussion}
Our results give information about the asymptotics of the counting function $\#(\gspec{V}\cap[0,R])$ as $R\to\infty$. 
For any $V\in\Vclass_1$ the results of Section \ref{sec:basic} give asymptotic upper and lower bounds of
\begin{equation}
\label{up&lowasybnd:eq}
C\frac{R}{\pi}\int_{\R}\abs{V(x)}\dr x
\quad\mbox{and}\quad
\frac{R}{\pi}\left\lvert\int_{\R}V(x)\dr x\right\rvert
\end{equation}
respectively. 
Using Theorem \ref{thm:genVL1upbnd} we can take $C=2\er$, in which case the upper bound is actually uniform for $R\ge0$. For an asymptotic upper bound the constant can be reduced to at least $C=\er$ in general (see Remark \ref{rem:otherests}). Theorem \ref{thm:singsigngvalasym} shows the constant can be reduced further to $C=1$ for single-signed potentials; in this case the asymptotic upper and lower bounds agree and an asymptotic formula for the points in $\gspec{V}$ is obtained. For variable-signed potentials the quantities $\int_\R\abs{V(x)}\dr x=\norm{V}_{L^1}$ and $\bigabs{\int_\R V(x)\dr x}$ differ, leading to differences in the upper and lower bounds in \eqref{up&lowasybnd:eq} even if we could take $C=1$. For no-gap potentials $V\in\BV$ Theorem \ref{thm:Vnogapaysm} shows that it is the lower bound that actually gives the leading order term in the asymptotics of $\#(\gspec{V}\cap[0,R])$ as $R\to\infty$.

The above results may lead to a hypothesis that, in fact, the lower bound always gives the leading order term in the asymptotics of the counting function of the $\gamma$-spectrum.  However, as we show in the next section, this is not the case. Moreover, the  precise asymptotic behaviour of $\#(\gspec{V}\cap[0,R])$ as $R\to\infty$ may depend on the properties of a variable--signed potential in a rather subtle way. In particular,  it is sensitive to the presence of  gaps, that is, intervals where $V\equiv0$, appearing between components of $\supp(V)$.  Even more surprisingly, the leading term of the asymptotics is affected by
the arithmetic properties of certain quantities determined by the potential, such as the rationality of the ratio $\bigabs{\int_\R V(x)\dr x}/\norm{V}_{L^1}$.

\subsection{One-gap potentials, zeros of trigonometric functions and arithmetic}
\label{sec:onegap}

Suppose $V\in\BV$. We say \emph{$V$ has one gap} if we can write $V=V_1+V_2$ for some non-zero $V_1,V_2\in\BV$ which have no gaps and disjoint supports. 
For $j=1,2$ the support of $V_j$ is a closed bounded interval; write $\supp(V_j)=[a_j,b_j]$.
Without loss of generality we may assume the support of $V_1$ lies to the left of that of $V_2$. 
Then $b_1<a_2$ and the gap is the interval $(b_1,a_2)$.
Set
\[
v_j=\int_{a_j}^{b_j} V(x)\,\dr x=\int_\R V_j(x)\,\dr x
\]
for $j=1,2$. Thus $\int_\R V(x)\,\dr x=v_1+v_2$ while $\abs{v_1}+\abs{v_2}\le\norm{V}_{L^1}$, with equality iff $V_1$ and $V_2$ are each single-signed.

\begin{theorem}
\label{thm:1gint0} 
If $\int_\R V(x)\dr x=0$ then $\gspec{V}\cap\R$ contains only finitely many points. 
\end{theorem}

\begin{remark}
\label{rem:twogapzerointegral}
This result extends Remark \ref{rem:nogapint0} to one gap potentials. However the same result does \emph{not} extend to zero integral potentials with two gaps; see Example \ref{ex:ex4}.
\end{remark}

We now suppose $v_1+v_2=\int_{\R}V(x)\,\dr x\neq0$. Set
\[
\alpha=\tanh(k(a_2-b_1))
\quad\text{and}\quad
\beta=\lrabs{\frac{v_1-v_2}{v_1+v_2}}.
\]
Then $\alpha\in(0,1)$ gives a measure of the gap length, while $0\le\beta<1$ if $v_1v_2>0$ and $\beta>1$ if $v_1v_2<0$. In particular, if $V$ is single-signed then $\beta < 1$.
When $\alpha\beta>1$ we can further define
\begin{equation}
\label{eq:defncoeffu}
\nu_{\alpha,\beta}=\frac2\pi\left[\beta\arcsin\frac{\sqrt{\alpha^2\beta^2-1}}{\sqrt{\beta^2-1}}
+\arcsin\frac{\sqrt{1-\alpha^2}}{\alpha\sqrt{\beta^2-1}}\right].
\end{equation}
If we fix $\beta>1$ and allow $\alpha$ to vary from $1/\beta$ to $1$ it is easy to check that $\nu_{\alpha,\beta}$ varies continuously and monotonically from $1$ to $\beta$.

If $\beta$ is positive and rational write $\beta=p/q$ where $p,q\in\N$ are coprime. 
If $p$ and $q$ are both odd set $p_\beta=p$ and $q_\beta=q$; if $p$ and $q$ have opposite parity set $p_\beta=2p$ and $q_\beta=2q$. 

Set
\begin{equation}
\label{eq:Adef}
A(\alpha,\beta)
=\begin{cases}
1&\text{if $\alpha\beta<1$,}\\
\nu_{\alpha,\beta}&\text{if $\alpha\beta>1$ and $\beta\notin\Q$,}\\
\displaystyle\frac4{q_\beta}\left\lfloor\frac14(p_\beta+q_\beta\nu_{\alpha,\beta})\right\rfloor-\frac{p_\beta}{q_\beta}+\frac2{q_\beta}
&\text{if $\alpha\beta>1$ and $\beta\in\Q$;}
\end{cases}
\end{equation}
we are using $\lfloor x\rfloor$ to denote the largest integer which does not exceed $x$.

\begin{theorem}
\label{thm:1ggen}
Suppose $\alpha\beta\neq1$. If $\alpha\beta>1$ and $\beta\in\Q$ suppose additionally that $p_\beta+q_\beta\nu_{\alpha,\beta}\notin4\Z$.
Then
\[
\#(\gspec{V}\cap[0,R])=\frac{1}{\pi}\,A(\alpha,\beta)\,\abs{v_1+v_2}\,R+o(R)
\]
as $R\to\infty$. The same estimate holds for $\#(\gspec{V}\cap[-R,0])$ (by symmetry).
\end{theorem}

\begin{remark}
If $\alpha\beta>1$ and $\beta\in\Q$ then the bounds $x-1\le\lfloor x\rfloor\le x$ give
\[
\nu_{\alpha,\beta}-\frac{2}{q_\beta}\le A(\alpha,\beta)\le\nu_{\alpha,\beta}+\frac{2}{q_\beta}.
\]
Therefore, for any sequence of rational numbers $\beta_n$ converging to $\beta \notin \Q$, we have $A(\alpha, \beta_n) \to A(\alpha,\beta)$, and hence
$A(\alpha, \beta)$ is continuous at irrational values of $\beta$.   At the same time it is clear that  $A(\alpha,\beta)$ has discontinuities at many rational values of $\beta$. Let us note that continuity at irrational values and discontinuity at rational values of a parameter was observed for other physically meaningful quantities --- see, for instance, \cite{GGL} (where the mathematical setting is somewhat similar to ours),  as well as  \cite{AMS}, \cite{JM}.
\end{remark}

Theorem \ref{thm:1ggen}  comes almost directly from a result about the zeros of a perturbed trigonometric function.
Consider the equation
\begin{equation}
\label{eq:cospert0}
\cos(x)+\alpha\cos(\beta x)+\phi(x)=0
\end{equation}
where $\phi$ satisfies the decay condition
\begin{equation}
\label{eq:decayprop012}
\phi\in C^2(\R),\quad\text{$\phi^{(n)}(x)=o(1)$ as $x\to\infty$ for $n=0,1,2$.}
\end{equation}

\begin{theorem}
\label{thm:coszeros} 
Let $0\le \alpha <1$, $\beta\ge0$ and $\alpha\beta\neq1$. 
If $\alpha\beta>1$ and $\beta\in\Q$ let us additionally assume that $p_\beta+q_\beta\nu_{\alpha,\beta}\notin 4\Z$.
Also suppose that $\phi$ satisfies \eqref{eq:decayprop012} and the solutions of \eqref{eq:cospert0} form a discrete subset of $\R$.
Then
\[
\#\bigl\{x\in[0,R]:\text{$x$ satisfies \eqref{eq:cospert0}}\bigr\}=\frac1\pi\,A(\alpha,\beta)\,R+o(R)
\]
as $R\to\infty$. 
\end{theorem}

Note that, $n=2$ in condition \eqref{eq:decayprop012} is only needed in the case that $\alpha\beta>1$ and $\beta\notin\Q$.

Consideration of Theorem \ref{thm:coszeros} in the case $\phi\equiv0$ goes back at least as far as \cite{St} where the irrational case was established (a somewhat different problem was considered in the rational case).

Using Theorem \ref{thm:1ggen} it is possible to show that we can obtain an asymptotic formula
\begin{equation}
\label{eq:genCasym}
\#(\gspec{V}\cap[0,R])=\frac{C}{\pi}\,R+o(R)
\end{equation}
as $R\to\infty$, where $C$ can indeed take any value (strictly) between $\bigabs{\int_\R V(x)\,\dr x}$ and $\norm{V}_{L^1}$; we state this as a separate result. 

\begin{theorem}
\label{thm:mrallasym}
Let $0<v<A<u$. Then there exists a piecewise constant one gap potential $V$ such that $\bigabs{\int_\R V(x)\,\dr x}=v$,  $\norm{V}_{L^1}=u$ and \eqref{eq:genCasym} holds with $C=A$.
\end{theorem}

\subsection{Remarks on non-zero modes}
\label{sec:nzmodes}
If we consider the eigenvalues of $\Drc{\gamma V}$ as functions of $\gamma$ we can view $\gspec{V}$ as the set of points at which these curves cross $0$. One could equally consider crossings at any other point $\lambda$ belonging to $(-k,k)$ (the spectral gap of the operator $\Drc{0}$). This leads to consideration of the set 
\[
\gspec{\lambda,V}=\bigl\{\gamma\in\C: \lambda\in\spec(\Drc{\gamma V})\bigr\}.
\]
With some straightforward modifications most of our analysis for $\gspec{V}$ can be carried over to $\gspec{\lambda,V}$ for any $\lambda\in(-k,k)$. We now summarise the changes to the main results.

Theorem \ref{thm:disgspec} holds for $\gspec{\lambda,V}$. For $V\in\Vclass_0$, $\gspec{\lambda,V}$ is still symmetric under conjugation and unchanged if we replace $k$ with $-k$; however, we cannot expect $\gspec{\lambda,V}$ to be symmetric about $0$ in general (this symmetry generalises to $-\gspec{\lambda,V}=\gspec{-\lambda,V}$). Theorem \ref{thm:oddV0} does not generalise.

Theorems \ref{thm:genVL1upbnd}, \ref{thm:lowerasym}, \ref{thm:singsignreal} and \ref{thm:singsigngvalasym} hold for $\gspec{\lambda,V}$ with two adjustments; firstly, the constant $C$ in Theorem \ref{thm:lowerasym} may depend on $\lambda$, and secondly, the results for points in $\gspec{\lambda,V}\cap\R^-$ no longer follow ``by symmetry'' (but can be obtained by similar arguments).

The latter comment also applies to Theorem \ref{thm:Vnogapaysm}, which otherwise holds for $\gspec{\lambda,V}$ in the case that $\int_\R V(x)\dr x\neq0$. When $\int_\R V(x)\dr x=0$ we need to impose further conditions on $\lambda$ (to ensure we avoid limiting values of the eigenvalues of $\Drc{\gamma V}$ as $\gamma\to\pm\infty$; cf.\ \cite{GGHKSSV}, Theorem 8.2(i)). 

Theorem \ref{thm:1gint0} does not admit a straightforward generalisation to the case $\lambda \neq 0$. 
A generalisation of Theorem \ref{thm:1ggen} will require Theorem \ref{thm:coszeros}  to be extended to cover equations of the form $\cos(x)+\alpha\cos(\beta x+\delta)+\phi(x)=0$, where $\delta\in\R$ is an additional parameter (cf. \cite{St}  for the case $\phi\equiv 0$, $\beta\not\in\Q$).

\subsection{Organisation of the paper}

Section \ref{sec:egs} is devoted to examples. 
The main arguments, together with a number of auxiliary constructions and results, are collected in Section \ref{sec:argu}. 
Theorem \ref{thm:disgspec} is essentially standard; its proof appears in Section \ref{subsec:argugen}.
Some key ideas from the Pr\"{u}fer method are introduced in  Section \ref{subsec:boundargu}. In particular, we re-characterise the set $\gspec{V}\cap\R$ in terms of a quantity $\Delta_V$, which is closely related to the Pr\"{u}fer argument (see Proposition \ref{prop:gspeccharD}). The asymptotic behaviour of $\Delta_V$ is described  (Proposition \ref{prop:Dgasym}) and leads directly to Theorem \ref{thm:lowerasym}. Theorem \ref{thm:singsigngvalasym} follows from a related argument, together with the additional monotonicity of $\Delta_V$ for single-signed $V$ (as described in Proposition \ref{prop:positiveDinc}).
An alternative approach based on the Birman-Schwinger principle that could be used in the case  of single-signed potentials  is discussed in Remark \ref{BS}.

A uniform bound on $\Delta_V$ (given in Proposition \ref{prop:genDeltaVest}) leads through several intermediate results to the proof of Theorem \ref{thm:genVL1upbnd}.  
Derivatives of $\Delta_V$ are considered in Sections \ref{subsec:deriv} and \ref{subsec:nogaps}. The justification of the monotonicity result (Proposition \ref{prop:positiveDinc}) appears in Section \ref{subsec:deriv}, while in Section \ref{subsec:nogaps} the proof of Theorem \ref{thm:Vnogapaysm} is reduced to some technical estimates (given in Proposition \ref{prop:intcossinest}). Theorem  \ref{thm:1gint0} is established in Section \ref{sub:onegap} while Theorem \ref{thm:1ggen} is reduced to Theorem  \ref{thm:coszeros}; Theorem \ref{thm:mrallasym} is then obtained as a straightforward consequence of the former.

For the sake of clarity the proofs of the results in Section \ref{sec:argu} which require more technical arguments are deferred to Section \ref{sec:tech}. In Section \ref{subsec:asym} we consider Lemma \ref{lem:+thetalim}, Section \ref{subsec:DeltaV} deals with Propositions \ref{prop:Dgasym} and \ref{prop:genDeltaVest}, and in Section \ref{subsec:estPhiJPsiJ} we establish Proposition \ref{prop:intcossinest}.

The last part of the paper is devoted to the proof of Theorem  \ref{thm:coszeros}, which is a variation on a classical theme of independent interest (cf.\ \cite{St}, \cite{Kac}, \cite{KKW}).  Some preliminary lemmas are established in Section \ref{sec:zprelim}, Section \ref{sec:zf} contains the proof of Theorem \ref{thm:coszeros} in the unperturbed case $\phi\equiv0$, while the general case is completed in Section \ref{sec:zpert}.

\subsection*{Acknowledgments}  The authors are grateful to  R.\ Frank, S.\ Jitomirskaya, A.\ Laptev, M.\ Portnoi, A.\ Pushnitski, Z.\ Rudnick, M.\ Solomyak  and T.\ Weidl for useful discussions. 
The research of I.P.\ is partially supported by NSERC, FQRNT and Canada Research Chairs program.

\section{Examples}
\label{sec:egs}

\subsection{General description}

The main purpose of this section is to illustrate the results stated above. We restrict our attention mostly to piecewise constant potentials with compact support; these allow the easiest analysis and already demonstrate the full range of effects. Consider points $a_0<a_1<\dots<a_m$ which partition the real line into $m$ finite intervals  $I_j=(a_{j-1},a_j)$, $j=1,\dots,m$, and two semi-infinite intervals $I_-=(-\infty,a_0)$ and $I_+=(a_m, +\infty)$. Consider a potential 
\begin{equation}
\label{eq:Vpiece}
V(x)=W\bigl(x; [a_0,\dots,a_m]; \{v_1,\dots,v_m\}\bigr):= 
\begin{cases}
v_j&\text{if $x\in I_j$, $j=1,\dots,m$},\\ 
0&\text{if $x\in I_-\cup I_+$},
\end{cases}  
\end{equation}
with some given real constants $v_j$.  On each interval, we need to solve the equations 
\begin{equation}
\label{eq:odeV}
\begin{aligned}
\nabla\psi_1&=(k-\gamma V)\psi_2,\\
\nabla\psi_2&=(k+\gamma V)\psi_1,
\end{aligned}
\end{equation}
with $V(x)=v_j=\operatorname{const}$, and then match the solutions to ensure continuity at the points $a_j$.

The following result is straightforward.

\begin{lemma}
\label{lem:constV}
For a given constant potential $V(x)=v$ such that $k\ne\pm\gamma v$, the system \eqref{eq:odeV} has the general solution 
\[
\begin{pmatrix}
\psi_1\\
\psi_2
\end{pmatrix}(x)=
C^{(1)}\begin{pmatrix}\sin\bigl(\sqrt{\gamma^2 v-k^2}x\bigr)\\-\sqrt{\frac{\gamma v+k}{\gamma v-k}}\cos\bigl(\sqrt{\gamma^2 v-k^2}x\bigr)\end{pmatrix}+
C^{(2)} \begin{pmatrix}\cos\bigl(\sqrt{\gamma^2 v-k^2}x\bigr)\\\sqrt{\frac{\gamma v+k}{\gamma v-k}}\sin\bigl(\sqrt{\gamma^2 v-k^2}x\bigr)\end{pmatrix}.
\]
If $v=0$ this solution can be equivalently written as 
\begin{equation}
\label{eq:V0soln}
\begin{pmatrix}
\psi_1\\
\psi_2
\end{pmatrix}(x)=
C^{(1)}\begin{pmatrix}1\\1\end{pmatrix}\er^{kx}+
C^{(2)}\begin{pmatrix}1\\-1\end{pmatrix}\er^{-kx}.
\end{equation}
In both cases $C^{(1)}$, $C^{(2)}$ are arbitrary complex constants.
\end{lemma}

\begin{remark}
\label{rem:endofinterval}
Lemma \ref{lem:constV}, or more precisely equation \eqref{eq:V0soln}, immediately implies that for any compactly-supported (not necessarily piecewise-constant) potential $V(x)$,  any eigenfunction $\psib\in L^2(\R)$ satisfies, under the assumption $k>0$, 
\begin{equation}
\label{eq:endofpotconds}
\psi_1(\min\supp(V))=\psi_2(\min\supp(V)),\qquad\psi_1(\max\supp(V))=-\psi_2(\max\supp(V)),
\end{equation}
in order to match the $L^2$ solutions at $\pm\infty$.
\end{remark}

Let us return to  the case of a piecewise constant potential  \eqref{eq:Vpiece}. The solution on each interval $I_j$, $j=1,\dots, m$, can be written down using Lemma \ref{lem:constV} with $v=v_j$ and $C^{(\ell)}= C^{(\ell)}_j$, $\ell=1,2$. 
By Remark \ref{rem:endofinterval}, we have $\psi_1(a_0)=\psi_2(a_0)$, and $\psi_1(a_m)=-\psi_2(a_m)$. Together with continuity conditions at each $a_j$, $j=1,\dots,m-1$ this leads to the homogeneous linear system of $2m$ equations with respect to $2m$ unknowns  $C^{(\ell)}_j$, $\ell=1,2$, $j=1,\dots,m$. Denote the determinant of the corresponding matrix of coefficients by $D_V(\gamma)$. As we are looking for a non-trivial $L^2$-solution $\psib$, we have $\gamma\in \gspec{V}$ if and only if 
\begin{equation}\label{eq:DVis0}
D_V(\gamma)=0.
\end{equation} 
Thus, in each particular case our problem is reduced to constructing  $D_V(\gamma)$ and finding its real or complex roots.

\subsection{Calculations, graphs, and further observations}

We visualise the real roots of $D_V(\gamma)$ by simply plotting its graph for real arguments. In the complex case we use the phase plot method (see \cite{WS}) in which the value of $\displaystyle \arg D_V(\gamma)=-\ir\log\frac{D_V(\gamma)}{|D_V(\gamma)|}$ is plotted using colours from a periodic scale. 
The roots of $D_V(\gamma)$ are singularities of $\arg D_V(\gamma)$ and appear on the phase plot as points at which all of the colours converge. The colour scale which we use in all such plots is shown in Figure \ref{fig:colorbar}.

\begin{figure}[!thb]
\begin{center}
\includegraphics{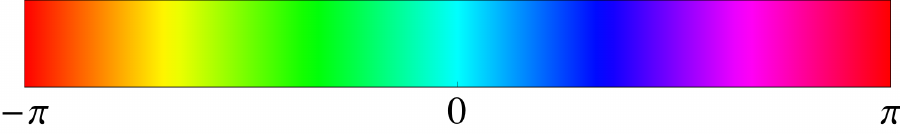}
\caption{\small Phase plot colour scale for the value of $ \arg D_V(\gamma)$.\label{fig:colorbar}}
\end{center}
\end{figure}

In the following examples it is convenient to set
\[
\tg=\sqrt{\gamma^2-1}.
\]
Also, we remark that our determinants $D_V(\gamma)$ are defined modulo a real or complex scaling constant, which we choose for convenience of presentation.  

\begin{example}[Illustration of Theorems \ref{thm:singsignreal} and \ref{thm:singsigngvalasym}]
\label{ex: ex1}
Set $V_1(x):=W(x; [-1, 1]; \{1\})$. Then 
\[
D_{V_1}(\gamma)=\frac{2(\tg\cos(2\tg)+\sin(2\tg))}{\gamma-1}
\]
As the potential is single-signed, the spectrum $\gspec{V}$ is real, as illustrated in Figures \ref{fig:examp1_plot} and  \ref{fig:examp1_phase}.

\begin{figure}[!hbt]
\begin{center}
\includegraphics{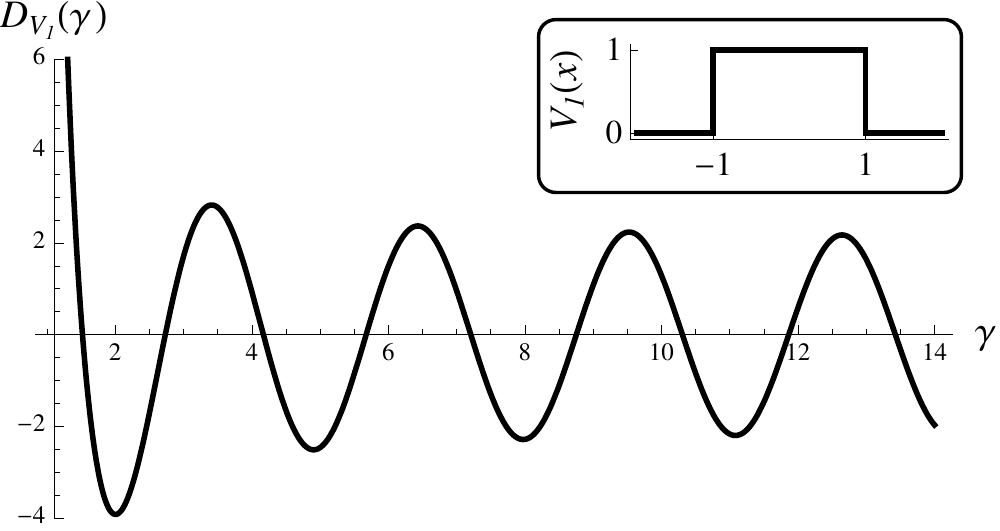}
\caption{\small The graph of $\displaystyle D_{V_1}(\gamma)=\frac{2(\tg\cos(2\tg)+\sin(2\tg))}{\gamma-1}$ against real $\gamma$ for the potential 
$V_1(x)=V(x; [-1, 1]; \{1\})$.\label{fig:examp1_plot}}
\end{center}
\end{figure}

\begin{figure}[!hbt]
\begin{center}
\includegraphics{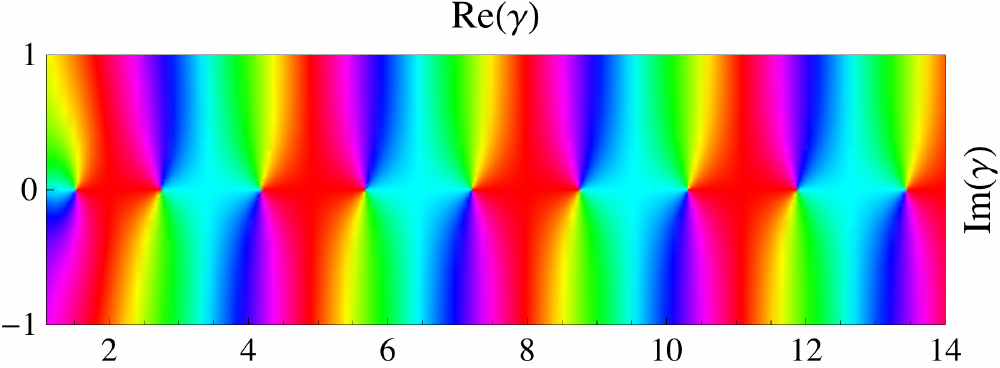}
\caption{\small The phase plot of $D_{V_1}(\gamma)$ for complex $\gamma$.\label{fig:examp1_phase}}
\end{center}
\end{figure}

For large (positive) values of $\gamma$, the solutions of \eqref{eq:DVis0} with $V=V_1$ are approximately those of $\cos(2\gamma) = 0$, i.e.,
\[
\gamma_n=\frac{\pi}{2}\,n+o(n)=\frac{\pi}{\norm{V_1}_{L^1}}\,n+o(n),
\] 
matching Theorem \ref{thm:singsigngvalasym}.
\end{example}

\begin{example}[Illustration of Theorem \ref{thm:oddV0}]\label{ex: ex2} 
Consider a class of anti-symmetric potentials $V_{2,g}(x):=W(x; [-1-g/2, -g/2, g/2,g/2+1]; \{-1, 0,1\})$ parametrised by the gap length $g\ge0$. Then, up to a multiplication by a non-zero constant, 
\[
D_{V_{2,g}}(\gamma)=
\frac{2\cosh(g)(\tg^2+1-\cos(2\tg)+\tg\sin(2\tg))+2\sinh(g)(\tg^2\cos(2\tg)+\tg\sin(\tg))}{\tg^2}.
\]
For any $g$, the potential $V_{2,g}$ is anti-symmetric; hence the spectrum $\gspec{V_{2,g}}$ is purely non-real and $D_{V_{2,g}}(\gamma)$ does not have any real roots. This is illustrated in Figure \ref{fig:examp2_plot} for $g=0$ and $g=1$.

\begin{figure}[!hbt]
\begin{center}
\includegraphics{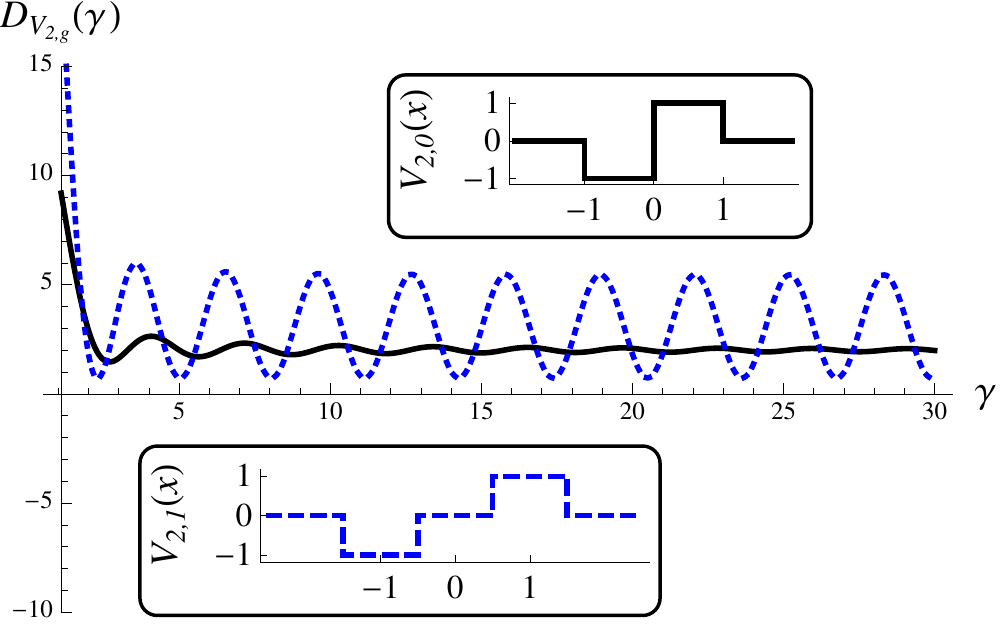}
\caption{\small The graphs of $D_{V_{2,g}}(\gamma)$ against real $\gamma$ for the potentials
$V_{2,g}(x):=W(x; [-1-g/2, -g/2, g/2,g/2+1]; \{-1, 0,1\})$ with $g=0$ (solid black line) and $g=1$ (dashed blue line).\label{fig:examp2_plot}}
\end{center}
\end{figure}

It turns out that the behaviour of complex eigenvalues for the potentials $V_{2,g}$ differs substantially for zero and non-zero gaps $g$. By a rather intricate asymptotic analysis of the corresponding transcendental equations (which in a sense extends Theorem \ref{thm:coszeros} to complex roots) we can show that the large eigenvalues with positive real parts are asymptotically located on the curves
\begin{equation}\label{eq:impart0}
\Im\gamma=\pm\frac{\ln\Re\gamma}{2}\qquad\text{if }g=0
\end{equation}
and on the straight lines
\begin{equation}\label{eq:impart1}
\Im\gamma=\pm\operatorname{arcsinh}\left(\frac{1}{\sinh g}\right)\qquad\text{if }g>0.
\end{equation}
 Figures  \ref{fig:examp2_phase0} and   \ref{fig:examp2_phase1} illustrate this behaviour of complex eigenvalues. For comparison, we also plot the corresponding curves \eqref{eq:impart0} and \eqref{eq:impart1}; one can see that the asymptotics is accurate even for the low eigenvalues.

\begin{figure}[!hbt]
\begin{center}
\includegraphics{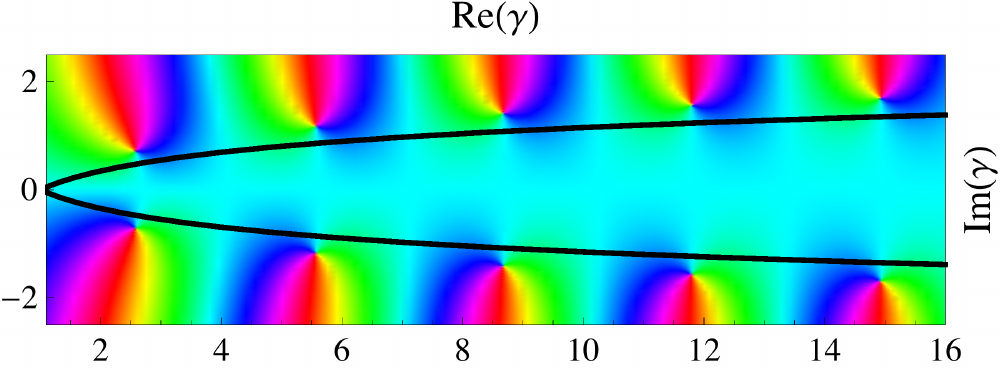}
\caption{\small The phase plot of $D_{V_{2,0}}(\gamma)$ for complex $\gamma$. The solid black curves \eqref{eq:impart0} illustrate the asymptotic behaviour of the imaginary parts of the eigenvalues.\label{fig:examp2_phase0}}
\end{center}
\end{figure}

\begin{figure}[!hbt]
\begin{center}
\includegraphics{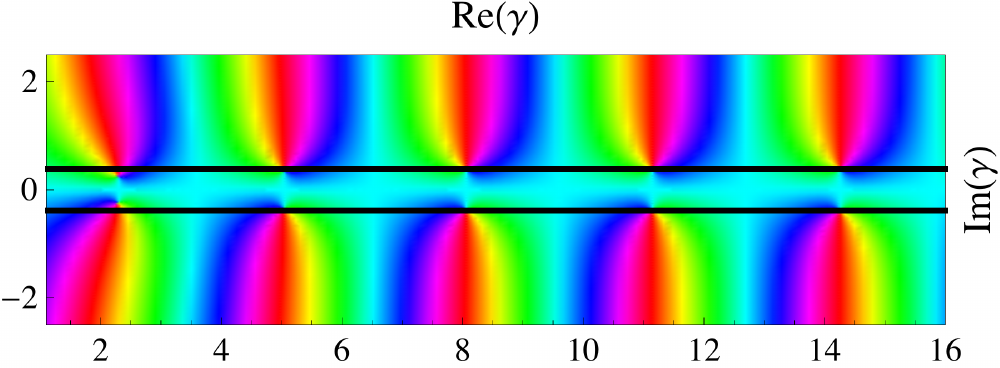}
\caption{\small The phase plot of $D_{V_{2,1}}(\gamma)$ for complex $\gamma$. The solid black lines \eqref{eq:impart1}, $g=1$, illustrate the asymptotic behaviour of the imaginary parts of the eigenvalues.\label{fig:examp2_phase1}}
\end{center}
\end{figure}

\end{example}

\begin{example}[Illustration of Theorem \ref{thm:1ggen}]
\label{ex:ex3}
Consider the one-gap potentials $V_{3,g,b}(x):=W(x;[-g-1,-g,0,b]; \{-1,0,1\})$ parametrised by the gap length $g\ge0$ and the maximum of the support $b>0$.  
For these potentials, $\int_\R V_{3,g,b}=b-1$ and
$\norm{V_{3,g,b}}_{L^1}=b+1$. Assume additionally $b\ne 1$. Explicit calculation gives, modulo multiplication by a constant,
\begin{align*}
D_{V_{3,g,b}}(\gamma)=\frac{2}{\tg^2}
&\left[\left((\tg^2+1)\cos((b-1)\tg) -\cos((b+1)\tg) + \tg\sin((b+1)\tg)\right)\cosh(g)\right.
\\
&\quad+\left.\left(\tg^2\cos((b+1)\tg)+\tg\sin((b+1)\tg)\right)\sinh(g)\right].
\end{align*}
The graphs of $D_{V_{3,g,2}}(\gamma)$ for real $\gamma$ and $g=0$ or $g=1$ are shown in Figure \ref{fig:examp3_plots}.
\begin{figure}[!hbt]
\begin{center}
\includegraphics{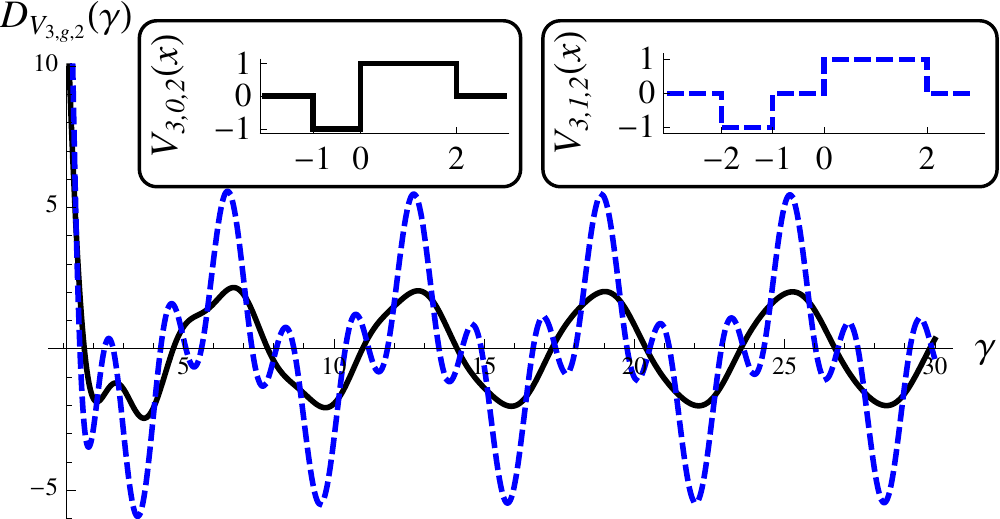}
\caption{\small The graphs of $D_{V_{3,g,2}}(\gamma)$ against real $\gamma$ for the potentials
$V_{3,g,2}(x):=W(x;[-g-1,-g,0,2]; \{-1,0,1\})$ with $g=0$ (solid black line) and $g=1$ (dashed blue line).\label{fig:examp3_plots}}
\end{center}
\end{figure}

We can expect asymptotics of the form
\[
\#(\gspec{V_{3,g,b}}\cap[0,R])=C_{g,b}\frac{R}{\pi}+o(R), 
\]
as $R\to\infty$. For the no-gap potential $V_{3,0,2}$, Theorem \ref{thm:Vnogapaysm} gives such an asymptotics with $C_{0,2}=1=\int_\R V_{3,0,2}$. On the hand, $D_{V_{3,1,2}}(\gamma)$ has three times as many real roots as $D_{V_{3,0,2}}(\gamma)$ (for sufficiently large $\gamma$). This leads to a constant $C_{1,2}=3=\norm{V_{3,1,2}}_{L^1}$ in the asymptotics for the one-gap potential $V_{3,1,2}$ as seen in Figure \ref{fig:examp3_plots}; c.f. the discussion in Sections~\ref{sec:gaps} and \ref{sec:discussion}.

This is just a partial case of a more complicated phenomenon, see Theorem \ref{thm:1ggen}. Set
\[
\alpha=\tanh(g)
\quad\text{and}\quad
\beta=\left|\frac{b+1}{b-1}\right|.
\]
After cancelling some non-zero factors equation \eqref{eq:DVis0} with $V=V_{3,g,b}$ takes the asymptotic form
\begin{equation}\label{eq:tildecos}
\cos((b-1)\tg)+\alpha\cos((b+1)\tg)+O\(\tg^{-1}\)=0
\end{equation}
as $\tg\to+\infty$ (where the first and second derivatives of the $O$-term are also $O\left(\tg^{-1}\right)$). 
Introducing the new variable $x=\abs{b-1}\tg$ leads to an equation in the form of \eqref{eq:cospert0}.
The asymptotics for the number of real zeros of \eqref{eq:tildecos} can then be obtained from Theorem \ref{thm:coszeros}. 
Alternatively, we can use  Theorem \ref{thm:1ggen} directly; both approaches give 
\[
C_{g,b}=A(\alpha,\beta)\,\abs{b-1}=A\(\tanh(g),\left|\frac{b+1}{b-1}\right|\)\abs{b-1},
\]
where $A(\alpha,\beta)$ is defined in \eqref{eq:Adef}.
\end{example}

\begin{example}[Illustration of a twin gap effect]
\label{ex:ex4}
The gap dependence illustrated in the previous Example can be made even more dramatic if we consider some special potentials. 
Introduce the symmetric twin gap potentials 
\[
V_{4,g}(x):=W\bigl(x; [-g-2, -g-1,-1,1,g+1,g+2]; \{-1,0,1,0,-1\}\bigr)
\]
parametrised by the gap length $g$.  
Note that $\int_\R V_{4,g}=0$ and $\norm{V_{4,g}}_{L^1}=4$ for any $g \ge 0$.
Figure  \ref{fig:examp4_plots}
shows the real curves $D_{V_{4,g}}(\gamma)$ for $g=0.5$ and $g=1$. One can see that there are only two real eigenvalues for the former, and an infinite number of real eigenvalues for the latter. 

\begin{figure}[!hbt]
\begin{center}
\includegraphics{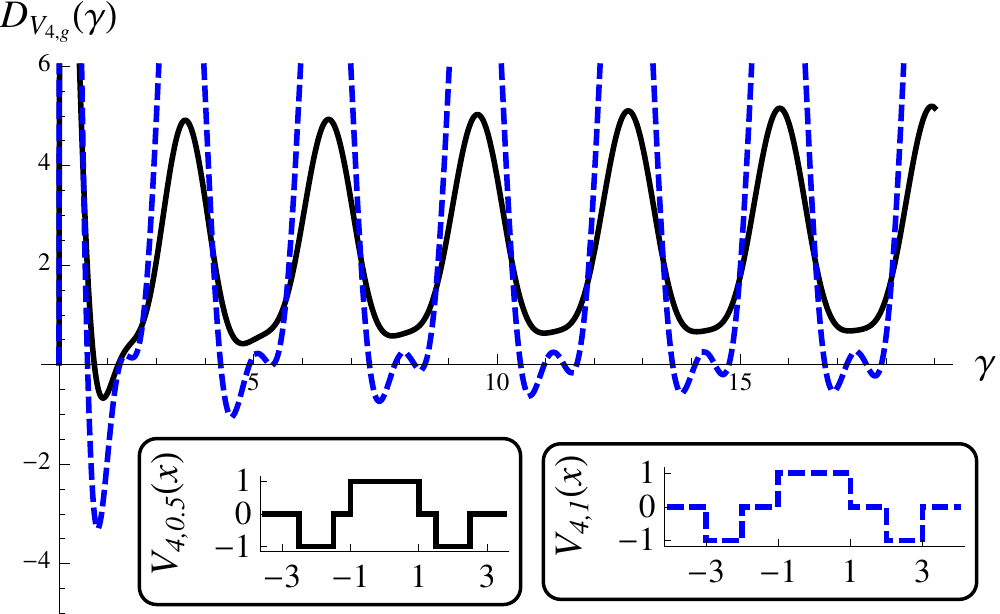}
\caption{\small The graphs of $D_{V_{4,g}}(\gamma)$ against real $\gamma$ for the potentials
$V_{4,g}(x):=W(x; [-g-2, -g-1,-1,1,g+1,g+2]; \{-1,0,1,0,-1\})$ with $g=0.5$ (solid black line) and $g=1$ (dashed blue line).\label{fig:examp4_plots}}
\end{center}
\end{figure}

To explain this phenomenon, we once more consider equation \eqref{eq:DVis0}, now with $V=V_{4,g}$. Although the explicit expression for the determinant 
$D_{V_{4,g}}$ is rather cumbersome, some simplifications lead to the asymptotic form
\begin{equation}
\label{eq:2gg}
\(1 + \sqrt{2}\tanh(g) \cos(2\tg-\pi/4)\)\(1 + \sqrt{2}\tanh(g) \cos(2\tg+\pi/4)\)+O\(\tg^{-1}\)=0
\end{equation} 
as $\tg\to\infty$ (where the derivative of the $O$-term is also $O\left(\tg^{-1}\right)$). 
The asymptotics of the number of zeros of \eqref{eq:2gg} reduces to consideration of a pair of elementary equations for cosine. 
It is then immediate that the asymptotics of 
$\#(\gspec{V_{4,g}}\cap[0,R])$ as $R\to\infty$ changes abruptly between $O(1)$ and $\displaystyle\frac{R}{\pi}\,\norm{V_{4,g}}_{L^1}+o(R)$, depending on whether $g<g_0:=\operatorname{arctanh}(1/\sqrt{2})\approx 0.8814$ or $g>g_0$, respectively.
This shows, as already announced in Remark \ref{rem:twogapzerointegral}, that unlike no-gap and one-gap potentials, two-gap potentials with zero integral may produce an infinite number of real eigenvalues.
\end{example}

\begin{example}[Potential from \cite{HRP}]
\label{ex:exPortnoi} 
Using a complicated explicit solution involving special functions, Hartmann, Robinson and Portnoi
found that, for the potential $\displaystyle V_\text{HRP}(x)=-1/{\cosh(x)}$ and any $k>0$, the positive part of the  spectrum $\gspec{V_\text{HRP}}$
coincides with the set $k-\frac{1}{2}+\N$. We treat this potential using the Pr\"ufer method and plot, for real $\gamma$, the quantity $\cos(\Delta_{V_\text{HRP}}(\gamma))$; by Proposition \ref{prop:gspeccharD}, $\gamma\in\gspec{{V_\text{HRP}}}\cap\R$ if and only if $\cos(\Delta_{V_\text{HRP}}(\gamma))=0$ (see Section \ref{subsec:boundargu} for the definition of $\Delta_{V_\text{HRP}}$ and further details). The curves in Figure \ref{fig:ex5} (drawn for $k=1$ and $k=1.5$) illustrate the result of \cite{HRP}.

\begin{figure}[!hbt]
\begin{center}
\includegraphics{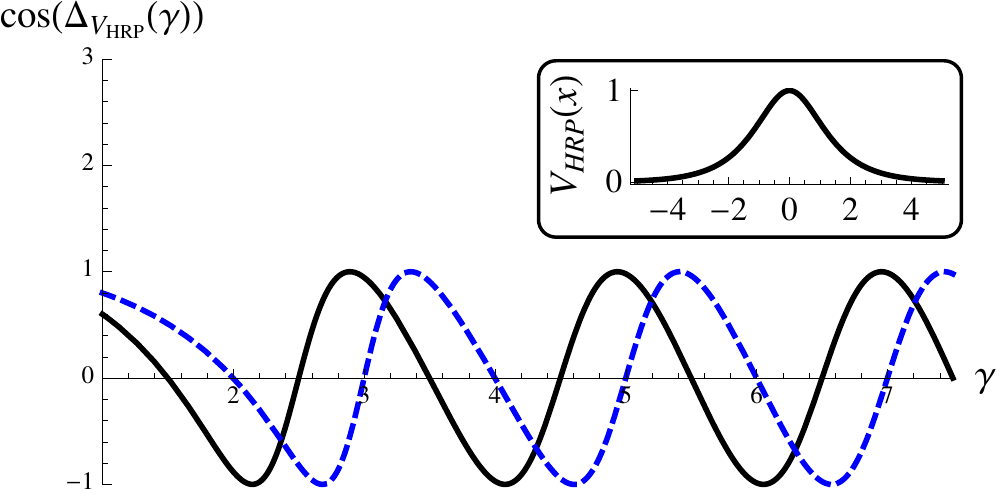}
\caption{\small The graphs of $\cos(\Delta_{V_\text{HRP}}(\gamma))$ for  $k=1$ (solid black line) and $k=1.5$ (dashed blue line).\label{fig:ex5}}
\end{center}
\end{figure}
\end{example}

\section{Main arguments}
\label{sec:argu}

In this section we give the arguments for the main theorems based on a series of more technical results; the proofs for the latter will be deferred to Section \ref{sec:tech}.

\subsection{General}
\label{subsec:argugen}

The unperturbed operator $\Drc{0}$ is an unbounded self-adjoint operator on $L^2$ whose domain is $\sob$, the Sobolev space of ($\C^2$ valued) functions on $\R$; that is
\[
\dom(\Drc{0})=\sob=\{\psib\in L^2\,:\,\nabla\psib\in L^2\}.
\]
In fact it is straightforward to check that $\norm{\Drc{0}\psib}^2=\norm{\nabla\psib}^2+k^2\norm{\psib}^2$, so $\norm{\Drc{0}\psib}$ is equivalent to $\norm{\psib}_{H^1}$. It follows that $\Drc{0}$ defines an isomorphism $\sob\to L^2$. 

Next we consider multiplication by an element of $\Vclass_0$. Firstly note that 
a norm can be defined on $\Vclass_0$ using the expression
\[
\norm{V}_{\Vclass_0}=\sup_{x\in\R}\norm{V}_{L^2(x-1,x+1)}.
\]
This norm makes $\Vclass_0$ a Banach space in which $C^\infty_0$ is a dense subset.

\begin{lemma}
\label{lem:Vcptop}
Multiplication by a fixed $V\in\Vclass_0$ defines a compact map $\sob\to L^2$.
\end{lemma}

\begin{proof}
Initially suppose $V\in C^\infty_0$. Choose a bounded interval $I$ with $\supp(V)\subseteq I$. We can view multiplication by $V$ as a composition $\psib\mapsto\psib\vert_{I}\mapsto (V\psib)\vert_{I}\mapsto V\psib$ where we firstly restrict to $I$, then multiply by $V$ and finally extend by $0$. This gives a map $\sob\to L^2(I)\to L^2(I)\to L^2$, where the last two steps are continuous and the first step is compact (by the Rellich-Kondrachov Theorem; see \cite{Adams}, for example).

Since $C^\infty_0$ is dense in $\Vclass_0$ and the set of compact maps is closed, it now suffices to show that multiplication defines a continuous bilinear map $\Vclass_0\times\sob\to L^2$. To this end firstly note that the Sobolev Embedding Theorem (\emph{ibid.}) gives
\[
\norm{\psib}_{L^\infty(x-1,x+1)}\le C\norm{\psib}_{\sob(x-1,x+1)}
\]
for some constant $C$ (which is independent of $x$). Thus
\[
\norm{V\psib}_{L^2(x-1,x+1)}
\le\norm{V}_{L^2(x-1,x+1)}\,\norm{\psib}_{L^\infty(x-1,x+1)}
\le C\,\norm{V}_{L^2(x-1,x+1)}\,\norm{\psib}_{\sob(x-1,x+1)}.
\]
On the other hand 
\[
\norm{\psib}_{L^2}^2=\frac12\int_{\R}\norm{\psib}_{L^2(x-1,x+1)}^2\dr x,
\]
with a similar expression holding for $\norm{\cdot}_{\sob}$. Combining the above then gives
\begin{align*}
\norm{V\psib}_{L^2}^2
&\le\frac12\,C^2\int_\R\norm{V}_{L^2(x-1,x+1)}^2\,\norm{\psib}_{\sob(x-1,x+1)}^2\dr x\\
&\le\frac12\,C^2\sup_{x\in\R}\norm{V}_{L^2(x-1,x+1)}^2\int_\R\norm{\psib}_{\sob(x-1,x+1)}^2\dr x
=C^2\,\norm{V}_{\Vclass_0}^2\,\norm{\psib}_{\sob}^2,
\end{align*}
for any $V\in\Vclass_0$ and $\psib\in\sob$.
\end{proof}

Since $\sob=\dom(\Drc{0})$, Lemma \ref{lem:Vcptop} is equivalent to the statement that (multiplication by) $V\in\Vclass_0$ is a relatively compact perturbation of $\Drc{0}$. It follows that the sum $\Drc{0}+V=\Drc{V}$ defines a self-adjoint operator with $\dom(\Drc{V})=\dom(\Drc{0})=\sob$ (see \cite{RSI}). 

Although we're interested in real valued potentials it is helpful to consider some basic results for more general complex valued potentials as well. Let $\CVclass_0=\Vclass_0\otimes_\R\C$ denote the complex valued version of $\Vclass_0$ (in other words, $\CVclass_0$ consists of functions of the form $U+iW$ where $U,W\in\Vclass_0$). Note that $\gamma V\in\CVclass_0$ for any $V\in\Vclass_0$ and $\gamma\in\C$. Now Lemma \ref{lem:Vcptop} clearly extends to $\CVclass_0$, so (multiplication by) some $\mathcal{V}\in\CVclass_0$ is still a relatively compact perturbation of $\Drc{0}$. Thus the sum $\Drc{0}+\mathcal{V}=\Drc{\mathcal{V}}$ defines a closed operator with $\dom(\Drc{\mathcal{V}})=\dom(\Drc{0})=\sob$. Although the operator $\Drc{\mathcal{V}}$ will not be self-adjoint (unless $\mathcal{V}\in\Vclass_0$), the essential spectrum is still given by \eqref{eq:TVessspec}, while any spectrum of $\Drc{\mathcal{V}}$ in $\C\setminus\Tess$ consists of isolated eigenvalues of finite (algebraic) multiplicity.

\begin{proof}[Proof of Theorem \ref{thm:disgspec}]
We have $\gamma\in\gspec{V}$ iff $(\Drc{0}+\gamma V)\psib=0$ for some non-trivial $\psib\in\sob$. In turn this is equivalent to $V\Drc{0}^{-1}\phib=\mu\phib$ where $\phib=\Drc{0}\psib\in L^2$ and $\mu=-1/\gamma$; in other words, $-1/\gamma$ should be an eigenvalue of the operator $V\Drc{0}^{-1}$. However  $V\Drc{0}^{-1}$ is a compact operator by Lemma \ref{lem:Vcptop}, so must have discrete spectrum away from $0$. The result follows.
\end{proof}

It is helpful to have a more symmetric version of the idea just used in the proof of Theorem \ref{thm:disgspec}. Fix $V\in\Vclass_0$ and let $J_V$ be the operator on $L^2$ given as multiplication by $\sgn(V)$ (where, for definiteness, we can set $\sgn(V)(x)=+1$ whenever $V(x)=0$). Then $J_V^*=J_V^{-1}=J_V$ while
\[
V=\sqrt{\abs{V}}J_V\sqrt{\abs{V}}.
\]
Define an operator $A_V$ by 
\[
A_V=\sqrt{\abs{V}}\Drc{0}^{-1}\sqrt{\abs{V}}.
\]
Using similar ideas to those in the proofs of Theorem \ref{thm:disgspec} and Lemma \ref{lem:Vcptop} we get the following:

\begin{lemma}
\label{lem:symmspecequiv}
The operator $A_V$ is a compact self-adjoint operator on $L^2$. Furthermore, we have $\gamma\in\gspec{V}$ iff $-1/\gamma\in\spec(J_VA_V)$.
\end{lemma}

If $V$ is single-signed we can choose $J_V=\pm I$ ($+I$ if $V\ge0$ or $-I$ if $V\le0$). Then $\gamma\in\gspec{V}$ iff $-1/\gamma$ is in the spectrum of the compact self-adjoint operator $\pm A_V$. This gives a justification of Theorem \ref{thm:singsignreal}, although a more elementary argument is also possible (see Section \ref{subsec:symmargu}). 
\begin{remark}
\label{BS}
Lemma \ref{lem:symmspecequiv} can be viewed as a Birman-Schwinger principle for the Dirac operator $\Drc{V}$. This wide ranging principle has been used to obtain a number of results related to those presented here, both in the single-sign case where the associated Birman-Schwinger operator $A_V$ is self-adjoint (see \cite{K,BL}), and in the variable sign case where we need to consider the non-self-adjoint operator $J_VA_V$ (see \cite{Sa}). Approaches based on the Birman-Schwinger principle rely on obtaining spectral information about the operator $A_V$ or $J_VA_V$. Potential sources of information include eigenvalue or singular value estimates (such as \cite{Cw}), or pseudo-differential techniques leading to eigenvalue asymptotics (such as \cite{BS}). In Section \ref{subsec:boundargu} we take a different approach, based on Pr\"ufer techniques; this is convenient in our one-dimensional setting, and allows for slightly less restrictive assumptions on the potential (see Remark \ref{rem:Cwandp=1} in particular).
\end{remark}

\medskip

For any $V\in\Vclass_0$ note that $\abs{V}$ is a single-signed potential which is also in $\Vclass_0$. Using the association with compact operators given by Lemma \ref{lem:symmspecequiv} we are able to link points in $\gspec{V}$ and $\gspec{\abs{V}}$ with the eigenvalues and singular values of a single operator, and thus estimate the former using the latter (via Weyl's Inequality). Let $0<\mu_1<\mu_2<\dots$ be the positive points in $\gspec{\abs{V}}$, ordered by size. Also let $\gamma_1,\gamma_2,\dots$ denote the points in $\gspec{V}\cap\{z\in\C:\text{$z\in\R^+$ or $\Im z>0$}\}$ ordered so as to have non-decreasing modulus (and counted according to algebraic multiplicity). 

\begin{lemma}
\label{lem:Weylineq}
For any $N\in\N$ we have
\[
\prod_{n=1}^N\abs{\gamma_n}\ge\prod_{n=1}^N\mu_n.
\]
\end{lemma}

\begin{proof}
By symmetry we know that the points in $\gspec{V}$ are just $\pm\gamma_n$, $n=1,2,\dots$, while the points in $\gspec{\abs{V}}$ are just $\pm\mu_n$, $n=1,2,\dots$. Lemma \ref{lem:symmspecequiv} now implies $\pm1/\gamma_n$, $n=1,2,\dots$ are the non-zero eigenvalues of $J_VA_V$, while $\pm1/\mu_n$, $n=1,2,\dots$ are the non-zero eigenvalues of $J_{\abs{V}}A_{\abs{V}}=A_V$ (note that, we can take $J_{\abs{V}}=I$). However $(J_VA_V)^*J_VA_V=A_VJ_V^2A_V=A_V^2$ so the singular values of $J_VA_V$ are just the eigenvalues of $\abs{A_V}$; that is, $1/\mu_n$, $n=1,2,\dots$, where each eigenvalue has multiplicity $2$. Given $N\in\N$ we can now use Weyl's Inequality (see \cite{Weyl}) to compare the largest $2N$ eigenvalues and singular values of the compact operator $J_VA_V$; this gives
\[
\prod_{n=1}^N\Bigl(\frac1{\abs{\gamma_n}}\Bigr)^2\le\prod_{n=1}^N\Bigl(\frac1{\mu_n}\Bigr)^2.
\]
The result follows.
\end{proof}

\subsection{Symmetries}

\label{subsec:symmargu}

Our unperturbed operator is
\[
\Drc{0}=-i\sigma_2\nabla+k\sigma_3
=\begin{pmatrix}k&-\nabla\\\nabla&-k\end{pmatrix}.
\]
Since $V$ is real valued we immediately get
\begin{equation}
\label{eq:cxconjsymm}
\overline{(\Drc{0}+\gamma V)\psib}=(\Drc{0}+\overline{\gamma}V)\overline{\psib}.
\end{equation}
It follows that $\gamma\in\gspec{V}$ iff $\overline{\gamma}\in\gspec{V}$. On the other hand, the commutator properties of the Pauli matrices (namely $\sigma_j\sigma_k=-\sigma_k\sigma_j$ if $j\neq k$) give us
\[
\sigma_1(\Drc{0}+\gamma V)=-(\Drc{0}-\gamma V)\sigma_1
\quad\text{and}\quad
\sigma_2(\Drc{0}+\gamma V)=(-i\sigma_2\nabla-\sigma_3k+\gamma V)\sigma_2.
\]
From the first identity we get $\gamma\in\gspec{V}$ iff $-\gamma\in\gspec{V}$, while the second shows that $\gspec{V}$ is invariant if we replace $k$ with $-k$ in the definition of $\Drc{0}$. 

\begin{remark}
The symmetry corresponding to $\sigma_3$ can be used to help study even potentials (compare with our consideration of anti-symmetric potentials below).
\end{remark}

\begin{proof}[Proof of Theorem \ref{thm:singsignreal}]
Suppose $\Drc{\gamma V}\psib=(\Drc{0}+\gamma V)\psib=0$ for some $\gamma\in\C$ and $\psib\in\sob$. Then $\ipd{\Drc{0}\psib}{\psib}=-\gamma\ipd{V\psib}{\psib}$ while $\ipd{\Drc{0}\psib}{\psib},\ipd{V\psib}{\psib}\in\R$. If $\gamma\notin\R$ we must therefore have $\ipd{V\psib}{\psib}=0$. Since $V$ is single-signed it follows that $V\psib=0$, leading to $\Drc{0}\psib=-\gamma V\psib=0$ and thus $\psib=0$ (recall that $\Drc{0}:\sob\to L^2$ is an isomorphism).
\end{proof}

\begin{proof}[Proof of Theorem \ref{thm:oddV0}]
We consider two symmetries of the operator $\Drc{\gamma V}$; define an anti-linear operator $C$ and a unitary operator $S$ on $L^2$ by
\[
C\psib=\overline{\psib}
\quad\text{and}\quad
(S\psib)(x)=\sigma_2\psib(-x),\ x\in\R.
\]
These operators map $\sob$ (isometrically) onto $\sob$, and satisfy $C^2=I=S^2$ and $CS=-SC$. Furthermore, \eqref{eq:cxconjsymm} can be rewritten as $C\Drc{\gamma V}=\Drc{\overline{\gamma} V}C$, while $\nabla S=-S\nabla$ and $SV=-VS$ (as $V$ is anti-symmetric) which leads to $S\Drc{\gamma V}=-\Drc{\gamma V}S$. 

Now suppose $\gamma\in\gspec{V}\cap\R$ and choose $0\neq\psib\in\sob$ which satisfies $\Drc{\gamma V}\psib=0$. Then
\[
\Drc{\gamma V}C\psib=C\Drc{\gamma V}\psib=0
\quad\text{and}\quad
\Drc{\gamma V}S\psib=-S\Drc{\gamma V}\psib=0.
\]
However $0$ is a simple eigenvalue of the operator $\Drc{\gamma V}$ by Lemma \ref{lem:simeval}, so we must have $C\psib=\alpha\psib$ and $S\psib=\beta\psib$ for some $\alpha,\beta\in\C$. Then
\[
\abs{\alpha}^2\psib=C^2\psib=\psib,
\qquad
\beta^2\psib=S^2\psib=\psib
\qquad\text{and}\qquad
\alpha\overline{\beta}\psib=CS\psib=-SC\psib=-\alpha\beta\psib,
\]
so $\abs{\alpha}^2=1=\beta^2$ and $\alpha\overline{\beta}=-\alpha\beta$ (recall that $\psib\neq0$). These equations clearly have no solution, so we must have $\gspec{V}\cap\R=\emptyset$.
\end{proof}

\subsection{General bounds and asymptotics}

\label{subsec:boundargu}

Suppose $V\in L^1_\loc$. We can view our basic equation $\Drc{\gamma V}\psib=0$ as the $2\times 2$ system of first order ordinary differential equations on $\R$ given by
\begin{equation}
\label{eq:l0ode}
\begin{aligned}
\nabla\psi_1&=(k-\gamma V)\psi_2,\\
\nabla\psi_2&=(k+\gamma V)\psi_1.
\end{aligned}
\end{equation}
The basic theory for such equations is well established 
(see \cite{Hartman}; here, and in subsequent references to \cite{Hartman}, some straightforward modifications of the results are needed in order to cover $L^1_\loc$ coefficients).
In particular, if $x_0\in\R$ and $\alpha_1,\alpha_2\in\C$, then there exists a unique absolutely continuous solution to \eqref{eq:l0ode} on $\R$ with $\psi_j(x_0)=\alpha_j$ for $j=1,2$. Furthermore, for given $x\in\R$, this solution depends continuously on $\alpha_1,\alpha_2$, $\lambda$ and $V$ (the latter as a function in $L^1(I)$ where $I$ is the interval between $x_0$ and $x$). A consequence of the uniqueness of solutions is that for any non-trivial solution $\psib$ of \eqref{eq:l0ode} we have $\psib(x)\neq0$ for all $x\in\R$.

\medskip

Now suppose $\gamma\in\R$. It follows that all coefficients in \eqref{eq:l0ode} are real, so we may restrict our attention to solutions which are also real valued. Since any non-trivial solution $\psib$ is absolutely continuous and satisfies $\psib(x)\neq0$ for all $x\in\R$ we can define an absolutely continuous function $S:\R\to\cir$ by
\[
S(x)=\frac{\psi_1(x)+\ir\psi_2(x)}{(\psi_1^2(x)+\psi_2^2(x))^{1/2}}
\]
(here $\cir$ denotes the unit circle in $\C$). By lifting to $\R$ (the universal cover of $\cir$) we can define a further absolutely continuous function $\theta:\R\to\R$ so that 
\begin{equation}
\label{eq:thetafrompsi}
\er^{\ir\theta}=S=\frac{\psi_1+\ir\psi_2}{\abs{\psib}};
\end{equation}
this function is the Pr\"{u}fer argument associated to $\psib$ and is unique up to the addition of a constant in $2\pi\Z$. We note that Pr\"ufer coordinates
are a standard tool  for problems of this kind (see, for example, \cite{Schm}).

A straightforward calculation gives
\[
\ir\nabla\theta\,\er^{\ir\theta}
=\nabla S
=\ir\,\frac{\psi_1\nabla\psi_2-\psi_2\nabla\psi_1}{\psi_1^2+\psi_2^2}\,\er^{\ir\theta}.
\]
Using \eqref{eq:l0ode} it follows that
\[
\nabla\theta
=\frac{\psi_1\nabla\psi_2-\psi_2\nabla\psi_1}{\psi_1^2+\psi_2^2}
=\gamma V+k\,\frac{\psi_1^2-\psi_2^2}{\psi_1^2+\psi_2^2}.
\]
However
\[
\frac{\psi_1^2-\psi_2^2}{\psi_1^2+\psi_2^2}
=\Re\er^{2\ir\theta}
=\cos(2\theta)
\]
so $\theta$ satisfies the first order non-linear equation
\begin{equation}
\label{eq:PruferODE}
\nabla\theta=\gamma V+k\cos(2\theta).
\end{equation}

Conversely, when $V\in L^1_\loc$ the differential equation \eqref{eq:PruferODE} has a unique absolutely continuous solution valid on $\R$ for a given value of $\theta(x_0)$, $x_0\in\R$ (see \cite{Hartman}).
If we have one solution $\theta$ of \eqref{eq:PruferODE} then $\theta+n\pi$ provides a further solution for any $n\in\Z$ (note that, for $\theta$ given by \eqref{eq:thetafrompsi} the solution $\theta+\pi$ corresponds to taking $-\psib$ as a solution of \eqref{eq:l0ode}). 

\medskip

Now suppose $V\in\Vclass_1$ and $\psib$ is a non-trivial solution of \eqref{eq:l0ode}. Let $\theta$ be given by \eqref{eq:thetafrompsi}. The fact that $\int_\R\abs{V(x)}\dr x<+\infty$ means that, to leading order as $x\to\pm\infty$, $\psib$ behaves like a solution to \eqref{eq:l0ode} with $V\equiv0$. The corresponding asymptotic behaviour of $\theta$ can be summarised as follows (see Section \ref{subsec:asym} for more details):

\begin{lemma}
\label{lem:+thetalim}
The quantity $\theta(x)$ has well defined limits as $x\to\pm\infty$ which satisfy
\begin{equation}
\label{eq:+thetalim}
\theta(\pm\infty):=\lim_{x\to\pm\infty}\theta(x)\in\frac{\pi}4+\frac{\pi}2\,\Z.
\end{equation}
Furthermore, upon restriction we have $\psib\in L^2(\R^\pm)$ iff
\[
\theta(\pm\infty)\in\mp\frac{\pi}4+\pi\Z.
\]
\end{lemma}

For any $\gamma\in\R$ we can uniquely specify two solutions $\theta_{\gamma,+}$ and $\theta_{\gamma,-}$ to \eqref{eq:PruferODE} by imposing the boundary conditions
\begin{equation}
\label{eq:bcthetapm}
\theta_{\gamma,\pm}(\pm\infty)=\mp\frac{\pi}4.
\end{equation}
Lemma \ref{lem:+thetalim} shows that $\theta_{\gamma,\pm}$ correspond to solutions of \eqref{eq:l0ode} which are in $L^2$ on $\R^\pm$. We get an $L^2$ solution on the whole of $\R$ precisely when these solutions `match' at one (or, equivalently, any) point of $\R$. Choosing $0$ as the point at which we check, the matching condition is simply that $\theta_{\gamma,+}(0)$ and $\theta_{\gamma,-}(0)$ must differ by a multiple of $\pi$. Define a function $\Delta_V:\R\to\R$ by setting
\begin{align}
\Delta_V(\gamma)
&=\bigl(\theta_{\gamma,+}(+\infty)-\theta_{\gamma,+}(0)\bigr)+\bigl(\theta_{\gamma,-}(0)-\theta_{\gamma,-}(-\infty)\bigr)
\nonumber\\
\label{eq:Deltadefn}
&=-\frac\pi2-\theta_{\gamma,+}(0)+\theta_{\gamma,-}(0).
\end{align}
We can thus characterise points in $\gspec{V}\cap\R$ as follows:

\begin{proposition}
\label{prop:gspeccharD}
We have $\gamma\in\gspec{V}\cap\R$ iff $\Delta_V(\gamma)\in\dfrac{\pi}2+\pi\Z$.
\end{proposition}

We have $\Delta_V(0)=0$ (note that $\theta_{0,\pm}$ are constant functions) while $\Delta_V(\gamma)$ depends continuously on $\gamma$ (this essentially follows from  standard results for the continuous dependence on parameters of solutions of ordinary differential equations; see \cite{Hartman}). For large $\gamma$ we have the following asymptotic behaviour (the proof is given in Section \ref{subsec:DeltaV}):

\begin{proposition}
\label{prop:Dgasym}
If $V\in\Vclass_1$ then 
\[
\Delta_V(\gamma)=\gamma\int_{\R}V(x)\dr x+o(\gamma)
\]
as $\abs{\gamma}\to\infty$.
\end{proposition}

Together Propositions \ref{prop:gspeccharD} and \ref{prop:Dgasym} allow us to establish Theorem \ref{thm:lowerasym}.

\begin{proof}[Proof of Theorem \ref{thm:lowerasym}]
Let $I_R\subset\R$ denote the closed interval with endpoints $\Delta_V(0)=0$ and $\Delta_V(R)$. By the Intermediate Value Theorem $\Delta_V(\gamma)$ takes each value in $I_R$ at least once for some $\gamma\in[0,R]$. From Proposition \ref{prop:gspeccharD} we then get
\begin{align}
\#(\gspec{V}\cap[0,R])
&=\#\bigl\{\gamma\in[0,R]:\Delta_V(\gamma)\in(\Z+1/2)\pi\bigr\}
\nonumber\\
\label{eq:estcross}
&\ge\#(I_R\cap(\Z+1/2)\pi\bigr)
\ge\frac{\abs{I_R}}{\pi}-\frac12,
\end{align}
where $\abs{I_R}$ is the length of $I_R$. On the other hand, Proposition \ref{prop:Dgasym} gives
\begin{equation}
\label{estof|IR|:eq}
\abs{I_R}=\abs{\Delta_V(R)}
=R\left\lvert\int_{\R}V(x)\dr x\right\rvert+o(R)
\end{equation}
as $R\to\infty$.
\end{proof}
\begin{remark}
\label{Schmidt}
For compactly supported potentials, Theorem  \ref{thm:lowerasym} admits a simpler proof,  similarly to \cite[Theorem 3]{Schm}. 
The details of this argument have been written down by Michael~Morin (supported by the NSERC Undegraduate Summer Research Award, 2011).
\end{remark}

If $V\in\Vclass_1$ is single-signed then $\Delta_V$ is monotonic (the proof is given in Section \ref{subsec:deriv}):

\begin{proposition}
\label{prop:positiveDinc}
Suppose $V\in\Vclass_1$ is single-signed and non-trivial. Then $\Delta_V$ is strictly increasing if $V\ge0$ and strictly decreasing if $V\le0$.
\end{proposition}

Using this monotonicity we can sharpen the argument given for Theorem \ref{thm:lowerasym} to establish Theorem \ref{thm:singsigngvalasym}.

\begin{proof}[Proof of Theorem \ref{thm:singsigngvalasym}]
If $V\in\Vclass_1$ is single-signed it follows from Proposition \ref{prop:positiveDinc} that there is at most one $\gamma$ with $\Delta_V(\gamma)=(n+1/2)\pi$ for any given $n\in\Z$. Thus  \eqref{eq:estcross} in the proof of Theorem \ref{thm:lowerasym} can be modified to give
\begin{align*}
\#(\gspec{V}\cap[0,R])
&=\#\bigl\{\gamma\in[0,R]:\Delta_V(\gamma)\in(\Z+1/2)\pi\bigr\}\\
&=\#(I_R\cap(\Z+1/2)\pi\bigr)
\le\frac{\abs{I_R}}{\pi}+\frac12.
\end{align*}
The required upper asymptotic bound on $\#(\gspec{V}\cap[0,R])$ can then be obtained using \eqref{estof|IR|:eq}. 
\end{proof}

\medskip

To give uniform estimates for $\Delta_V(\gamma)$ we firstly introduce an auxiliary function. For any $a\ge0$ let $\lfloor a\rfloor$ denote the largest integer not exceeding $a$; that is,
\[
\lfloor a\rfloor=\max\{n\in\Z:n\le a\}.
\]
Now define a function $h:[0,+\infty)\to[0,+\infty)$ by
\[
h(a)=a+\frac{\pi}2\Bigl\lfloor\frac{2}{\pi}a\Bigr\rfloor.
\]
In particular, $h$ is strictly increasing, $h(a)=a$ for $a<\pi/2$, $a\le h(a)\le 2a$ for all $a\ge0$, and $h(a)+h(b)\le h(a+b)$ for all $a,b\ge0$ (note that $\lfloor a\rfloor+\lfloor b\rfloor\le\lfloor a+b\rfloor$).

As a general bound on $\Delta_V(\gamma)$ we have the following (the proof is given in Section \ref{subsec:DeltaV}):

\begin{proposition}
\label{prop:genDeltaVest}
If $V\in\Vclass_1$ and $\gamma\in\R$ then $\abs{\Delta_V(\gamma)}\le h\bigl(\abs{\gamma}\,\norm{V}_{L^1}\bigr)$.
\end{proposition}

This result leads to uniform lower bounds for points in $\gspec{V}$, and in turn helps to justify Theorem \ref{thm:genVL1upbnd}. We firstly deal with the case when $V$ is single-signed (Proposition \ref{prop:ssVevalupbnd}), and then consider arbitrary potentials (Proposition \ref{prop:genVevalupbnd}) by using Lemma \ref{lem:Weylineq} to reduce this to the single-sign case.

\begin{proposition}
\label{prop:ssVevalupbnd}
Suppose $V\in\Vclass_1$ is single-signed and non-trivial. Let $0<\gamma_1<\gamma_2<\dots$ denote the sequence of positive points in $\gspec{V}$, arranged in order of increasing size. Then
\[
\gamma_n\ge\frac{\pi}{2\norm{V}_{L^1}}\,n,\quad n=1,2,\dots.
\]
\end{proposition}

Note that the constant in the bound is only half the asymptotic value (see Theorem \ref{thm:singsigngvalasym}).

\begin{proof}
Suppose $V\ge0$ (the case $V\le0$ can be treated similarly). By Propositions \ref{prop:gspeccharD} and \ref{prop:positiveDinc} we have $\Delta_V(\gamma_n)=(n-1/2)\pi$ for $n=1,2,\dots$. Now if $a<n\pi/2$ then $\lfloor 2a/\pi\rfloor<n$ so 
\[
h(a)\le a+(n-1)\frac{\pi}2\le\bigl(n-\frac12\bigr)\pi
=\Delta_V(\gamma_n)
\le h\bigl(\gamma_n\,\norm{V}_{L^1}\bigr)
\]
by Proposition \ref{prop:genDeltaVest}. However $h$ is strictly increasing so we must have $a\le\gamma_n\,\norm{V}_{L^1}$. Taking $a\to(n\pi/2)^-$ now gives the result.
\end{proof}

\begin{proposition}
\label{prop:genVevalupbnd}
Suppose $V\in\Vclass_1$ is non-trivial and let $\gamma_1,\gamma_2,\dots$ denote the points in $\gspec{V}\cap\bigl\{z\in\C:\text{$z\in\R^+$ or $\Im z>0$}\bigr\}$ ordered so as to have non-decreasing modulus (and counted according to algebraic multiplicity). Then
\[
\abs{\gamma_n}\ge\frac{\pi}{2\er\norm{V}_{L^1}}\,n,\quad n=1,2,\dots.
\]
\end{proposition}

\begin{proof}
Let $0<\mu_1<\mu_2<\dots$ denote the positive points in $\gspec{\abs{V}}$, arranged in order of increasing size. Noting that $\abs{V}\in\Vclass_1$ is single-signed and non-trivial, Proposition \ref{prop:ssVevalupbnd} gives $\mu_n\ge\nu n$ for $n=1,2,\dots$, where $\nu=\pi/(2\norm{V}_{L^1})$. Using Lemma \ref{lem:Weylineq} and the ordering on the $\gamma_n$'s we now get
\[
\abs{\gamma_N}^N\ge\prod_{n=1}^N\abs{\gamma_n}\ge\prod_{n=1}^N\mu_n\ge\nu^N N!\ge\nu^N\Bigl(\frac{N}{\er}\Bigr)^N
\]
for any $N\in\N$. 
\end{proof}

\begin{proof}[Proof of Theorem \ref{thm:genVL1upbnd}]
If $V=0$ then $\gspec{V}=\emptyset$ and there is nothing to prove. Now suppose $V$ is non-trivial. Let $R\ge0$, set $N=\#\bigl(\gspec{V}\cap\{z\in\C:\abs{z}\le R\}\bigr)$ and suppose $N>0$. Using the symmetry $\gspec{V}=-\gspec{V}$ we know that $N=2M$ for some $M\in\N$. With $\gamma_1,\gamma_2,\dots$ defined as in  Proposition \ref{prop:genVevalupbnd} it follows that $\abs{\gamma_M}\le R$. We then obtain
\[
N=2M\le 2\frac{2\er}{\pi}\,\norm{V}_{L^1}\abs{\gamma_M}\le\frac{4\er}{\pi}\,\norm{V}_{L^1}R
\]
using this result.
\end{proof}

\begin{remark}
\label{rem:otherests}
Various other estimates can obtained from straightforward modifications to the proof of Theorem \ref{thm:genVL1upbnd} presented above. For example, if $V\in\Vclass_1$ is single-signed we can use Proposition \ref{prop:ssVevalupbnd} in place of Proposition \ref{prop:genVevalupbnd} to obtain the uniform upper bound
\[
\#(\gspec{V}\cap[0,R])\le\frac{2}{\pi}\,\norm{V}_{L^1}R
\]
for any $R\ge0$. Alternatively, for any $V\in\Vclass_1$ we can estimate $\mu_n$ in the proof of Proposition \ref{prop:genVevalupbnd} using Theorem \ref{thm:singsigngvalasym} instead of Proposition \ref{prop:ssVevalupbnd}; this leads to the asymptotic bound
\[
\#\bigl(\gspec{V}\cap\{z\in\C:\abs{z}\le R\}\bigr)\le\frac{2\er}{\pi}\,\norm{V}_{L^1}R+o(R)
\]
as $R\to\infty$.
\end{remark}

As a direct application of the main result in \cite{EltonTa} we can extend the general upper bound given by Theorem \ref{thm:genVL1upbnd} to potentials in $\Vclass_0\cap L^p$ for $1<p<\infty$.

\begin{theorem}
\label{thm:cfnupperbnd}
Suppose $V\in\Vclass_0\cap L^p$ for some $1<p<\infty$. Then
\[
\#\bigl(\gspec{V}\cap\{z\in\C:\abs{z}\le R\}\bigr)\le C_p\,\norm{V}_{L^p}^pR^p
\]
for any $R\ge 0$, where $C_p$ is a constant depending only on $p$.
\end{theorem}

\begin{remark}
\label{rem:Cwandp=1}
The main result in \cite{EltonTa} is based on singular value estimates from \cite{Cw}. We can't extend the argument to cover $p=1$ as it corresponds to an excluded boundary case in \cite{Cw}.
\end{remark}

\subsection{Derivatives}

\label{subsec:deriv}

To study the derivatives of $\Delta_V(\gamma)$ we need to obtain more information about the $\gamma$ dependence of solutions to \eqref{eq:PruferODE}. Firstly consider any closed bounded interval $I=[a,b]\subset\R$ and potential $V\in L^1_\loc$ on $I$. For each $\gamma\in\R$ suppose we have a solution $\theta_\gamma$ of \eqref{eq:PruferODE} where $\theta_\gamma(b)$ depends twice differentiably on $\gamma$.
Standard results for ordinary differential equations (see \cite{Hartman}) then imply $\theta_\gamma(x)$ is twice differentiable in $\gamma$ for each $x\in I$ (in fact $\theta_\gamma$ will depend analytically on $\gamma$ provided $\theta_\gamma(b)$ does), so we can set
\[
\omega_\gamma(x)=\frac{\dr}{\dr\gamma}\,\theta_\gamma(x)
\quad\text{and}\quad
\rho_\gamma(x)=\frac{\dr^2}{\dr\gamma^2}\,\theta_\gamma(x).
\]
From \eqref{eq:PruferODE} we immediately get
\[
\nabla\omega_\gamma=V-2k\sin(2\theta_\gamma)\,\omega_\gamma
\quad\text{and}\quad
\nabla\rho_\gamma(x)=-4k\cos(2\theta_\gamma)\,\omega_\gamma^2-2k\sin(2\theta_\gamma)\,\rho_\gamma.
\]
Then $\nabla(\er^G\omega_\gamma)=\er^GV$, where $G$ is any function satisfying $\nabla G(x)=2k\sin(2\theta_\gamma(x))$. 
Thus
\begin{align}
\omega_\gamma(x)&=\er^{G(b)-G(x)}\omega_\gamma(b)-\int_x^b\er^{G(t)-G(x)}V(t)\dr t\nonumber\\
\label{diffwint:eq}
&=\er^{2k\Psi_{[x,b]}}\omega_\gamma(b)-\int_x^b\er^{2k\Psi_{[x,t]}}V(t)\dr t,
\end{align}
where, for any interval $J\subseteq I$, 
\begin{equation}
\label{eq:defnPsiJ}
\Psi_J:=\int_J\sin(2\theta_\gamma(x))\dr x.
\end{equation}

We need to consider \eqref{diffwint:eq} with $x=a$ and $b\to+\infty$ when we take $\theta_\gamma=\theta_{\gamma,+}$. Let $\omega_{\gamma,+}$ denote the corresponding derivative in this case. Since $\theta_{\gamma,+}(+\infty)=-\pi/4$ is constant (recall \eqref{eq:bcthetapm}) we would expect $\omega_{\gamma,+}(+\infty)=0$. Furthermore $\sin(2\theta_{\gamma,+}(t))\to-1$ as $t\to\infty$, so $\er^{2k\Psi_{[a,t]}}\to0$ as $t\to\infty$. The precise properties that we require are given in the next result; these can be justified by straightforward if somewhat lengthy arguments.

\begin{proposition}
\label{prop:dthetadgamma+lim}
The solution $\theta_{\gamma,+}$ to \eqref{eq:PruferODE} depends on $\gamma$ differentiably, while the derivative satisfies
\[
\omega_{\gamma,+}(a):=\frac{\dr}{\dr\gamma}\,\theta_{\gamma,+}(a)
=-\int_a^\infty\er^{2k\Psi_{[a,x]}}V(x)\dr x
\]
for all $a\in\R$.
\end{proposition}

There is a corresponding result for $\theta_{\gamma,-}$. Proposition \ref{prop:positiveDinc} is now an easy corollary of these results.

\begin{proof}[Proof of Proposition \ref{prop:positiveDinc}]
From \eqref{eq:Deltadefn} we get
\[
\frac{\dr}{\dr\gamma}\,\Delta_V(\gamma)=-\omega_{\gamma,+}(0)+\omega_{\gamma,-}(0).
\]
Now suppose $V\ge0$ (the case $V\le0$ can be handled similarly). By Proposition \ref{prop:dthetadgamma+lim} we have
\[
-\omega_{\gamma,+}(0)=\int_0^\infty\er^{2k\Psi_{[0,x]}}V(x)\dr x.
\]
Since $\er^{2k\Psi_J}>0$ for any (bounded) interval $J\subseteq\R$ the right hand side is non-negative and equal to $0$ only if $V=0$ on $\R^+$ (as an $L^1$ function). A similar argument shows that $\omega_{\gamma,-}(0)$ is also non-negative and equal to $0$ only if $V=0$ on $\R^-$. The result follows.
\end{proof}

\subsection{No gaps}
\label{subsec:nogaps}

For a potential $V\in L^1_\loc$ and interval $I\subseteq\R$ let $\var_I(V)$ denote the total variation of $V$ on $I$. We also say that \emph{$V$ has no gaps on $I$} if
\[
\bigabs{I\cap V^{-1}(0)}=0.
\]

To work with no-gap potentials we need estimates for the integrals of $\cos(2\theta_\gamma(x))$ and $\sin(2\theta_\gamma(x))$; to complement $\Psi_J$ (see \eqref{eq:defnPsiJ}) set 
\[
\Phi_J=\int_J\cos(2\theta_\gamma(x))\dr x
\]
for any interval $J\subseteq\R$. For large $\gamma$ equation \eqref{eq:PruferODE} suggests $\theta_\gamma(x)$ should be changing rapidly wherever $V(x)\neq0$; it follows that $\cos(2\theta_\gamma(x))$ and $\sin(2\theta_\gamma(x))$ should be rapidly oscillating, leading to cancellation in the integrals defining $\Phi_J$ and $\Psi_J$. This idea lies at the heart of the following result (the proof is given in Section \ref{subsec:estPhiJPsiJ}):

\begin{proposition}
\label{prop:intcossinest}
Let $I\subset\R$ be a closed bounded interval, and suppose a potential $V$ satisfies $\var_I(V)<+\infty$ and has no gaps on $I$. Also suppose $\theta_\gamma$ satisfies \eqref{eq:PruferODE} on $I$. For any sub-interval $J\subseteq I$ we have $\Phi_{J},\,\Psi_{J}=o(1)$ as $\gamma\to\infty$, uniformly in $J$ and possible choices of (the initial condition for) the solution $\theta_\gamma$.  
\end{proposition}

These estimates for $\Phi_J$ and $\Psi_J$ lead directly to the following asymptotic information about the change in the value of $\theta_\gamma$ across $I$:

\begin{proposition}
\label{prop:nogapsDeltaVI}
Let $I$, $V$ and $\theta_\gamma$ be as in Proposition \ref{prop:intcossinest}. Write $I=[a,b]$. 
Also suppose $\dr^n\theta_\gamma(b)/\dr\gamma^n$ exists and is bounded in $\gamma$ for $n=1,2$. 
Then, for $n=0,1,2$, 
\begin{equation}
\label{eq:nogapsDeltakest}
\frac{\dr^n}{\dr\gamma^n}\left(\theta_\gamma(b)-\theta_\gamma(a)-\gamma\int_a^b V(x)\dr x\right)=o(1)
\end{equation}
as $\gamma\to\infty$; in particular, $\dr^n\theta_\gamma(a)/\dr\gamma^n$ also exists and is bounded in $\gamma$ for $n=1,2$.
\end{proposition}

\begin{proof}
Integrating \eqref{eq:PruferODE} gives
\[
\theta_\gamma(b)-\theta_\gamma(a)
=\int_a^b\nabla\theta_\gamma(x)\dr x
=\gamma\int_a^b V(x)\dr x+k\Phi_I,
\]
so \eqref{eq:nogapsDeltakest} for $n=0$ follows directly from Proposition \ref{prop:intcossinest}. 

From \eqref{diffwint:eq} we get
\[
\omega_\gamma(b)-\omega_\gamma(x)-\int_x^b V(t)\dr t
=\bigl(1-\er^{2k\Psi_{[x,b]}}\bigr)\omega_\gamma(b)-\int_x^b\bigl(1-\er^{2k\Psi_{[x,t]}}\bigr)V(t)\dr t
\]
for all $x\in I$. By Proposition \ref{prop:intcossinest} we have $\Psi_J=o(1)$ and hence $1-\er^{2k\Psi_J}=o(1)$ as $\gamma\to\infty$, uniformly for all sub-intervals $J\subseteq I$. Setting $W(x)=\omega_\gamma(b)-\int_x^bV(t)\dr t$ it follows that $W(x)$ is bounded in $\gamma$ while $W(x)-\omega_\gamma(x)=o(1)$ as $\gamma\to\infty$, both uniformly for $x\in I$. With $x=a$ this becomes \eqref{eq:nogapsDeltakest} for $n=1$. More generally we have that $\omega_\gamma(x)$ is uniformly bounded for all $\gamma$ and $x\in I$, while
\begin{equation}
\label{eq:o1estW} 
W^2(x)-\omega_\gamma^2(x)=\bigl(W(x)-\omega_\gamma(x)\bigr)\bigl(W(x)+\omega_\gamma(x)\bigr)=o(1)\quad 
\text{as $\gamma\to\infty$,}
\end{equation}
uniformly for $x\in I$.

Arguing as for \eqref{diffwint:eq} we get
\[
\rho_\gamma(a)=\er^{2k\Psi_{I}}\rho_\gamma(b)+4k\int_a^b\er^{2k\Psi_{[a,x]}}\cos(2\theta_\gamma(x))\,\omega_\gamma^2(x)\dr x.
\]
Thus
\begin{align*}
&\rho_\gamma(b)-\rho_\gamma(a)
=\bigl(1-\er^{2k\Psi_I}\bigr)\rho_\gamma(b)+4k\int_a^b\bigl(1-\er^{2k\Psi_{[a,x]}}\bigr)\cos(2\theta_\gamma(x))\,\omega_\gamma^2(x)\dr x\\
&\qquad\qquad{}+4k\int_a^b\cos(2\theta_\gamma(x))\,\bigl(W^2(x)-\omega_\gamma^2(x)\bigr)\dr x
\;-\;4k\int_a^b\cos(2\theta_\gamma(x))\,W^2(x)\dr x.
\end{align*}
As above $\Psi_J=o(1)$ and hence $1-\er^{2k\Psi_J}=o(1)$ as $\gamma\to\infty$, uniformly for all sub-intervals $J\subseteq I$. 
Combined with \eqref{eq:o1estW} and the fact that $\cos(2\theta_\gamma)$ and $\omega_\gamma$ are uniformly bounded it follows that the first three terms on the right hand side are $o(1)$ as $\gamma\to\infty$. The following claim deals with the final term and completes the argument. 

\smallskip

\noindent
\emph{Claim: We have $\int_a^b\cos(2\theta_\gamma(x))\,W^2(x)\dr x=o(1)$ as $\gamma\to\infty$.}
Firstly note that
\[
\nabla\Phi_{[a,x]}=\cos(2\theta_\gamma(x))
\quad\text{and}\quad
\nabla W^2(x)=2W(x)V(x).
\]
Integrating by parts thus gives
\[
\int_a^b\cos(2\theta_\gamma(x))\,W^2(x)\dr x
=\Phi_IW^2(b)-2\int_a^b\Phi_{[a,x]}W(x)V(x)\dr x
\]
since $\Phi_{[a,a]}=0$. By Proposition \ref{prop:intcossinest} we have $\Phi_J=o(1)$ as $\gamma\to\infty$, uniformly for all sub-intervals $J\subseteq I$.
Furthermore $W(x)$ is uniformly bounded for $\gamma$ and $x\in I$ while $V\in L^1([a,b])$. It follows that $\int_a^b\Phi_{[a,x]}W(x)V(x)\dr x=o(1)$ 
as $\gamma\to\infty$, completing the claim.
\end{proof}

Suppose a potential $V\in L^1_\loc$ has support contained in the bounded interval $I=[a,b]$. Clearly constant functions taking values in $\pi/4+\pi\Z/2$ solve \eqref{eq:PruferODE} outside $I$. The condition \eqref{eq:bcthetapm} (together with the uniqueness and continuity of $\theta_{\gamma,\pm}$) then gives us $\theta_{\gamma,-}(x)=\pi/4$ for $x\le a$ and $\theta_{\gamma,+}(x)=-\pi/4$ for $x\ge b$ (this can also be seen from the form of the corresponding solutions to \eqref{eq:l0ode}). When defining $\Delta_V$ in this case it is convenient to choose the left endpoint of $I$ as the point at which to check whether $\theta_{\gamma,-}$ and $\theta_{\gamma,+}$ can be `matched'; alternatively, we can keep our current definition of $\Delta_V$ if we simply translate our problem so that $a=0$ (see Remark \ref{rem:translateV}). Making such a choice we get
\[
\Delta_V(\gamma)
=-\frac\pi2-\theta_{\gamma,+}(a)+\theta_{\gamma,-}(a)
=\theta_{\gamma,+}(b)-\theta_{\gamma,+}(a),
\]
where $\theta_{\gamma,+}(b)=-\pi/4$ is constant (as a function of $\gamma$). We will work with $\Delta_V$ in this form for the remainder of the present section.

For compactly supported potentials without gaps we can now rephrase the conclusions of Proposition \ref{prop:nogapsDeltaVI} to get the following improved and extended version of Proposition \ref{prop:Dgasym}:

\begin{corollary}
\label{cor:derivDgasym}
Suppose $V\in\BV$ has no gaps. Then, for $n=0,1,2$, 
\[
\frac{\dr^n}{\dr\gamma^n}\left(\Delta_V(\gamma)-\gamma\int_{\R}V(x)\dr x\right)=o(1)
\]
as $\gamma\to\infty$.
\end{corollary}

When $\int_{\R}V(x)\dr x\neq0$ it follows that $\Delta_V$ is monotonic for sufficiently large $\gamma$. In this case Theorem \ref{thm:Vnogapaysm} can be proved using an argument very similar to that used for Theorem \ref{thm:singsigngvalasym}. The case $\int_{\R}V(x)\dr x=0$ can be treated with a separate observation.

\begin{proof}[Proof of Theorem \ref{thm:Vnogapaysm}]
If $\int_{\R}V(x)\dr x=0$ Corollary \ref{cor:derivDgasym} (for $n=0$) enables us to find $S>0$ so that $\abs{\Delta_V(\gamma)}<\pi/2$ when $\abs{\gamma}>S$. Then $\gspec{V}\cap\R\subseteq[-S,S]$ by Proposition \ref{prop:gspeccharD}. However $\gspec{V}$ is a discrete subset of $\C$ (Theorem \ref{thm:disgspec}) so $\gspec{V}\cap[-S,S]$ contains at most finitely many points.

Now suppose $\int_{\R}V(x)\dr x\neq0$. By Corollary \ref{cor:derivDgasym} (for $n=1$) there exists $S>0$ such that $\Delta_V(\gamma)$ is strictly monotonic for all $\gamma\ge S$. Suppose $R>S$ and let $I_{S,R}$ denote the closed interval with endpoints $\Delta_V(S)$ and $\Delta_V(R)$. Arguing as for Theorem \ref{thm:singsigngvalasym} we then get
\begin{align*}
\#(\gspec{V}\cap[S,R])=\#(I_{S,R}\cap(\Z+1/2)\pi\bigr)
=\frac{\abs{I_{S,R}}}{\pi}+O(1).
\end{align*}
However Corollary \ref{cor:derivDgasym} (for $n=0$) also gives
\[
\abs{I_{S,R}}=\abs{\Delta_V(R)-\Delta_V(S)}
=\abs{\Delta_V(R)}+O(1)
=R\left\lvert\int_{\R}V(x)\dr x\right\rvert\;+O(1)
\]
as $R\to\infty$ (note that, $S$ is fixed). The fact that $\gspec{V}\cap[0,S]$ contains at most finitely many points (see above) completes the argument.
\end{proof}

\subsection{One gap}
\label{sub:onegap}
Let $V\in\BV$ be a one-gap potential as considered in Section \ref{sec:onegap}. 
We can use Proposition \ref{prop:nogapsDeltaVI} to estimate the change in $\theta_{\gamma,+}$ across the intervals $[a_j,b_j]$ for $j=1,2$.
Information about the change in $\theta_{\gamma,+}$ across the gap $(b_1,a_2)$ will be obtained from the next result.  

\begin{lemma}
\label{lem:gaptheta}
Suppose $\nabla\theta=k\cos(2\theta)$ on some interval $I=(a,b)$ and set $\alpha=\tanh(k(b-a))$. Then 
\begin{equation}
\label{eq:gaptheta}
\sin\bigl(\theta(b)-\theta(a)\bigr)=\alpha\cos\bigl(\theta(b)+\theta(a)\bigr).
\end{equation}
\end{lemma}

\begin{proof}
Let $L=\dfrac{\pi}4+\dfrac{\pi}2\Z$, the zero set of $\cos(2\theta)$. Values in $L$ give constant solutions to $\nabla\theta=k\cos(2\theta)$. 
For such solutions $\theta(b)-\theta(a)=0$ and $\theta(b)+\theta(a)\in\dfrac{\pi}2+\pi\Z$, so both sides of \eqref{eq:gaptheta} are zero. 

Now suppose $\theta(x_0)\notin L$ for some $x_0\in I$. Using the uniqueness and continuity of solutions to the equation $\nabla\theta=k\cos(2\theta)$ it follows that $\theta(x)$ must remain within the same connected component of $\R\setminus L$ for all $x\in I$. In particular $\cos(2\theta(x))\neq0$. Now set
\[
F(\theta)=\frac{1-\sin(2\theta)}{\cos(2\theta)}=\frac{\cos(\theta)-\sin(\theta)}{\cos(\theta)+\sin(\theta)}.
\]
Then $F'(\theta)=-2F(\theta)/\cos(2\theta)$ so
\[
\nabla\bigl(e^{2kx}F(\theta(x))\bigr)
=e^{2kx}F(\theta(x))\left(2k-2\,\frac{\nabla\theta(x)}{\cos(2\theta(x))}\right)
=0.
\]
Hence $e^{2kb}F(\theta(b))=e^{2ka}F(\theta(a))$. The second expression for $F(\theta)$ then leads to
\[
\frac{e^{2kb}-e^{2ka}}{e^{2kb}+e^{2ka}}\,\cos(\theta(b)+\theta(a))=\sin(\theta(b)-\theta(a)).
\]
The first part of the expression on the left hand side is just $\tanh(k(b-a))=\alpha$. 
\end{proof}

\smallskip

As in the discussion proceeding Corollary \ref{cor:derivDgasym} we observe that $\theta_{\gamma,-}(x)=\pi/4$ for $x\le a_1$ and $\theta_{\gamma,+}(x)=-\pi/4$ for $x\ge b_2$
(note that $V$ has support contained in $[a_1,b_2]$). 
We shall also define $\Delta_V$ by choosing $a_1$ as the point at which to check whether $\theta_{\gamma,-}$ and $\theta_{\gamma,+}$ can be `matched'.
It follows that
\[
\Delta_V
=\Delta_V(\gamma)
=\theta_{\gamma,+}(b_2)-\theta_{\gamma,+}(a_1).
\]
Now set $\Delta_j=\Delta_j(\gamma)=\theta_{\gamma,+}(b_j)-\theta_{\gamma,+}(a_j)$ for $j=1,2$. Then
\[
\theta_{\gamma,+}(a_2)-\theta_{\gamma,+}(b_1)=\Delta_V-\Delta_1-\Delta_2
\quad\text{and}\quad
\theta_{\gamma,+}(a_2)+\theta_{\gamma,+}(b_1)=\Delta_1-\Delta_2-\Delta_V-\dfrac\pi2.
\]
Since $V=0$ on $(b_1,a_2)$ Lemma \ref{lem:gaptheta} then gives
\begin{align}
&\sin\bigl(\Delta_V-\Delta_1-\Delta_2\bigr)=\alpha\sin\bigl(\Delta_1-\Delta_2-\Delta_V\bigr)\nonumber\\
\Longrightarrow\quad&
\bigl[\cos(\Delta_1+\Delta_2)+\alpha\cos(\Delta_1-\Delta_2)\bigr]\sin(\Delta_V)\nonumber\\
\label{eq:genrelgapD1D2V}
&\qquad{}=\bigl[\sin(\Delta_1+\Delta_2)+\alpha\sin(\Delta_1-\Delta_2)\bigr]\cos(\Delta_V).
\end{align}

\begin{lemma}
\label{lem:altspecchar1gap}
We have $\gamma\in\gspec{V}\cap\R$ iff 
\begin{equation}
\label{eq:altspecchar1gap}
\cos(\Delta_1+\Delta_2)+\alpha\cos(\Delta_1-\Delta_2)=0.
\end{equation}
\end{lemma}

\begin{proof}
If $\gamma\in\gspec{V}\cap\R$ we get $\cos(\Delta_V)=0$ and $\sin(\Delta_V)=\pm1$ from Proposition \ref{prop:gspeccharD}, so \eqref{eq:altspecchar1gap} follows from \eqref{eq:genrelgapD1D2V}.
Now suppose \eqref{eq:altspecchar1gap} holds. Then \eqref{eq:genrelgapD1D2V} gives either $\cos(\Delta_V)=0$, in which case $\gamma\in\gspec{V}\cap\R$ by Proposition \ref{prop:gspeccharD}, or
\begin{equation}
\label{eq:sinveraltspecchar1gap}
\sin(\Delta_1+\Delta_2)+\alpha\sin(\Delta_1-\Delta_2)=0.
\end{equation}
However \eqref{eq:altspecchar1gap} and \eqref{eq:sinveraltspecchar1gap} imply
\[
1=\cos^2(\Delta_1+\Delta_2)+\sin^2(\Delta_1+\Delta_2)
=\alpha^2\bigl(\cos^2(\Delta_1-\Delta_2)+\sin^2(\Delta_1-\Delta_2)\bigr)
=\alpha^2,
\]
contradicting the fact that $\alpha=\tanh(k(a_2-b_1))\in(0,1)$.
\end{proof}

Before giving the proofs of the first two results in Section \ref{sec:onegap} we note that
if $g$ satisfies \eqref{eq:decayprop012} and $\beta\in\R$ then it is straightforward to check that 
$\cos(\beta x+g(x))-\cos(\beta x)$ also satisfies \eqref{eq:decayprop012}.

\begin{proof}[Proofs of Theorems \ref{thm:1gint0} and \ref{thm:1ggen}]
By Proposition \ref{prop:nogapsDeltaVI} applied to $V_j$ on the interval $[a_j,b_j]$, $j=1,2$, we have
\[
\Delta_1+\Delta_2=(v_1+v_2)\gamma+g_+(\gamma)
\quad\text{and}\quad
\Delta_1-\Delta_2=(v_1-v_2)\gamma+g_-(\gamma)
\]
for some functions $g_+$ and $g_-$ which satisfy \eqref{eq:decayprop012}.
Using the observation proceeding the proof we can then write
\[
\cos(\Delta_1+\Delta_2)+\alpha\cos(\Delta_1-\Delta_2)
=\cos(\abs{v_1+v_2}\gamma)+\alpha\cos(\abs{v_1-v_2}\gamma)+\psi(\gamma)
\]
for some $\psi$ which also satisfies \eqref{eq:decayprop012}.
By Lemma \ref{lem:altspecchar1gap} we thus have $\gamma\in\gspec{V}\cap\R$ iff 
\begin{equation}
\label{eq:altspecchar1gap2}
\cos(\abs{v_1+v_2}\gamma)+\alpha\cos(\abs{v_1-v_2}\gamma)+\psi(\gamma)=0.
\end{equation}
Since $\gspec{V}$ is discrete (Theorem \ref{thm:disgspec}) 
the solutions of \eqref{eq:altspecchar1gap2} (in $\gamma$) form a discrete subset of $\R$.

\smallskip

If $v_1+v_2=0$ then \eqref{eq:altspecchar1gap2} reduces to $1+\alpha\cos(2\abs{v_1}\gamma)+\psi(\gamma)=0$. 
However $0<\alpha<1$ and $\psi(\gamma)=o(1)$ as $\gamma\to\infty$, so we can we can choose $S$ 
such that $\abs{\psi(\gamma)}<1-\alpha$ for all $\gamma\ge S$. Then $\gspec{V}\cap\R\subset[-S,S]$ 
(recall that $\gspec{V}$ is symmetric about $0$). The discreteness of $\gspec{V}$ then limits $\gspec{V}\cap\R$ 
to at most finitely many points, establishing Theorem \ref{thm:1gint0}.

\smallskip

Now suppose $v_1+v_2\neq0$. Writing $x=\abs{v_1+v_2}\gamma$, \eqref{eq:altspecchar1gap2} then gives
\[
\#(\gspec{V}\cap[0,R])=\#\bigl\{x\in[0,\abs{v_1+v_2}R]:\cos(x)+\alpha\cos(\beta x)+\phi(x)=0\bigr\},
\]
where $\phi(x)=\psi(x/\abs{v_1+v_2})$ which clearly satisfies \eqref{eq:decayprop012}. 
Theorem \ref{thm:1ggen} now follows directly from Theorem \ref{thm:coszeros}.
\end{proof}

\begin{proof}[Proof of Theorem \ref{thm:mrallasym}]
Since $A<u$ we can choose $w$ so that $0<v<A<w\le u$ and $w/v$ is irrational. 
Set $v_0=(u-w)/2\ge0$, $v_1=(v-w)/2<0$, $v_2=(v+w)/2>0$ and
\[
\beta=\lrabs{\frac{v_1-v_2}{v_1+v_2}}=\frac{w}{v}.
\]
Then $1<\dfrac{A}{v}<\beta$; thus we can choose $\alpha\in(1/\beta,1)$ so that $\nu_{\alpha,\beta}=\dfrac{A}{v}$. Now set $g=\dfrac1{k}\,\tanh^{-1}(\alpha)>0$ and consider the potentials 
\[
V_1(x)=W\bigl(x;[-1,0];\{v_1\})
\quad\text{and}\quad
V_2(x)=W\bigl(x;[g,g+1,g+2];\{v_2+v_0,-v_0\}\bigr)
\]
(see \eqref{eq:Vpiece} for notation).
Then $V=V_1+V_2\in\BV$ is a one gap potential, with gap from $0$ to $g$. 
Since $\alpha\beta>1$ and $\beta$ is irrational Theorem \ref{thm:1ggen} shows that \eqref{eq:genCasym} holds with $C=\abs{v_1+v_2}\,\nu_{\alpha,\beta}=v\,\dfrac{A}{v}=A$. Furthermore 
\[
\int_\R V(x)\,\dr x=\int_\R V_1(x)\,\dr x+\int_\R V_2(x)\,\dr x=v_1+(v_2+v_0-v_0)=v
\]
while $\norm{V}_{L^1}=\norm{V_1}_{L^1}+\norm{V_2}_{L^1}=-v_1+(v_2+2v_0)=u$.
\end{proof}

\section{Technicalities}
\label{sec:tech}

\subsection{Asymptotic behaviour of ode solutions}
\label{subsec:asym}

We return to viewing our basic equation $\Drc{\gamma V}\psib=0$ as the system of ordinary differential equations \eqref{eq:l0ode}. When $V\equiv0$ we have the exponential solutions
\begin{equation}
\label{eq:V=0expsoln}
e^{\pm kx}\begin{pmatrix}1\\\pm1\end{pmatrix}
\end{equation}
(see Lemma \ref{lem:constV}). When $V\in\Vclass_0$ (or, more generally, $V$ belongs to the locally $L^1$ version of $\Vclass_0$) we can find solutions of \eqref{eq:l0ode} with similar asymptotic properties as $x\to\pm\infty$ (see \cite[Chapter X]{Hartman}).
In particular, there are non-trivial solutions $\asol{+}{\pm}$ and $\asol{-}{\pm}$ which satisfy
\[
\lim_{x\to+\infty}x^{-1}\log\bigabs{\asol{+}{\pm}(x)}=\pm k
\quad\text{and}\quad
\lim_{x\to-\infty}x^{-1}\log\bigabs{\asol{-}{\pm}(x)}=\pm k.
\]
Since $k\neq0$ $\asol{+}{+}(x)$ and $\asol{+}{-}(x)$ have different asymptotic behaviour as $x\to+\infty$ (the growth of $\asol{+}{\pm}(x)$ as $x\to+\infty$ is roughly like $e^{\pm kx}$); it follows that these solutions must be linearly independent. A similar discussion applies to $\asol{-}{+}$ and $\asol{-}{-}$.

\begin{lemma}
\label{lem:simeval}
If $V\in\Vclass_0$ then either $0\notin\spec(\Drc{V})$ or $0$ is a simple eigenvalue of $\Drc{V}$.
\end{lemma}

\begin{proof}
We have $0\in\spec(\Drc{V})$ iff $0$ is an isolated eigenvalue of $\Drc{V}$ (see Remark \ref{rem:lambdaspec}). Now suppose $\psib$ is an eigenfunction corresponding to $0$. Then $\psib$ satisfies \eqref{eq:l0ode} (with $\gamma=1$). Since $\asol{+}{+}$ and $\asol{+}{-}$ are linearly independent solutions of this equation we must have $\psib=\alpha_+\asol{+}{+}+\alpha_-\asol{+}{-}$ for some constants $\alpha_\pm$. Restricting to the interval $\R^+$ we have $\asol{+}{+}\notin L^2(\R^+)$ while $\asol{+}{-}\in L^2(\R^+)$. It follows that $\alpha_+=0$, and so $\psib$ is a multiple of $\asol{+}{-}$. This must also be true for any other eigenfunction corresponding to $0$, so any two such eigenfunctions are linearly dependent. 
\end{proof}

\begin{remark}
Using a similar argument we can also get $\psib=\beta_+\asol{-}{+}$ for some constant $\beta_+$, showing that $\asol{+}{-}$ and $\asol{-}{+}$ are linearly dependent. In fact this is an alternative characterisation of when $0$ is an eigenvalue of $\Drc{V}$.
\end{remark}

When $V\in\Vclass_1$ solutions to \eqref{eq:l0ode} have well defined leading order asymptotics as $x\to\pm\infty$; these asymptotics are solutions to the same equation with $V\equiv0$, so must be linear combinations of the exponential functions given in \eqref{eq:V=0expsoln} (see \cite[Chapter X]{Hartman}).
In particular, we can choose $\asol{+}{\pm}$ and $\asol{-}{\pm}$ so that
\begin{subequations}
\label{eq:inftyasylim}
\begin{equation}
\label{eq:+inftyasylim}
\lim_{x\to+\infty}\left\lvert\er^{\mp kx}\left(\asol{+}{\pm}(x)-\er^{\pm kx}\begin{pmatrix}1\\\pm1\end{pmatrix}\right)\right\rvert=0\phantom{.}
\end{equation}
while
\begin{equation}
\lim_{x\to-\infty}\left\lvert\er^{\mp kx}\left(\asol{-}{\pm}(x)-\er^{\pm kx}\begin{pmatrix}1\\\pm1\end{pmatrix}\right)\right\rvert=0.
\end{equation}
\end{subequations}
Furthermore $\asol{+}{-}$ and $\asol{-}{+}$ are uniquely determined by these asymptotic conditions. (The solutions $\asol{+}{+}$ and $\asol{-}{-}$ are however only determined up to the addition of a multiple of $\asol{+}{-}$ and $\asol{-}{+}$ respectively.)

\begin{proof}[Proof of Lemma \ref{lem:+thetalim}]
We consider the case $x\to+\infty$; $x\to-\infty$ can be handled similarly. Since $\asol{+}{+}$ and $\asol{+}{-}$ are linearly independent we can write $\psib=\alpha_+\asol{+}{+}+\alpha_-\asol{+}{-}$ for some constants $\alpha_\pm$ (which can't both be $0$). Now by \eqref{eq:+inftyasylim}
\[
\lim_{x\to+\infty}\frac{\psib_1(x)}{\psib_2(x)}
=\lim_{x\to+\infty}\frac{(\alpha_+\asol{+}{+}(x)+\alpha_-\asol{+}{-}(x))_1}{(\alpha_+\asol{+}{+}(x)+\alpha_-\asol{+}{-}(x))_2}
=\begin{cases}
1&\text{if $\alpha_+\neq0$}\\
-1&\text{if $\alpha_+=0$.}
\end{cases}
\]
However $\tan\theta=\psi_1/\psi_2$ so \eqref{eq:+thetalim} (for $x\to+\infty$) follows. On the other hand, $\psib\in L^2(\R^+)$ iff $\alpha_+=0$, leading to the second part of the result.
\end{proof}

\subsection{Estimates and asymptotics for $\Delta_V$}
\label{subsec:DeltaV}

Propositions \ref{prop:genDeltaVest} and \ref{prop:Dgasym} both follow almost directly from the following estimate:

\begin{lemma}
\label{lem:tailest}
Let $V\in\Vclass_1$ and suppose $\theta$ solves \eqref{eq:PruferODE} with $\theta(+\infty)=-\pi/4$. For any $x\in\R$ we then have
\[
\Bigl\lvert\theta(x)+\frac{\pi}4\Bigr\rvert
\le h\left(\abs{\gamma}\int_x^\infty\abs{V(t)}\dr t\right).
\]
A similar result holds on $(-\infty,x]$ when $\theta(-\infty)=\pi/4$.
\end{lemma}

The basic idea is that if $\theta(x)$ increases on an interval $[a,b]$ crossing a range of values where $k\cos(2\theta(x))$ is non-positive then we must have 
\[
\int_a^b\gamma V(x)\dr x\ge\theta(b)-\theta(a).
\]
We can then add the contributions from each such crossing.

\begin{proof}
Let $n\in\Z$ and suppose $\theta(x)\in\bigl[-(n+3/4)\pi,-(n+1/4)\pi\bigr]$ for all $x$ in some interval $[a,b]\subset\R$. Then $k\cos(2\theta(x))\le0$ so $\nabla\theta(x)\le\gamma V(x)$, and hence
\begin{equation}
\label{eq:tailest1step}
\theta(b)-\theta(a)\le\int_a^b\gamma V(t)\dr t.
\end{equation}
This estimate continues to hold with the obvious interpretation when $b=+\infty$. 

Now let $x\in\R$ and suppose $\theta(x)<-\pi/4$. Choose $n\in\N_0$ so that $\theta(x)\in\bigl(-(n+5/4)\pi,-(n+1/4)\pi\bigr]$. Using the continuity of $\theta$, and the assumption that $\theta(+\infty)=-\pi/4$, we can now choose a sequence of points 
\[
x\le a_n<b_n<a_{n-1}<b_{n-1}<\dots<a_0<b_0\le+\infty
\]
such that
\begin{itemize}
\item[(i)]
For $j=0,\dots,n$ and $x\in[a_j,b_j]$ we have $\theta(x)\in\bigl[-(j+3/4)/\pi,-(j+1/4)\pi\bigr]$.
\item[(ii)]
For $j=0,\dots,n$ we have $\theta(a_j)=-(j+3/4)\pi$ and $\theta(b_j)=-(j+1/4)\pi$, with the exception that $\theta(a_n)=\theta(x)$ if $\theta(x)\in\bigl(-(n+3/4)\pi,-(n+1/4)\pi\bigr]$.
\end{itemize}
Applying \eqref{eq:tailest1step} we then get
\begin{equation}
\label{eq:tailestnstep}
\sum_{j=0}^n(\theta(b_j)-\theta(a_j))
\le\sum_{j=0}^n\int_{a_j}^{b_j}\gamma V(t)\dr t
\le\int_x^\infty\abs{\gamma V(t)}\dr t.
\end{equation}
If $\theta(x)\in\bigl(-(n+5/4)\pi,-(n+3/4)\pi\bigr]$ then each term in the sum on the left of \eqref{eq:tailestnstep} is $\pi/2$ so
\[
(n+1)\frac\pi2\le\int_x^\infty\abs{\gamma V(t)}\dr t.
\]
It follows that
\[
(n+1)\frac\pi2\le\frac{\pi}2\left\lfloor\frac{2}{\pi}\int_x^\infty\abs{\gamma V(t)}\dr t\right\rfloor.
\]
On the other hand $\theta(x)\ge-(n+5/4)\pi$ so
\[
-\theta(x)-\frac{\pi}4\le(n+1)\pi\le h\left(\int_x^\infty\abs{\gamma V(t)}\dr t\right).
\]
Alternatively suppose $\theta(x)\in\bigl(-(n+3/4)\pi,-(n+1/4)\pi\bigr]$ so $\theta(a_n)=\theta(x)$. Then the term in the sum on the left of \eqref{eq:tailestnstep} is $\pi/2$ for $j=0,\dots,n-1$, so
\[
n\frac{\pi}2=\sum_{j=0}^{n-1}(\theta(b_j)-\theta(a_j))
\le\int_x^\infty\abs{\gamma V(t)}\dr t,
\]
leading to 
\[
n\frac\pi2\le\frac{\pi}2\left\lfloor\frac{2}{\pi}\int_x^\infty\abs{\gamma V(t)}\dr t\right\rfloor.
\]
Furthermore
\begin{align*}
\sum_{j=0}^n(\theta(b_j)-\theta(a_j))
=-\Bigl(n+\frac14\Bigr)\pi-\theta(a_n)+n\frac{\pi}2
=-\theta(x)-\frac{\pi}4-n\frac{\pi}2.
\end{align*}
Using \eqref{eq:tailestnstep} again gives
\[
-\theta(x)-\frac{\pi}4
\le \int_x^\infty\abs{\gamma V(t)}\dr t\,+n\frac{\pi}2
\le h\left(\int_x^\infty\abs{\gamma V(t)}\dr t\right).
\]
A similar argument can be used to deal with the case $\theta(x)>-\pi/4$ (we need to consider intervals where $\theta(x)$ decreases across the range $\bigl[(n-1/4)\pi,(n+1/4)\pi\bigr]$). 
\end{proof}

\begin{proof}[Proof of Proposition \ref{prop:genDeltaVest}]
By Lemma \ref{lem:tailest} we get
\[
\Bigl\lvert\theta_{\gamma,+}(0)+\frac{\pi}4\Bigr\rvert
\le h\left(\abs{\gamma}\int_0^\infty\abs{V(t)}\dr t\right).
\]
The equivalent estimate on $(-\infty,0]$ gives
\[
\Bigl\lvert\theta_{\gamma,-}(0)-\frac{\pi}4\Bigr\rvert
\le h\left(\abs{\gamma}\int_{-\infty}^0\abs{V(t)}\dr t\right).
\]
The result now follows from \eqref{eq:Deltadefn} and the fact that $h(a)+h(b)\le h(a+b)$ for any $a,b\ge0$.
\end{proof}

\begin{proof}[Proof of Proposition \ref{prop:Dgasym}]
Recalling \eqref{eq:Deltadefn} we can write
\[
\Delta_V(\gamma)=\Delta_V^+(\gamma)+\Delta_V^-(\gamma).
\]
where $\Delta_V^{\pm}(\gamma):=-\pi/4\mp\theta_{\gamma,\pm}(0)$.
Now suppose $K\ge0$. Since $\theta_{\gamma,+}$ satisfies \eqref{eq:PruferODE} we get
\[
\theta_{\gamma,+}(K)-\theta_{\gamma,+}(0)
=\gamma\int_0^{K}V(x)\dr x+k\int_0^{K}\cos(2\theta_{\gamma,+}(x))\dr x
\]
so
\[
\left\lvert\theta_{\gamma,+}(K)-\theta_{\gamma,+}(0)-\gamma\int_0^\infty V(x)\dr x\right\rvert
\le\abs{\gamma}\int_{K}^\infty\abs{V(x)}\dr x+kK.
\]
On the other hand Lemma \ref{lem:tailest} gives us
\[
\left\lvert\theta_{\gamma,+}(K)+\frac{\pi}4\right\rvert
\le h\left(\abs{\gamma}\int_{K}^\infty\abs{V(x)}\dr x\right)
\le2\abs{\gamma}\int_{K}^\infty\abs{V(x)}\dr x.
\]
Therefore
\begin{align*}
&\left\lvert\Delta_V^+(\gamma)-\gamma\int_0^\infty V(x)\dr x\right\rvert\\
&\qquad{}=\left\lvert-\theta_{\gamma,+}(K)-\frac{\pi}4+\theta_{\gamma,+}(K)-\theta_{\gamma,+}(0)-\gamma\int_0^\infty V(x)\dr x\right\rvert\\
&\qquad{}\le3\abs{\gamma}\int_{K}^\infty\abs{V(x)}\dr x+kK.
\end{align*}
Now choose $K=K_\gamma\ge0$ for each $\gamma$ so that $K_\gamma\to\infty$ and $\abs{\gamma}^{-1}K_\gamma\to0$ as $\abs{\gamma}\to\infty$ (we can take $K_\gamma\sim\abs{\gamma}^\mu$ with $0<\mu<1$, for example). Then $\int_{K_\gamma}^\infty\abs{V(x)}\dr x=o(1)$ (since $V\in L^1$) so
\[
3\abs{\gamma}\int_{K_\gamma}^\infty\abs{V(x)}\dr x+kK_\gamma=o(\gamma)
\]
as $\abs{\gamma}\to\infty$. A similar estimate can be obtained for $\Delta_V^-(\gamma)$.
\end{proof}

\subsection{Estimates for $\Phi_J$ and $\Psi_J$}
\label{subsec:estPhiJPsiJ}

\subsubsection*{Intervals}

It is easiest to deal with the potential $V$ in pieces where it is single-signed and bounded away from $0$; the next result is the key to identifying these pieces and sets up some of our notation.

\begin{lemma}
\label{lem:intervals}
Suppose $V\in\BV$ has no gaps and set $I=\supp(V)$. For each $\epsilon>0$ there exists a finite collection of disjoint closed intervals $\Ien\subseteq I$, $n\in\Nind$, such that 
\begin{itemize}
\item[(i)]
For each $n\in\Nind$ $V$ has constant sign on $\Ien$ and $\abs{V(x)}\ge\epsilon$ for all $x\in \Ien$.
\item[(ii)]
Setting $\Ee=I\setminus\bigcup_{n\in\Nind}\Ien$ we have $\abs{\Ee}\to0$ as $\epsilon\to0$.
\item[(iii)]
$\#\Nind\le\nu\epsilon^{-1}$, where $\nu:=\var_\R(V)$ is the total variation of $V$.
\end{itemize}
\end{lemma}

We would ideally like $\Ee$ to be $V^{-1}((-\epsilon,\epsilon))$, so the union of the $\Ien$'s would be 
\begin{equation}
\label{eq:easychoiceSe}
\{x:V(x)\le-\epsilon\}\cup\{x:V(x)\ge\epsilon\}.
\end{equation}
There are however several technical issues associated with this choice;
\begin{itemize}
\item
We are not assuming that $V$ is continuous so \eqref{eq:easychoiceSe} may not be closed (or even particularly well behaved).
\item
Even if $V$ is continuous the number of intervals in $V^{-1}((-\epsilon,\epsilon))$ may be uncontrollable (even infinite); we need a useful bound on this quantity.
\end{itemize}
The technicalities in the following argument arise from the need to deal with these issues.

\begin{proof}[Proof of Lemma \ref{lem:intervals}]
Set
\[
X_\epsilon^+=\bigcup_{x\in V^{-1}((-\infty,2\epsilon))}(x-\epsilon,x+\epsilon)
\quad\text{and}\quad
X_\epsilon^-=\bigcup_{x\in V^{-1}((-2\epsilon,+\infty))}(x-\epsilon,x+\epsilon);
\]
(so $X_\epsilon^+$ is the `$\epsilon$-neighbourhood' of $\{x:V(x)<2\epsilon\}$, while $X_\epsilon^-$ is the `$\epsilon$-neighbourhood' of $\{x:V(x)>-2\epsilon\}$).

\medskip

Since $X_\epsilon^+$ is an open subset of $\R$ it consists of a countable union of disjoint open intervals; let $Y_\epsilon^+\subseteq X_\epsilon^+$ be the union of those intervals which intersect $V^{-1}((-\infty,\epsilon))$. If $J=(\alpha,\beta)$ is a maximal interval in $Y_\epsilon^+$ it follows that we can find $x\in J$ with $V(x)<\epsilon$. Furthermore $J$ must also be a maximal interval in $X_\epsilon^+$ so, if $J$ is bounded, its endpoints satisfy $\alpha,\beta\notin X_\epsilon^+$. Thus $V(\alpha),V(\beta)\ge2\epsilon$ and so $\var_{\,\overline{J}\,}(V)\ge\abs{V(x)-V(\alpha)}+\abs{V(\beta)-V(x)}\ge2\epsilon$. If $J$ is a half-infinite interval a similar argument shows that $\var_{\,\overline{J}\,}(V)\ge\epsilon$. 

Now let $M_\epsilon^+$ and $N_\epsilon^+$ denote the number of semi-infinite and bounded maximal intervals in $Y_\epsilon^+$ (we're presently allowing $N_\epsilon^+=+\infty$; if $Y_\epsilon^+=\R$ set $M_\epsilon^+=2$, $N_\epsilon^+=-1$). Then
\[
\nu=\var_{\R}(V)\ge\var_{\,\overline{Y_\epsilon^+}\,}(V)\ge \epsilon M_\epsilon^++2\epsilon N_\epsilon^+
\quad\Longrightarrow\quad
N_\epsilon^+\le\frac12(\nu\epsilon^{-1}-M_\epsilon^+);
\]
in particular, $N_\epsilon^+$ must be finite. However $V$ has compact support, so $Y_\epsilon^+$ must be unbounded both above and below. Since $Y_\epsilon^+$ consists of a finite collection of intervals, it must therefore contain semi-infinite intervals at either end; that is, $M_\epsilon^+=2$. The set $Z_\epsilon^+:=\R\setminus Y_\epsilon^+$ is then the union of $N_\epsilon^++1$ closed bounded intervals; write $Z_\epsilon^+=\bigcup_{n\in\Nind^+}\Ien$ where $\Nind^+$ is some indexing set with $\#\Nind^+=N_\epsilon^++1$. It is straightforward to check that we have $V(x)\ge\epsilon$ for any $x\in Z_\epsilon^+$. 

\medskip

We can similarly define $Y_\epsilon^-$, $Z_\epsilon^-$, $\Ien$ for $n\in\Nind^-$ and $N_\epsilon^-$. Set $\Nind=\Nind^+\sqcup\Nind^-$. Property (i) is immediate, while (iii) holds since
\[
\#\Nind
=\#\Nind^++\#\Nind^-
=N_\epsilon^++N_\epsilon^-+2
\le2\frac12(\nu\epsilon^{-1}-2)+2
=\nu\epsilon^{-1}.
\]
Now $\bigcup_{n\in\Nind}\Ien=Z_{\epsilon}^+\cup Z_{\epsilon}^-=:Z_\epsilon$. If $0<\epsilon_1<\epsilon_2$ a straightforward check gives us
\[
X_{\epsilon_1}^\pm\subseteq X_{\epsilon_2}^\pm
\Longrightarrow Y_{\epsilon_1}^\pm\subseteq Y_{\epsilon_2}^\pm
\Longrightarrow Z_{\epsilon_1}^\pm\supseteq Z_{\epsilon_2}^\pm,
\]
so $Z_{\epsilon_1}\supseteq Z_{\epsilon_2}$ and hence $\Ee[\epsilon_1]=I\setminus Z_{\epsilon_1}\subseteq I\setminus Z_{\epsilon_2}=\Ee[\epsilon_2]$. Since $\Ee\subseteq I$ for all $\epsilon>0$ and $\abs{I}<+\infty$, property (ii) will now follow if we can show $\bigabs{\bigcap_{\epsilon>0}\Ee}=0$ 
(see \cite{R}, for example).

\medskip

For each $\delta>0$ let
\[
\Omega_\delta^\pm=\bigl\{x\in{\textstyle\bigcap_{\epsilon>0} X_\epsilon^\pm}\,:\,\pm V(x)>\delta\bigr\}.
\]
Suppose $x_1,\dots,x_N$ are distinct points in $\Omega_\delta^+$ and let $0<\epsilon<\min_{1\le i,j\le N}\abs{x_i-x_j}$. From the definition of $\Omega_\delta^+$ we can find $x_{j,\epsilon}$ with $V(x_{j,\epsilon})<2\epsilon$ and $\abs{x_{j,\epsilon}-x_j}<\epsilon$ for $j=1,\dots,N$. It follows that
\[
\nu=\var_{\R}(V)
\ge\sum_{j=1}^N\bigabs{V(x_{j,\epsilon})-V(x_j)}
\ge N(\delta-2\epsilon).
\]
Taking $\epsilon\to0^+$ gives $N\delta\le\nu$. Hence $\Omega_\delta^+$ is finite (with $\#\Omega_\delta^+\le\nu\delta^{-1}$). Clearly we also have $\Omega_{\delta_1}^+\supseteq \Omega_{\delta_2}^+$ if $0<\delta_1<\delta_2$. 

Similar properties hold for $\Omega_\delta^-$. It follows that the set
\[
\Omega:=\bigcap_{\epsilon>0}(X_\epsilon^+\cap X_\epsilon^-)\setminus V^{-1}(0)
\subseteq\bigcup_{\delta>0}\bigl(\Omega_\delta^+\cup\Omega_\delta^-\bigr)
=\bigcup_{n\in\N}\bigl(\Omega_{1/n}^+\cup\Omega_{1/n}^-\bigr)
\]
is countable (it is contained in a countable union of finite sets). However 
\[
\Ee=I\setminus Z_\epsilon
=I\setminus(Z_\epsilon^+\cup Z_\epsilon^-)
\subseteq Y_\epsilon^+\cap Y_\epsilon^-
\subseteq X_\epsilon^+\cap X_\epsilon^-
\]
so
\[
\bigcap_{\epsilon>0}\Ee
\subseteq I\cap\bigcap_{\epsilon>0}(X_\epsilon^+\cap X_\epsilon^-)
\subseteq \bigl(I\cap V^{-1}(0)\bigr)\cup\Omega.
\]
The no gap condition on $V$ and the countability of $\Omega$ now imply $\bigabs{\bigcap_{\epsilon>0}\Ee}=0$.
\end{proof}

\subsubsection*{Integrals}

To justify Proposition \ref{prop:intcossinest} we start by considering the cancellations over each period of $\cos$ or $\sin$ (Lemma \ref{lem:piint}) and deal with any incomplete periods (Lemma \ref{lem:endint}). In both cases we work on an interval $J=[a_0,a_1]$ where $\nabla\theta(x)>0$, so $\theta$ is invertible. Making the substitution $u=2\theta(x)$ we get
\[
\dr u=2\nabla\theta(x)\,\dr x=2[\gamma V(x)+k\cos(2\theta(x))]\dr x,
\]
so
\begin{equation}
\label{eq:intx0x1usubs}
\int_{a_0}^{a_1}f(2\theta(x))\dr x
=\int_{2\theta(a_0)}^{2\theta(a_1)}\frac{f(u)}{2[\gamma V(\theta^{-1}(u/2))+k\cos(u)]}\dr u
\end{equation}
for any $f$ (we will take either $f=\sin$ or $f=\cos$).

\begin{lemma}
\label{lem:piint}
Suppose $\theta$ satisfies \eqref{eq:PruferODE} on an interval $J=[a_0,a_1]$ where $V(x)\ge\epsilon$ with $\gamma\epsilon>k$. If $\theta(a_1)=\theta(a_0)+\pi$ then
\[
\abs{\Phi_J}\le\frac{k\abs{J}}{2(\gamma\epsilon-k)}+\frac{\var_J(V)}{\epsilon(\gamma\epsilon-k)}
\quad\text{and}\quad
\abs{\Psi_J}\le\frac{\var_J(V)}{\epsilon(\gamma\epsilon-k)}.
\]
\end{lemma}

Note that we get a better estimate for the $\sin$ integral; the extra term in the estimate for the $\cos$ integral is needed to cope with the fact that this integral is non-zero even when $V$ is constant.

\begin{proof}
Since $\theta(a_1)=\theta(a_0)+\pi$ we can define $\chi:[0,2\pi)\to[2\theta(a_0),2\theta(a_1))$ to be the unique bijection with $\chi(u)-u\in2\pi\Z$ for all $u\in[0,2\pi)$; $\chi$ is piecewise affine with at most one jump ($\chi$ will have no jumps iff $\theta(a_0)\in2\pi\Z$). Using \eqref{eq:intx0x1usubs} we can then write
\[
\Psi_J=\int_{a_0}^{a_1}\sin(2\theta(x))\dr x
=\int_0^{2\pi}\frac{\sin(u)}{2[\gamma V(\theta^{-1}(\chi(u)/2))+k\cos(u)]}\dr u
=\Psi_++\Psi_-
\]
where
\[
\Psi_+=\int_0^{\pi}\frac{\sin(u)}{2[\gamma V(\xi_+(u))+k\cos(u)]}\dr u
\]
with $\xi_+(u)=\theta^{-1}(\chi(u)/2)$, and 
\[
\Psi_-=\int_0^{\pi}\frac{\sin(u+\pi)}{2\bigl[\gamma V\bigl(\theta^{-1}(\chi(u+\pi)/2)\bigr)+k\cos(u+\pi)\bigr]}\dr u
=-\int_0^{\pi}\frac{\sin(u)}{2[\gamma V(\xi_-(u))-k\cos(u)]}\dr u
\]
with $\xi_-(u)=\theta^{-1}(\chi(u+\pi)/2)$; $\xi_+$ and $\xi_-$ are bounded piecewise continuous functions $[0,\pi]\to\R$ with ranges contained in $J=[a_0,a_1]$ (in fact the range of $\xi_\pm$ is just the set of $x\in J$ where $\pm\sin(2\theta(x))\ge0$). 

We now seek bounds on $\Psi_\pm$ using (constant) bounds on $V$. 
Set $m=\inf\{V(x):x\in J\}$ and $M=\sup\{V(x):x\in J\}$ so $\epsilon\le m\le M\le+\infty$ and $m\le V(x)\le M$ for all $x\in J$. Thus
\[
0<\gamma m\pm k\cos(u)\le\gamma V(\xi_\pm(u))\pm k\cos(u)\le\gamma M\pm k\cos(u),\quad u\in[0,\pi],
\]
so
\[
\int_0^\pi\frac{\sin(u)}{2[\gamma M\pm k\cos(u)]}\dr u
\le\pm\Psi_\pm\le\int_0^\pi\frac{\sin(u)}{2[\gamma m\pm k\cos(u)]}\dr u.
\]
The integrals appearing in these bounds can be calculated explicitly, leading to
\[
\frac{1}{2k}\log\left(\frac{\gamma M+k}{\gamma M-k}\right)
\le\pm\Psi_\pm\le\frac{1}{2k}\log\left(\frac{\gamma m+k}{\gamma m-k}\right).
\]
Hence
\begin{align}
\abs{\Psi_J}
&=\abs{\Psi_+-(-\Psi_-)}\nonumber\\
&\le\frac{1}{2k}\left[\log\left(\frac{\gamma m+k}{\gamma m-k}\right)-\log\left(\frac{\gamma M+k}{\gamma M-k}\right)\right]\nonumber\\
&=\frac{1}{2k}\log\left(1+2k\frac{\gamma(M-m)}{(\gamma m-k)(\gamma M+k)}\right)\nonumber\\
\label{eq:branchptest}
&\le\frac{\gamma(M-m)}{(\gamma m-k)(\gamma M+k)}
\end{align}
(note that, $\log(1+t)\le t$ for any $t\ge0$). 
Since $(\gamma m-k)(\gamma M+k)\ge(\gamma\epsilon-k)\gamma\epsilon>0$ and $M-m\le\var_J(V)$ the required estimate for $\Psi_J$ now follows.

\bigskip

In a similar manner we can write
\[
\Phi_J=\int_{a_0}^{a_1}\cos(2\theta(x))\dr x
=\Phi_++\Phi_-
\]
where
\[
\Phi_\pm=\pm\int_{-\pi/2}^{\pi/2}\frac{\cos(u)}{2[\gamma V(\eta_\pm(u))\pm k\cos(u)]}\dr u
\]
for some bounded piecewise continuous functions $\eta_\pm:[-\pi/2,\pi/2]\to\R$ whose ranges are contained in $J=[a_0,a_1]$ (the range of $\eta_\pm$ is just the set of $x\in J$ where $\pm\cos(2\theta(x))\ge0$). For $u\in[-\pi/2,\pi/2]$ we have $0\le\cos(u)$ and
\[
0<\gamma m-k\cos(u)\le\gamma V(\eta_\pm(u))\pm k\cos(u)\le\gamma M+k\cos(u),
\]
so
\[
\int_{-\pi/2}^{\pi/2}\frac{\cos(u)}{2[\gamma M+k\cos(u)]}\dr u
\le\pm\Phi_\pm
\le\int_{-\pi/2}^{\pi/2}\frac{\cos(u)}{2[\gamma m-k\cos(u)]}\dr u.
\]
Now, for any $0\le c\le1$, $(\gamma m-kc)(\gamma M+kc)\ge(\gamma m-k)(\gamma M+k)>0$
(recall that $M\ge m$ while $\gamma m\ge\gamma\epsilon>k$); hence
\[
0\le\frac1{\gamma m-kc}-\frac1{\gamma M+kc}
=\frac{2kc+\gamma(M-m)}{(\gamma m-kc)(\gamma M+kc)}
\le\frac{2kc+\gamma(M-m)}{(\gamma m-k)(\gamma M+k)}.
\]
Putting these estimates together gives
\begin{align}
\abs{\Phi_J}
&=\abs{\Phi_+-(-\Phi_-)}\nonumber\\
&\le\int_{-\pi/2}^{\pi/2}\frac{\cos(u)}{2[\gamma m-k\cos(u)]}\dr u-\int_{-\pi/2}^{\pi/2}\frac{\cos(u)}{2[\gamma M+k\cos(u)]}\dr u\nonumber\\
&\le\frac{k}{(\gamma m-k)(\gamma M+k)}\int_{-\pi/2}^{\pi/2}\cos^2(u)\dr u
  +\frac{\gamma(M-m)}{2(\gamma m-k)(\gamma M+k)}\int_{-\pi/2}^{\pi/2}\cos(u)\dr u\nonumber\\
\label{eq:cos1stest}
&=\frac{k\pi}{2(\gamma m-k)(\gamma M+k)}+\frac{\gamma(M-m)}{(\gamma m-k)(\gamma M+k)}.
\end{align}
The second term can be estimated as for \eqref{eq:branchptest}. On the other hand, \eqref{eq:PruferODE} implies $\abs{\nabla\theta}\le\gamma M+k$ on $J$, so 
\[
\pi=\theta(a_1)-\theta(a_0)\le(\gamma M+k)(a_1-a_0)=(\gamma M+k)\abs{J}.
\]
The required estimate for the first term in \eqref{eq:cos1stest} follows.
\end{proof}

\begin{lemma}
\label{lem:endint}
Suppose $\theta$ satisfies \eqref{eq:PruferODE} on an interval $J=[a_0,a_1]$ where $V(x)\ge\epsilon$ with $\gamma\epsilon>k$. If $\abs{\theta(a_1)-\theta(a_0)}\le\pi$ then
\[
\abs{\Phi_J},\;\lrabs{\Psi_J}\le\frac{2}{\gamma\epsilon-k}.
\]
\end{lemma}

\begin{proof}
For any $u\in[2\theta(a_0),2\theta(a_1)]$ we have 
\[
\gamma V(\theta^{-1}(u/2))+k\cos(u)
\ge\gamma\epsilon-k>0
\]
(note that $\theta^{-1}(u/2)\in J$). For either $f=\sin$ or $f=\cos$ \eqref{eq:intx0x1usubs} now leads to the estimate
\[
\lrabs{\int_{a_0}^{a_1}f(2\theta(x))\dr x}
\le\int_{2\theta(a_0)}^{2\theta(a_1)}\frac{\abs{f(u)}}{2(\gamma\epsilon-k)}\dr u
\le\frac1{2(\gamma\epsilon-k)}\int_{2\theta(a_0)}^{2\theta(a_0)+2\pi}\abs{f(u)}\dr u
=\frac2{\gamma\epsilon-k},
\]
where the middle step follows since $\theta(a_1)\le\theta(a_0)+\pi$.
\end{proof}

The previous Lemmas can be combined in a straightforward way to deal with more general intervals:

\begin{lemma}
\label{lem:oscillsinint}
Suppose $\theta$ satisfies \eqref{eq:PruferODE} on an interval $J=[a,b]$ where $V(x)\ge\epsilon$ with $\gamma\epsilon>k$. Then
\[
\abs{\Phi_J}\le\frac{4+k\abs{J}}{2(\gamma\epsilon-k)}+\frac{\var_J(V)}{\epsilon(\gamma\epsilon-k)}
\quad\text{and}\quad
\abs{\Psi_J}\le\frac{2}{\gamma\epsilon-k}+\frac{\var_J(V)}{\epsilon(\gamma\epsilon-k)}.
\]
\end{lemma}

A similar result holds in the case that $V(x)\le-\epsilon$ for all $x\in J$.

\begin{proof}
Since $\theta$ is absolutely continuous (as is any solution of \eqref{eq:PruferODE} when $V\in L^1_\loc$) and $\nabla\theta(x)\ge\gamma\epsilon-k>0$ we can choose $n\ge0$ and points $a=a_0<a_1<\dots<a_n\le b$ such that $\theta(a_j)=\theta(a_{j-1})+\pi$ for $j=1,\dots,n$ and $\theta(b)<\theta(a_n)+\pi$. Set $J_j=[a_{j-1},a_j]$ for $j=1,\dots,n$ and $\widetilde{J}=[a_n,b]$ so $J_1,\dots,J_n,\widetilde{J}$ have disjoint interiors while $J=J_1\cup\dots\cup J_n\cup\widetilde{J}$. Now
\[
\abs{\Psi_J}
\le\sum_{j=1}^n\abs{\Psi_{J_j}}\;+\abs{\Psi_{\widetilde{J}}}
\le\sum_{j=1}^n\frac{\var_{J_j}(V)}{\epsilon(\gamma\epsilon-k)}\;+\frac{2}{\gamma\epsilon-k}
\]
using Lemmas \ref{lem:piint} and \ref{lem:endint}. The $\Psi_J$ estimate now follows from the fact that $\var_{J_1}(V)+\dots+\var_{J_n}(V)\le\var_{J}(V)$. A similar argument leads to the $\Phi_J$ estimate.
\end{proof}

We can now combine the previous result with Lemma \ref{lem:intervals} to deal with arbitrary sub-intervals of $I=\supp(V)$:

\begin{proposition}
\label{prop:unifbndIest}
Suppose $V\in\BV$ has no gaps. Let $\epsilon>0$ and consider the notation of Lemma \ref{lem:intervals}. If $\theta$ satisfies \eqref{eq:PruferODE} on $I$ and $\gamma\epsilon>k$ then for any sub-interval $J\subseteq I$ we have
\[
\abs{\Phi_J}\le\frac{6\nu\epsilon^{-1}+k\abs{J}}{2(\gamma\epsilon-k)}+\abs{\Ee}
\quad\text{and}\quad
\abs{\Psi_J}
\le\frac{3\nu\epsilon^{-1}}{\gamma\epsilon-k}+\abs{\Ee}.
\]
\end{proposition}

\begin{proof}
Set $\Jen=\Ien\cap J$ for $n\in\Nind$, and $\Nind'=\{n\in\Nind:\Jen\neq\emptyset\}$. In particular
\begin{equation}
\label{eq:Ne'est}
\#\Nind'\le\#\Nind\le\nu\epsilon^{-1}
\end{equation}
by Lemma \ref{lem:intervals}(iii), while
\[
\Fe:=J\setminus\bigcup_{n\in\Nind'}\Jen
\subseteq I\setminus\bigcup_{n\in\Nind}\Ien
=\Ee,
\]
so
\begin{equation}
\label{eq:FeEe}
\abs{\Fe}\le\abs{\Ee}.
\end{equation}
Now
\[
\Psi_J=\sum_{n\in\Nind'}\Psi_{\Jen}\;
+\int_{\Fe}\sin(2\theta(x))\dr x.
\]
Using \eqref{eq:FeEe} the modulus of the final term is clearly bounded by $\abs{\Ee}$. On the other hand, by Lemma \ref{lem:intervals}(i) $V$ has constant sign on $\Jen\subseteq\Ien$ and satisfies $\abs{V(x)}\ge\epsilon$ for all $x\in\Jen$. Lemma \ref{lem:oscillsinint} thus gives
\[
\abs{\Psi_{\Jen}}
\le\frac{2+\epsilon^{-1}\var_{\Jen}(V)}{\gamma\epsilon-k}.
\]
However the $\Jen$ for $n\in\Nind'$ are disjoint subintervals of $J\subseteq I$, so
\[
\sum_{n\in\Nind'}\var_{\Jen}(V)
\le\var_J(V)
\le\var_{\R}(V)
=\nu.
\]
Together with \eqref{eq:Ne'est} we then get
\[
\sum_{n\in\Nind'}\bigl(2+\epsilon^{-1}\var_{\Jen}(V)\bigr)
\le2\,\#\Nind'+\nu\epsilon^{-1}
\le3\nu\epsilon^{-1}.
\]
The estimate for $\Psi_J$ follows. Since 
\[
\sum_{n\in\Nind'}\abs{\Jen}
=\biggl\lvert{\bigcup_{n\in\Nind'}\Jen}\biggr\rvert
\le\abs{J}
\]
a similar argument leads to the estimate for $\Phi_J$.
\end{proof}

\begin{proof}[Proof of Proposition \ref{prop:intcossinest}]
Choose $\epsilon=\epsilon_\gamma>0$ for each $\gamma>0$ so that $\epsilon_\gamma\to0$ and $\epsilon_\gamma^2\gamma\to\infty$ as $\gamma\to\infty$ (we can take $\epsilon_\gamma\sim\gamma^{-\mu}$ with $0<\mu<1/2$, for example). Letting $\gamma\to\infty$ it follows that $\nu\epsilon_\gamma^{-1}(\gamma\epsilon_\gamma-k)^{-1}\to0$ and $\abs{\Ee[\epsilon_\gamma]}\to0$ (from Lemma \ref{lem:intervals}(ii)); both limits are independent of $J$. On the other hand, $\abs{J}\le\abs{I}$ (since $J\subseteq I$) so we also have $\abs{J}(\gamma\epsilon_\gamma-k)^{-1}\to0$ at a rate that can be bounded uniformly in $J$. The result now follows directly from the estimates in Proposition \ref{prop:unifbndIest}.
\end{proof}

\section{Real zeros of a perturbed trigonometric function}
\label{sec:rzcos}

This section is devoted to the proof of Theorem \ref{thm:coszeros}. For notational convenience define a function $\f:\R\to\R$ by
\[
\f(x)=\cos(x)+\alpha\cos(\beta x).
\]
For any function $\phi:\R\to\R$ we also set $\f[\phi]=\f+\phi$; thus \eqref{eq:cospert0} becomes $\f[\phi](x)=0$. 

The case $\alpha\beta<1$ is straightforward; an elementary argument shows that $\f[\phi]$ is a ``small'' perturbation of $\cos(x)$ leading to a one-to-one association between points in the corresponding zero sets (Lemma \ref{lem:zintab<1}). 

The case $\alpha\beta>1$ is dealt with in two main steps; firstly the result is obtained directly when $\phi\equiv0$ (Section \ref{sec:zf}) and secondly we show that the addition of $\phi$ cannot change the leading order asymptotic of the number of zeros (Section \ref{sec:zpert}). In both steps the arguments for $\beta\in\Q$ and $\beta\notin\Q$ are unified only in the initial stages.

When $\beta\in\Q$ $\f$ is periodic and an exact count of the number of zeros can be made. 
In order to deal with the perturbation $\phi$ we simply need to avoid cases where $\f$ has tangential zeros; this leads to the extra condition $p_\beta+q_\beta\nu_{\alpha,\beta}\notin4\Z$ in this case (see Lemma \ref{lem:tangzsbrat}). Since we are then only dealing with the perturbation of transversal zeros we do not require $n=2$ in condition \eqref{eq:decayprop012}; we will establish the result using the weaker decay condition
\begin{equation}
\label{eq:decayprop01}
\phi\in C^1(\R),\quad\text{$\phi^{(n)}(x)=o(1)$ as $x\to\infty$ for $n=0,1$.}
\end{equation}

When $\beta\notin\Q$ $\f$ is no longer periodic but we can appeal to ergodicity to determine the asymptotic distribution of zeros. However dealing with the perturbation $\phi$ is more subtle in this case as we must consider points where $\f$ comes arbitrarily close to having a tangential zero. The number of such points is limited (Corollary \ref{cor:diststp}) while condition \eqref{eq:decayprop012} ensures that the addition of $\phi$ will alter the number of zeros by at most $2$ near each such point (Corollary \ref{cor:dzIstp}).

\subsection{Some preliminaries}
\label{sec:zprelim}

Let $\numz$ denote the counting function given by
\[
\numz(I)=\#\bigl\{x\in I:\f[\phi](x)=0\bigr\}\in\N\cup\{0,\infty\}
\]
for any interval $I\subseteq\R$; if $R\ge0$ we'll abuse notation slightly and write $\numz(R)$ for $\numz([0,R])$.
We need to determine the limit of $\numz(R)/R$ as $R\to\infty$. 
To do this it will be convenient to work with sequences of increasing values of $R$; 
the next result is straightforward (note that $\numz[\phi](R)$ is an non-decreasing function of $R$).

\begin{lemma}
\label{lem:numzRRseqlim}
The quantity $\numz[\phi](R)/R$ has a limit as $R\to\infty$ iff the quantity $\numz[\phi](R_n)/R_n$ has a limit as $n\to\infty$ for some (equivalently, any) positive increasing sequence $(R_n)_{n\ge1}$ with $R_n\to\infty$ and $R_{n-1}/R_n\to1$ as $n\to\infty$. When the limits exist they are equal.
\end{lemma}

\subsubsection*{Perturbation of zeros}

For any function $h\in C^1(\R)$ set $\Eng{h}=(h)^2+(h')^2$. 

\begin{lemma}
\label{lem:genpertz1}
Suppose $g,\psi\in C^1(\R)$ satisfy $\Eng{\psi}<\Eng{g}$ on an interval $I$ where $g'$ doesn't change sign (that is, either $g'\ge0$ or $g'\le0$). Then $g+\psi$ can have at most one zero on $I$.
\end{lemma}

\begin{proof}
Assume $g'\ge0$ (the case $g'\le0$ can be handled similarly). 
Now suppose $g+\psi$ has (at least) two zeros on $I$. 
Then (at least) one of these zeros, say $x_0$, satisfies $(g+\psi)'(x_0)\le0$, so
\[
\psi(x_0)=-g(x_0)
\quad\text{and}\quad
\psi'(x_0)\le-g'(x_0)\le0.
\]
This leads to the contradiction $\Eng{\psi}(x_0)\ge\Eng{g}(x_0)$.
\end{proof}

\begin{lemma}
\label{lem:genpertz2}
Let $g,\psi\in C^1(\R)$ and suppose $I=[s,t]$ is a closed interval with $g'(s)=0=g'(t)$ and $g'(x)\neq0$ for $x\in(s,t)$. Also suppose $\Eng{\psi}<\Eng{g}$ on $I$. Then $g+\psi$ and $g$ have the same number of zeros on $I$. Furthermore the endpoints of $I$ can't be zeros of either function.
\end{lemma}

\begin{proof}
Since $g'(x)\neq0$ for $x\in(s,t)$, $g$ can have at most one zero on $I$. Also $g'\ge0$ or $g'\le0$ on $I$ (by continuity), so $g+\psi$ has at most one zero on $I$ by Lemma \ref{lem:genpertz1}.

We have $0\le\psi^2(s)\le\Eng{\psi}(s)<\Eng{g}(s)=g^2(s)$. Thus $g$ and $g+\psi$ are non-zero and have the same sign at $s$. A similar result applies at $t$.

If $g(s)g(t)<0$ then the same is true for $g+\psi$. In this case both functions have at least one, and hence exactly one, zero on $I$. 

If $g(s)g(t)>0$ (equivalently $(g+\psi)(s)\,(g+\psi)(t)>0$) then $g$ (respectively $g+\psi$) either has no zeros on $I$ or a single zero $x_0\in I$ which is also a turning point. It follows that $\Eng{g}(x_0)=0$ (respectively $\psi(x_0)=-g(x_0),\,\psi'(x_0)=-g'(x_0)$) which leads to a contradiction since $\Eng{g}>0$ on $I$ (respectively $\Eng{\psi}(x_0)\neq\Eng{g}(x_0)$).
Hence neither $g$ nor $g+\psi$ have any zeros on $I$ when $g(s)g(t)>0$.
\end{proof}

\subsubsection*{The case $\alpha\beta<1$}

The case $\alpha\beta<1$ in Theorem \ref{thm:coszeros} follows easily from the next result.

\begin{lemma}
\label{lem:zintab<1}
Suppose $0\le\alpha,\alpha\beta<1$ and $\phi$ satisfies \eqref{eq:decayprop01}. Then there exists $N\in\N$ such that $\f[\phi]$ has exactly one zero in $[n\pi,(n+1)\pi]$ for any $n\ge N$; furthermore this zero cannot occur at an endpoint of the interval.
\end{lemma}

\begin{proof}
Set $\psi(x)=\alpha\cos(\beta x)+\phi(x)$ and $\epsilon=\frac12\min\{1-\alpha^2,1-\alpha^2\beta^2\}>0$. Now
\begin{align*}
\Eng{\psi}(x)
&=\alpha^2\cos^2(\beta x)+\alpha^2\beta^2\sin^2(\beta x)
+2\alpha\cos(\beta x)\phi(x)+2\alpha\beta\sin(\beta x)\phi'(x)+\Eng{\phi}(x)\\
&\le\alpha^2\max\{1,\beta^2\}+2\alpha\abs{\phi(x)}+2\alpha\beta\abs{\phi'(x)}+\Eng{\phi}(x)\\
&=1-2\epsilon+2\alpha\abs{\phi(x)}+2\alpha\beta\abs{\phi'(x)}+\Eng{\phi}(x).
\end{align*}
Our assumptions on $\phi$ then allow us to find $N\in\N$ such that 
$\Eng{\psi}(x)\le1-\epsilon$ for all $x\ge N\pi$. 
If $n\ge N$ then Lemma \ref{lem:genpertz2} (with $g=\cos$ so $\Eng{g}=1$) shows that $\cos$ and $\cos+\psi=\f[\phi]$ have the same number of zeros in $[n\pi,(n+1)\pi]$ and any zeros lie in the interior. 
The result follows.
\end{proof}

\subsection{The unperturbed function}
\label{sec:zf}

Throughout this section we shall assume $\alpha\beta>1$ (although several of the results can be extended to cover other cases). 
Since $\alpha<1$ it follows that $\beta>1$.
Define $\xi,\eta\in(0,\pi/2)$ by
\begin{equation}
\label{eq:alphabetadef}
\xi=\arcsin\frac{\sqrt{\alpha^2\beta^2-1}}{\sqrt{\beta^2-1}}
\quad\text{and}\quad
\eta=\arcsin\frac{\sqrt{1-\alpha^2}}{\alpha\sqrt{\beta^2-1}}.
\end{equation}
The complementary angles satisfy
\begin{equation}
\label{eq:alpha'beta'def}
\xi'=\frac\pi2-\xi=\arcsin\frac{\beta\sqrt{1-\alpha^2}}{\sqrt{\beta^2-1}}
\quad\text{and}\quad
\eta'=\frac\pi2-\eta=\arcsin\frac{\sqrt{\alpha^2\beta^2-1}}{a\sqrt{\beta^2-1}}.
\end{equation}
If we fix $\beta>1$ and vary $\alpha$ from $1/\beta$ to $1$ it is easy to check that $\xi$ increases from $0$ to $\pi/2$ and $\eta$ decreases from $\pi/2$ to $0$. Also note that $\nu_{\alpha,\beta}=\dfrac2{\pi}\bigl(\beta\xi+\eta\bigr)$ (recall \eqref{eq:defncoeffu}). 

\subsubsection*{Zeros of $f$}

The aim of this section is to establish Theorem \ref{thm:coszeros} in the case that $\phi\equiv0$.
For any $n\in\Z$ set $\numzj=\numz[0](2\pi n+[0,2\pi))$; in particular
\begin{equation}
\label{eq:numz0Jnumzj}
\numz[0](2\pi N)=\sum_{n=0}^{N-1}\numzj
\end{equation}
for any $N\in\N$. For any $t\in\R$ let 
\[
\numt(t)=\#\bigl\{x\in[0,\pi):\cos(x)+\alpha\cos(\beta x+t)=0\bigr\}.
\]
Clearly $\numt(t)$ is $2\pi$-periodic in $t$ while
\[
\cos 0+\alpha\cos(\beta.0+t)\ge 1-\alpha>0
\quad\text{and}\quad\cos\pi+\alpha\cos(\beta\pi+t)\le-1+\alpha<0
\]
so the inclusion/exclusion of the possibility $x=0$ or $x=\pi$ does not alter the definition of $\numt$. 

\begin{lemma}
\label{lem:numzjnumt}
For all $n\in\Z$ we have
\[
\numzj=\numt(2\pi n\beta)+\numt(-2\pi(n+1)\beta).
\]
\end{lemma}

\begin{proof}
Let $x_1=x-2\pi n$, $x_2=2\pi-x_1$ (so $\cos(x)=\cos(x_1)=\cos(x_2)$). Then
\begin{align}
\f(x)=0
\label{eq:x1cond}\quad
&\Longleftrightarrow\quad\cos(x_1)+\alpha\cos(\beta x_1+2\pi n\beta)=0\\
\label{eq:x2cond}
&\Longleftrightarrow\quad\cos(x_2)+\alpha\cos(\beta x_2-2\pi(n+1)\beta)=0.
\end{align}
Clearly $x_2\in(0,\pi)$ iff $x_1\in(\pi,2\pi)$. Thus
\begin{align*}
\numzj&=\#\bigl\{x_1\in(0,\pi):\text{\eqref{eq:x1cond} holds}\bigr\}+\#\bigl\{x_2\in(0,\pi):\text{\eqref{eq:x2cond} holds}\bigr\}\\
&=\numt(2\pi n\beta)+\numt(-2\pi(n+1)\beta)
\end{align*}
(from the definition of $\numt$). 
\end{proof}

Set $\mu=\beta\xi-\eta'$ (recall \eqref{eq:alphabetadef} and \eqref{eq:alpha'beta'def}) and define an open interval by
\[
J=-\beta\,\frac{\pi}2+\frac{3\pi}2+(-\mu,\mu);
\]
in particular
\begin{equation}
\label{eq:uv|J|}
\nu_{\alpha,\beta}=1+\dfrac{2}{\pi}\,\mu
\quad\text{and}\quad
\abs{J}=2\mu.
\end{equation}
For any open interval $I$ set 
\[
\mcf{I}=\frac12\bigl(\chi_{I}+\chi_{\o{I}}\bigr)
\]
(so $\mcf{I}$ is $1$ on $I$, takes the value $1/2$ at the end points of $I$, and is $0$ elsewhere).

\begin{lemma}
\label{lem:altexpnumt}
For any $t$ we have
\[
\numt(t)=1+2\sum_{n\in\Z}\mcf{J}(t-2\pi n).
\]
\end{lemma}

\begin{remark}
If $\alpha\beta<1$ a simplified version of the following argument can be used to show $\numt(t)=1$ for all $t$. 
\end{remark}

\begin{proof}[Proof of Lemma \ref{lem:altexpnumt}]
Set $\zeta=\arcsin(\alpha)\in\Bigl(0,\dfrac{\pi}2\Bigr)$ and $K=\dfrac{\pi}{2}+(-\zeta,\zeta)$.
Now suppose 
\begin{equation}
\label{eq:cxct2}
\cos(\beta x+t)=-\frac1{\alpha}\,\cos(x).
\end{equation}
for some $x\in[0,\pi)$ and $t\in\R$. Firstly observe that $\abs{\cos(x)}\le\alpha$ so $x\in\o{K}$. 
Define $\theta_-:\o{K}\to\R$ and $\theta_+:K\to\R$ by
\[
\theta_-(x)=\arccos\Bigl(-\frac1{\alpha}\,\cos(x)\Bigr)
\quad\text{and}\quad
\theta_+(x)=2\pi-\theta_-(x).
\]
[These functions correspond to the two basic branches of the inverse of $\cos$; the remaining branches can be obtained by adding multiples of $2\pi$.] 
Also define $\omega_-:\o{K}\to\R$ and $\omega_+:K\to\R$ by $\omega_{s}(x)=-\beta x+\theta_{s}(x)$ for $s\in\{-,+\}$. 

Now \eqref{eq:cxct2} is equivalent to the existence of unique $s\in\{-,+\}$ and $n\in\Z$ such that $\beta x+t=\theta_{s}(x)+2\pi n$. In turn, this is equivalent to
\begin{equation}
\label{eq:cxct3}
\text{there exists unique $s\in\{-,+\}$, $n\in\Z$ with $t-2\pi n=\omega_s(x)$}
\end{equation}
(note that, if $x=\pi/2\pm\zeta$ the we must take $s=-$). We can determine $\numt(t)$ using \eqref{eq:cxct3} if we know the ranges of $\omega_-$ and $\omega_+$, together with the multiplicity of covering. 

\smallskip

\noindent
\emph{Range of $\omega_-$:}
The function $\theta_-$ is monotonically decreasing while $\beta>0$ so $\omega_-$ is also monotonically decreasing on $\o{K}$. Thus $\Ran\omega_-=I_-$ where
\[
I_-=\Bigl[\omega_-\Bigl(\frac{\pi}2+\zeta\Bigr),\,\omega_-\Bigl(\frac{\pi}2-\zeta\Bigr)\Bigr]
=\Bigl[-\beta\Bigl(\frac{\pi}2+\zeta\Bigr),\,-\beta\Bigl(\frac{\pi}2-\zeta\Bigr)+\pi\Bigr].
\]
The multiplicity of covering is 1.

\smallskip

\noindent
\emph{Range of $\omega_+$:}
The turning points of $\omega_+$ (on $K$) satisfy
\[
\theta_+'(x)=\beta
\quad\Longleftrightarrow\quad
\frac{\sin x}{\sqrt{\alpha^2-\cos^2(x)}}=\beta
\quad\Longleftrightarrow\quad
\cos^2(x)=\frac{\alpha^2\beta^2-1}{\beta^2-1}.
\]
This gives precisely two turning points, at $x_\pm=\dfrac{\pi}2\pm\xi$. 
Furthermore $\omega_+$ is monotonically increasing on $\Bigl(\dfrac{\pi}2-\zeta,\,x_-\Bigr)$ and $\Bigl(x_+,\,\dfrac{\pi}2+\zeta\Bigr)$, and monotonically decreasing on $[x_-,x_+]$. 
Now
\[
\theta_+(x_\pm)=2\pi-\arccos\Bigl(-\frac1{\alpha}\,\cos(x_\pm)\Bigr)
=2\pi-\frac{\pi}2+\arcsin\left(\pm\frac{\sqrt{\alpha^2\beta^2-1}}{\alpha\sqrt{\beta^2-1}}\right)
=\frac{3\pi}2\pm\eta',
\]
so 
\[
\omega_+(x_\pm)=-\beta\,\frac{\pi}2\mp\beta\xi+\frac{3\pi}2\pm\eta'=-\beta\,\frac{\pi}2+\frac{3\pi}2\mp\mu.
\]
Hence 
\begin{align*}
\Ran\omega_+&=\Bigl(\omega_+\Bigl(\frac\pi2-\zeta\Bigr),\,\omega_+(x_-)\Bigr)
\cup\Bigl(\omega_+(x_+),\,\omega_+\Bigl(\frac\pi2+\zeta\Bigr)\Bigr)\cup[\omega_+(x_+),\,\omega_+(x_-)]\\
&=I_+^-\cup I_+^+\cup\o{J},
\end{align*}
where
\[
I_+^-=\Bigl(-\beta\Bigl(\frac\pi2-\zeta\Bigr)+\pi,\,-\beta\,\frac\pi2+\frac{3\pi}2+\mu\Bigr)
\quad\text{and}\quad
I_+^+=\Bigl(-\beta\,\frac\pi2+\frac{3\pi}2-\mu,\,-\beta\Bigl(\frac\pi2+\zeta\Bigr)+2\pi\Bigr).
\]
Each interval has a multiplicity of 1.

\smallskip

To complete the proof note that the definition of $\numt$ and \eqref{eq:cxct3} give
\[
\numt(t)=\sum_{n\in\Z}\bigl(\chi_{I_-}+\chi_{I_+^-}+\chi_{I_+^+}+\chi_{\o{J}}\bigr)(t-2\pi n).
\]
Now $I_-$ and $I_+^-$ are disjoint sets with union
\[
\widetilde{I}=\Bigl[-\beta\Bigl(\frac\pi2+\zeta\Bigr),\,-\beta\,\frac\pi2+\frac{3\pi}2+\mu\Bigr). 
\]
It follows that
\[
\chi_{I_-}+\chi_{I_+^-}+\chi_{I_+^+}
=\chi_{\widetilde{I}}+\chi_{I_+^+}
=\chi_{-\beta(\pi/2+\zeta)+[0,2\pi)}+\chi_{J},
\]
where the last step uses the identity $\chi_{[p,r)}+\chi_{(q,s)}=\chi_{[p,s)}+\chi_{(q,r)}$ which holds whenever $p,q<r,s$. The result follows.
\end{proof}

Ergodicity can now be used in the case $\beta\notin\Q$. A convenient ergodic theorem gives us
\begin{equation}
\label{eq:Weylequidist}
\lim_{N\to\infty}\frac1{N}\,\sum_{n=0}^{N-1}g(2\pi n\beta)=\frac1{2\pi}\int_0^{2\pi}g(x)\dr x
\end{equation}
whenever $\beta\notin\Q$ and $g:\R\to\R$ is a $2\pi$-periodic Riemann integrable function 
(this is a version of Weyl equidistribution; see \cite{StSh} for example).

\begin{proof}[Proof of Theorem \ref{thm:coszeros} when $\phi\equiv0$, $\alpha\beta>1$ and $\beta\notin\Q$]
Combine \eqref{eq:numz0Jnumzj} with Lemmas \ref{lem:numzjnumt} and \ref{lem:numzRRseqlim} to get
\[
2\pi\lim_{R\to\infty}\frac{\numz[0](R)}{R}
=\lim_{N\to\infty}\frac1{N}\sum_{n=0}^{N-1}\numzj
=\lim_{N\to\infty}\frac1{N}\sum_{n=0}^{N-1}\bigl(\numt(2\pi n\beta)+\numt(-2\pi(n+1)\beta)\bigr).
\]
Since $\beta\notin\Q$ and $\numt$ is a $2\pi$-periodic piecewise constant function \eqref{eq:Weylequidist} and Lemma \ref{lem:altexpnumt} then give
\begin{align*}
&2\pi\lim_{R\to\infty}\frac{\numz[0](R)}{R}
=\frac1{2\pi}\int_0^{2\pi}\bigl(\numt(x)+\numt(-x-2\pi\beta)\bigr)\dr x
=\frac1\pi\int_0^{2\pi}\numt(x)\dr x\\
&\qquad{}=\frac1\pi\Biggl(2\pi+2\sum_{n\in\Z}\int_0^{2\pi}\mcf{J}(x-2\pi n)\dr x\Biggr)
=\frac1\pi\left(2\pi+2\int_{\R}\mcf{J}\dr x\right)
=\frac1\pi(2\pi+2\abs{J}).
\end{align*}
The result now follows from \eqref{eq:uv|J|}.
\end{proof}

\medskip

Now suppose $\beta\in\Q$. Write $\beta=p/q$ where $p,q\in\N$ are coprime.

\begin{lemma}
\label{lem:numzRRseqlimbrat}
We have
\[
\lim_{R\to\infty}\frac{\numz[0](R)}{R}
=\frac1\pi\Biggl(1+\frac2{q}\sum_{n\in\Z}\mcf{J}(2\pi n/q)\Biggr).
\]
\end{lemma}

\begin{proof}
For $N\in\N$ Lemma \ref{lem:numzjnumt} gives
\begin{align*}
\sum_{n=0}^{qN-1}\numzj
&=\sum_{j=0}^{N-1}\sum_{k=0}^{q-1}\left[\numt\Bigl(2\pi(k+qj)\,\frac{p}{q}\Bigr)+\numt\Bigl(-2\pi(k+qj+1)\,\frac{p}{q}\Bigr)\right]\\
&=N\sum_{k=0}^{q-1}\left[\numt\Bigl(2\pi k\,\frac{p}{q}\Bigr)+\numt\Bigl(-2\pi(k+1)\,\frac{p}{q}\Bigr)\right]
\end{align*}
since $\numt$ is $2\pi$-periodic. 
Now the mappings
\[
k\mapsto kp\mod q
\quad\text{and}\quad
k\mapsto -(k+1)p\mod q
\]
give bijections on $\{0,1,\dots,q-1\}$ (since $p$ and $q$ are coprime). 
Together with Lemma \ref{lem:numzRRseqlim} and \eqref{eq:numz0Jnumzj} we then get
\[
\lim_{R\to\infty}\frac{\numz[0](R)}{R}
=\frac1{2\pi}\lim_{N\to\infty}\frac{1}{qN}\sum_{n=0}^{qN-1}\numzj
=\frac1{\pi q}\sum_{k=0}^{q-1}\numt(2\pi k/q).
\]
On the other hand, Lemma \ref{lem:altexpnumt} gives
\[
\sum_{k=0}^{q-1}\numt(2\pi k/q)
=\sum_{k=0}^{q-1}\Biggl(1+2\sum_{l\in\Z}\mcf{J}\Bigl(2\pi(k-lq)\,\frac{1}{q}\Bigr)\Biggr)
=q+2\sum_{n\in\Z}\mcf{J}(2\pi n/q),
\]
with $n=k-lq$.
\end{proof}

\begin{proof}[Proof of Theorem \ref{thm:coszeros} when $\phi\equiv0$, $\alpha\beta>1$ and $\beta\in\Q$]
Let $n\in\Z$. Then
\begin{align}
\mcf{J}(2\pi n/q)=1
\quad&\Longleftrightarrow\quad \frac{2\pi n}{q}\in J\nonumber\\
\quad&\Longleftrightarrow\quad -\frac{\beta}4+\frac34-\frac14\,\frac{2}\pi\,\mu<\frac{n}{q}
<-\frac{\beta}4+\frac34+\frac14\,\frac{2}\pi\,\mu\nonumber\\
\label{eq:equiv2pijkinJ}
\quad&\Longleftrightarrow\quad -\frac14\,(p+q\nu_{\alpha,\beta})+q<n<\frac14\,(p+q\nu_{\alpha,\beta})-\frac{p+q}{2}+q.
\end{align}
We get $\mcf{J}(2\pi n/q)=1/2$ iff $n$ is equal to one of the endpoints, and $\mcf{J}(2\pi n/q)=0$ iff $n$ lies beyond the given range. 
To ensure we miss the endpoints we require $p+q\nu_{\alpha,\beta}\notin4\Z$ (left endpoint) and $(p+q\nu_{\alpha,\beta})-2(p+q)\notin4\Z$ (right endpoint). 
If $p+q$ is even these conditions are equivalent; otherwise they combine as the requirement $p+q\nu_{\alpha,\beta}\notin2\Z$. We will now assume this condition is satisfied. 
From \eqref{eq:equiv2pijkinJ} we then get
\begin{align*}
N:=\sum_{n\in\Z}\mcf{J}(2\pi n/q)
&=\left\lfloor\frac14\,(p+q\nu_{\alpha,\beta})-\frac{p+q}{2}+q\right\rfloor-\left\lfloor-\frac14\,(p+q\nu_{\alpha,\beta})+q\right\rfloor\\
&=\left\lfloor\frac14\,(p+q\nu_{\alpha,\beta})-\frac{p+q}{2}\right\rfloor+\left\lfloor\frac14\,(p+q\nu_{\alpha,\beta})\right\rfloor+1
\end{align*}
(since $-\lfloor-x\rfloor=\lfloor x\rfloor+1$ for any $x\notin\Z$).

\smallskip

\noindent
\emph{Case $p$, $q$ are both odd.}
Then $(p+q)/2\in\Z$ so
\[
N=2\left\lfloor\frac14\,(p+q\nu_{\alpha,\beta})\right\rfloor-\frac{p+q}{2}+1.
\]
Lemma \ref{lem:numzRRseqlimbrat} now gives
\[
\lim_{R\to\infty}\frac{\numz[0](R)}{R}
=\frac1\pi\Bigl(1+\frac2{q}\,N\Bigr)
=\frac{1}{\pi}\left(\frac4{q}\left\lfloor\frac14(p+q\nu_{\alpha,\beta})\right\rfloor-\frac{p}{q}+\frac2{q}\right);
\]
the right hand side is just $A(\alpha,\beta)/\pi$ since $p_\beta=p$, $q_\beta=q$ in this case.

\smallskip

\noindent
\emph{Case $p$, $q$ have opposite parity.}
Then $(p+q-1)/2\in\Z$ so 
\[
N=\left\lfloor\frac14\,(p+q\nu_{\alpha,\beta})-\frac12\right\rfloor-\frac{p+q-1}{2}+\left\lfloor\frac14\,(p+q\nu_{\alpha,\beta})\right\rfloor+1
=\left\lfloor\frac12\,(p+q\nu_{\alpha,\beta})\right\rfloor-\frac{p+q}{2}+\frac12
\]
(since $\lfloor x-1/2\rfloor+\lfloor x\rfloor=\lfloor 2x\rfloor-1$ whenever $2x\notin\Z$).
Lemma \ref{lem:numzRRseqlimbrat} now gives
\[
\lim_{R\to\infty}\frac{\numz[0](R)}{R}
=\frac1\pi\Bigl(1+\frac2{q}\,N\Bigr)
=\frac{1}{\pi}\left(\frac2{q}\left\lfloor\frac12(p+q\nu_{\alpha,\beta})\right\rfloor-\frac{p}{q}+\frac1{q}\right);
\]
since $p_\beta=2p$, $q_\beta=2q$ in this case, the right hand side becomes $A(\alpha,\beta)/\pi$  (note that, the condition $p+q\nu_{\alpha,\beta}\notin2\Z$ becomes $p_\beta+q_\beta\nu_{\alpha,\beta}\notin4\Z$).
\end{proof}

\subsubsection*{Turning points}

Consider the set of non-negative turning points of $\f$,
\[
\tp=\{x\ge0:\f'(x)=0\}.
\]
Since $f$ is an analytic function $\tp$ is a discrete subset of $\R$. 
List the points of $\tp$ in increasing order as $t_0=0,t_1,t_2,\dots$ (note that $\f'(0)=0$). 
For each $n\ge1$ set $\Itp{n}=[t_{n-1},t_n]$. 

\begin{lemma}
\label{lem:basictplem}
We have $t_n\to\infty$ as $n\to\infty$ and $t_n-t_{n-1}\le2\pi/\beta$ for all $n\in\N$. 
It follows that $t_n/t_{n-1}\to1$ as $n\to\infty$.
\end{lemma}

\begin{remark}
When $\alpha\beta\le1$ the same result holds with $2\pi$ in place of $2\pi/\beta$.
\end{remark}

\begin{proof}[Proof of Lemma \ref{lem:basictplem}]
Since $\f'(x)=-\sin(x)-\alpha\beta\sin(\beta x)$ we have $(-1)^nf'(x)<0$ when $x=(n+1/2)\pi/\beta$. 
Thus $\tp$ contains at least one point in the interval $\bigl((n-1/2)\pi/\beta,\,(n+1/2)\pi/\beta\bigr)$ for any $n\in\Z$; this forces $t_n\to\infty$ 
and $t_n-t_{n-1}\le2\pi/\beta$.
\end{proof}

\subsubsection*{Bound on $\f''$}

Let $x\in\R$ and set $u_n=\f^{(n)}(x)$ for $n=0,1,2$. Then $\abs{u_n}\le1+\alpha\beta^n$ while
\begin{align*}
\cos(x)&=-\alpha\cos(\beta x)+u_0,\\
\sin(x)&=-\alpha\beta\sin(\beta x)+u_1,\\
\cos(x)&=-\alpha\beta^2\cos(\beta x)+u_2.
\end{align*}
Squaring and rearranging each equation leads to
\begin{subequations}
\begin{align}
\label{eq:sin2eq0}
\sin^2(x)-\alpha^2\sin^2(\beta x)&=1-\alpha^2+v_0,&&v_0=2u_0\alpha\cos(\beta x)-u_0^2,\\
\label{eq:sin2eq1}
\sin^2(x)-\alpha^2\beta^2\sin^2(\beta x)&=v_1,&&v_1=-2u_1\alpha\beta\sin(\beta x)+u_1^2,\\
\label{eq:sin2eq2}
\sin^2(x)-\alpha^2\beta^4\sin^2(\beta x)&=1-\alpha^2\beta^4+v_2,&&v_2=2u_2\alpha\beta^2\cos(\beta x)-u_2^2.
\end{align}
\end{subequations}
In particular,
\begin{equation}
\label{eq:upbndetak}
\abs{v_n}\le2\alpha\beta^n\abs{u_n}+\abs{u_n}^2\le(1+3\alpha\beta^n)\abs{u_n}.
\end{equation}
Solving \eqref{eq:sin2eq0} and \eqref{eq:sin2eq1} as linear equations for $\sin^2(x)$ and $\sin^2(\beta x)$ leads to
\begin{equation}
\label{eq:sin2matsol}
\begin{pmatrix}
\sin^2(x)\\\sin^2(\beta x)
\end{pmatrix}
=\frac{1-\alpha^2}{\alpha^2(\beta^2-1)}\begin{pmatrix}
\alpha^2\beta^2\\1
\end{pmatrix}
+\frac1{\alpha^2(\beta^2-1)}\begin{pmatrix}
\alpha^2\beta^2v_0+\alpha^2v_1\\-\alpha^2v_0-v_1
\end{pmatrix}
\end{equation}
(recall that $\beta>1$). 
Using \eqref{eq:sin2eq2} then gives
\[
v_2
=(1+\beta^2)(\alpha^2\beta^2-1)+\frac{\beta^2(1+\alpha^2\beta^2)}{\beta^2-1}\,v_0+\frac{1+\beta^4}{\beta^2-1}\,v_1.
\]
The estimates \eqref{eq:upbndetak} now imply $\abs{u_2}\ge c_0-c_{1,0}\abs{u_0}-c_{1,1}\abs{u_1}$ where
\[
c_0=\frac{(1+\beta^2)(\alpha^2\beta^2-1)}{1+3\alpha^2\beta^4},
\ \ c_{1,0}=\frac{\beta^2(1+\alpha^2\beta^2)(1+3\alpha^2)}{(\beta^2-1)(1+3\alpha^2\beta^4)}
\ \ \text{and}\ \ c_{1,1}=\frac{(1+\beta^4)(1+3\alpha^2\beta^2)}{(\beta^2-1)(1+3\alpha^2\beta^4)}
\]
are positive constants. Taking $c_1=\sqrt2\max\{c_{1,0},c_{1,1}\}$ we have thus established the following. 

\begin{lemma}
\label{lem:lowbndf''}
There exist positive constants $c_0$ and $c_1$ (depending only on $\alpha$ and $\beta$) such that $\abs{\f''(x)}\ge c_0-c_1\sqrt{\Eng{\f}(x)}$ for all $x\in\R$.
\end{lemma}

It follows that we can find $\kappa>0$ so that 
\begin{equation}
\label{eq:defpropkappa}
\text{if $\Eng{\f}(x)\le\kappa^2$ for some $x\in\R$ then $\abs{\f''(x)}\ge\kappa$} 
\end{equation}
(we can choose $\kappa$ to be anything in $(0,c_0/(1+c_1)]$). Now set 
\[
\sEng=\bigl\{x\in\R^+:\Eng{f}(x)<\kappa^2\bigr\}.
\]
Then $\sEng$ is open (as $\Eng{f}$ is continuous) and $\abs{\f''(x)}\ge\kappa$ for all $x\in\sEng$ 
(by \eqref{eq:defpropkappa}). Further useful properties are as follows.

\begin{lemma}
\label{lem:propssEng}
Let $I$ be a maximal connected component of $\sEng$.
\begin{itemize}
\item[(i)]
If $\Eng{\f}'(x)=0$ for some $x\in\sEng$ then $x\in\tp$.
\item[(ii)]
$I$ contains a unique point of $\tp$.
\item[(iii)]
If $\phi\in C^2(\R)$ satisfies $\abs{\phi''(x)}<\kappa$ on $I$ 
then $\f[\phi]$ can have at most two zeros on $I$.
\end{itemize}
\end{lemma}

\begin{proof}
For part (i) let $x\in\sEng$ and suppose $0=\Eng{f}'(x)=2\f'(x)\bigl(\f(x)+\f''(x)\bigr)$.
However $\abs{\f(x)}^2\le\Eng{f}(x)<\kappa^2$ so $\f(x)+\f''(x)\neq0$. Thus $\f'(x)=0$. 

For part (ii) write $I=(s,t)$. Then $\Eng{f}(s)=\kappa^2=\Eng{f}(t)$ so $\Eng{f}'(x)=0$ for some $x\in I$ and hence $x\in\tp$ by part (i) (note that, $\Eng{f}(0)=(1+a)^2>1>c_0^2>\kappa^2$ so we can't have $s=0$). If there were distinct points $x_1,x_2\in I\cap\tp$ then we could find $x_0$ between $x_1$ and $x_2$ (and hence in $I$) with $\f''(x_0)=0$, contradicting the fact that $\abs{f''(x_0)}\ge\kappa$. 

For part (iii) note that we have $\abs{\f[\phi]''(x)}\ge\abs{\f''(x)}-\abs{\phi''(x)}>\kappa-\kappa=0$ for all $x\in I$.
\end{proof}

\subsubsection*{Tangential zeros}

In this section we use the notation of Theorem \ref{thm:coszeros}. 

\begin{lemma}
\label{lem:tangzsbrat}
Suppose $\beta\in\Q$. If $\f(x)=0=\f'(x)$ for some $x$ then $p_\beta+q_\beta\nu_{\alpha,\beta}\in4\Z$.
\end{lemma}

\begin{proof}
Suppose $\f(x)=0=\f'(x)$ so $\cos(x)=-\alpha\cos(\beta x)$ and $\sin(x)=-\alpha\beta\sin(\beta x)$.
In particular $\cos(x)$ and $\cos(\beta x)$ have opposite signs, as do $\sin(x)$ and $\sin(\beta x)$. It follows that $x$ and $\beta x$ lie in diametrically opposite quadrants. 

From \eqref{eq:sin2matsol} (with $v_0=0=v_1$) we get
\[
\sin^2(x)=\frac{\beta^2(1-\alpha^2)}{\beta^2-1}
\quad\text{and}\quad
\sin^2(\beta x)=\frac{1-\alpha^2}{\alpha^2(\beta^2-1)}.
\]
Hence
\[
x\in(\xi'+\pi\Z)\cup(-\xi'+\pi\Z)
\quad\text{and}\quad
\beta x\in(\eta+\pi\Z)\cup(-\eta+\pi\Z).
\]
Since $\xi',\eta\in(0,\pi/2)$ this can be combined with the earlier observation to give one of four possibilities; 
\[
\begin{cases}
x\in\xi'+\pi2\Z,&\beta x\in\eta+\pi(2\Z+1),\\
x\in\xi'+\pi(2\Z+1),&\beta x\in\eta+\pi2\Z,\\
x\in-\xi'+\pi2\Z,&\beta x\in-\eta+\pi(2\Z+1),\\
x\in-\xi'+\pi(2\Z+1),&\beta x\in-\eta+\pi2\Z.
\end{cases}
\]
Comparing the expressions for $x$ and $\beta x$ we can thus find integers $m,n\in\Z$ of opposite parity such that $\beta(\xi'+m\pi)=\eta+n\pi$. Then
\[
\beta+\nu_{\alpha,\beta}=\beta+\frac{2}{\pi}(\beta\xi+\eta)
=2\beta+\frac{2}{\pi}(-\beta\xi'+\eta)
=2\bigl((1+m)\beta-n\bigr).
\]
However $\beta=p_\beta/q_\beta$ so
\[
p_\beta+q_\beta\nu_{\alpha,\beta}=q_\beta(\beta+\nu_{\alpha,\beta})=2\bigl((1+m)p_\beta-nq_\beta\bigr).
\]
Now $p_\beta$ and $q_\beta$ have the same parity, as do $1+m$ and $n$. Therefore $(1+m)p_\beta-nq_\beta\in2\Z$ and hence $p_\beta+q_\beta\nu_{\alpha,\beta}\in4\Z$.
\end{proof}

If $\beta\in\Q$ then $\Eng{f}$ is smooth, non-negative and periodic, so $\Eng{f}$ can be uniformly bounded away from $0$ if it is nowhere zero. 
Lemma \ref{lem:tangzsbrat} thus leads to the following.

\begin{corollary}
\label{cor:bratEngunifnz}
Suppose $\beta\in\Q$. If $p_\beta+q_\beta\nu_{\alpha,\beta}\notin4\Z$ then there exists $\delta>0$ such that $\Eng{\f}(x)\ge\delta$ for all $x\in\R$.
\end{corollary}

\subsection{Perturbations}
\label{sec:zpert}

Suppose $\phi$ satisfies \eqref{eq:decayprop01}. 
Choose a decreasing function $\sigma:\R\to\R^+$ so that $\Eng{\phi}(x)<\sigma^2(x)$ for all $x$ and $\sigma(x)\to0$ as $x\to\infty$. 
Set 
\[
\stp=\bigl\{t_n:n\ge1,\,\Eng{f}(t_n)<\sigma^2(t_{n-1})\bigr\};
\]
these are the turning points of $\f$ which are `small' in some sense (relative to $\phi$) and can cause changes in the number of zeros when $\phi$ is added to $\f$. 

Choose $m_0>1$ so that $\sigma(t_{m_0-2})<\kappa$ (which is possible since $\sigma(x)\to0$ as $x\to\infty$). Now suppose $t_n\in\stp$ for some $n\ge m_0-1$. Then 
\[
\Eng{f}(t_n)<\sigma^2(t_{n-1})\le\sigma^2(t_{m_0-2})<\kappa
\]
so $t_n\in\sEng$. Let $\Istp{n}$ denote the maximal connected component of $\sEng$ which contains $t_n$. 
If $t_n\notin\stp$ for some $n\ge m_0-1$ set $\Istp{n}=\emptyset$. 

For any $n\ge m_0$ let $\rItp{n}=\Itp{n}\setminus(\Istp{n-1}\cup\Istp{n})$.
Firstly note that $\rItp{n}$ is non-empty (as otherwise we would have $\Itp{n}\subseteq\Istp{n-1}\cup\Istp{n}\subseteq\sEng$, 
implying that $\sEng$ has a connected component containing the distinct elements $t_{n-1},t_n$ of $\tp$). 
Also $\rItp{n}$ is an interval (removal of $\Istp{n-1}$ and $\Istp{n}$ could only split the interval $\Itp{n}$ if either $\Istp{n-1}\subset\Itp{n}$ or $\Istp{n}\subset\Itp{n}$ which, in turn, is only possible if $\Istp{n-1}=\emptyset$ or $\Istp{n}=\emptyset$).

\begin{lemma}
\label{lem:minKn}
Let $n\ge m_0$ and suppose $I$ is a closed and bounded sub-interval of $\Itp{n}$ 
with $\Eng{\f}(x)<\kappa$ for some $x\in I$. 
Then the minimum of $\Eng{\f}$ on $I$ occurs at an endpoint.
\end{lemma}

\begin{proof}
Suppose the minimum of $\Eng{\f}$ on $I$ occurs at $x_0$ which is an interior point of $I$. Then $x_0$ is also in the interior of $\Itp{n}$ and hence $x_0\notin\tp$. On the other hand, we must have $\Eng{\f}'(x_0)=0$ and $\Eng{\f}(x_0)<\kappa$, so $x_0\in\tp$ by Lemma \ref{lem:propssEng}(i).
\end{proof}

\begin{lemma}
\label{lem:Enf>sig}
For any $n\ge m_0$ we have $\Eng{\f}(x)\ge\sigma^2(x)$ for all $x\in\rItp{n}$. 
\end{lemma}

\begin{proof}
Set $\rItp{n}=[s,t]$ and suppose $\Eng{\f}(x)<\sigma^2(x)$ for some $x\in\rItp{n}$. Since $x\ge t_{n-1}>t_{m_0-2}$ we get $\Eng{\f}(x)<\sigma^2(t_{m_0-2})<\kappa$.
Lemma \ref{lem:minKn} then shows that the minimum of $\Eng{\f}$ on $\rItp{n}$ must occur at either $s$ or $t$. 
Now $\Eng{\f}(s)=\kappa$ if $t_{n-1}\in\stp$ while $s=t_{n-1}$ if $t_{n-1}\notin\stp$. 
However $\kappa$ is not the minimum value of $\Eng{\f}$ on $\rItp{n}$, while having $\Eng{\f}(t_{n-1})$ as the minimum value would imply
\[
\Eng{\f}(t_{n-1})\le\Eng{\f}(x)<\sigma^2(x)\le\sigma^2(t_{n-1})\le\sigma^2(t_{n-2}),
\]
leading to the contradiction $t_{n-1}\in\stp$. A similar argument shows that the minimum value of $\Eng{\f}$ on $\rItp{n}$ can't occur at $t$. 
\end{proof}

If $n\ge1$ then $\f'(x)\neq0$ for any $x\in(t_{n-1},t_n)$ so $\f$ can have at most 1 zero on $\Itp{n}$; that is, $\numz[0](\Itp{n})\le1$. Since $\Eng{\phi}(x)<\sigma^2(x)$ for all $x$ we immediately get the following corollary of Lemmas \ref{lem:genpertz1}, \ref{lem:genpertz2} and \ref{lem:Enf>sig}.

\begin{corollary}
\label{cor:easyintcomp}
Suppose $n\ge m_0$. Then $\numz(\rItp{n})\le1$. Furthermore, if $\rItp{n}=\Itp{n}$ then $\numz(\Itp{n})=\numz[0](\Itp{n})$; in this case neither $\f$ nor $\f[\phi]$ can have a zero at an endpoint of $\Itp{n}$.
\end{corollary}

Note that, the requirement $\rItp{n}=\Itp{n}$ is equivalent to $\Istp{n-1}=\emptyset=\Istp{n}$, or $\Itp{n}\cap\stp=\emptyset$.

\subsubsection*{Proof of Theorem \ref{thm:coszeros} when $\alpha\beta>1$, $\beta\in\Q$}

Suppose $\alpha\beta>1$ and $\beta\in\Q$. Theorem \ref{thm:coszeros} for general $\phi$ then follows from the case $\phi\equiv0$ and the following result.

\begin{proposition}
Suppose $\phi$ satisfies \eqref{eq:decayprop01} and $\f[\phi]^{-1}(0)$ is a discrete subset of $\R$. Then
\[
\lim_{R\to\infty}\frac{\abs{\numz(R)-\numz[0](R)}}{R}=0.
\]
\end{proposition}

\begin{proof}
Using Corollary \ref{cor:bratEngunifnz} we can choose $m\ge m_0$ so that $\Eng{\f}(x)\ge\sigma^2(t_{m-1})$ for all $x$. 
Then $\stp\cap[t_m,\infty)=\emptyset$ so $\rItp{n}=\Itp{n}$ for any $n>m$.
Corollary \ref{cor:easyintcomp} then gives $\numz([t_m,t_M])=\numz[0]([t_m,t_M])$ for any $M>m$.  
On the other hand, $\f[0]^{-1}(0)$ and $\f[\phi]^{-1}(0)$ are discrete subsets of $\R$ so $\numz[0]([0,t_m))$ and $\numz([0,t_m))$ are both finite. 
Thus $\bigabs{\numz(t_M)-\numz[0](t_M)}=O(1)$ as $M\to\infty$. Then
\[
\lim_{R\to\infty}\frac{\abs{\numz(R)-\numz[0](R)}}{R}
=\lim_{M\to\infty}\frac{\abs{\numz(t_M)-\numz[0](t_M)}}{t_M}
=0
\]
with the help of Lemmas \ref{lem:numzRRseqlim} and \ref{lem:basictplem}.
\end{proof}

\subsubsection*{Distribution of points in $\stp$}

\begin{lemma}
\label{lem:minsep}
There exists a constant $c>0$ such that $\abs{t-s}\ge c$ for all distinct $s,t\in\stp$.
\end{lemma}

\begin{proof}
Let $s,t\in\stp$ with $s<t$. Now $s,t\in\sEng$ so $\abs{\f''(s)},\,\abs{\f''(t)}\ge\kappa$. 
Also $\f'(s)=0=\f'(t)$ so we can find $x_0\in(s,t)$ with $\f''(x_0)=0$. Then the total variation of $\f''$ between $s$ and $t$ is at least $2\kappa$.
However $\abs{\f'''(x)}\le1+\alpha\beta^3$ for all $x$; thus $t-s\ge2\kappa/(1+\alpha\beta^3)$. 
\end{proof}

Split $\stp$ into a pair of increasing sequences of distinct points $(s^+_n)_{n\ge1}$ and $(s^-_n)_{n\ge1}$ so that $s^+_n\in[0,\pi/2)+\pi\Z$ and $s^-_n\in[\pi/2,\pi)+\pi\Z$ for all $n$.

\begin{lemma}
\label{lem:limdiffs+n}
Suppose $(s^+_n)_{n\ge1}$ is an infinite sequence. Then $s^+_n-s^+_{n-1},\;b(s^+_n-s^+_{n-1})\longrightarrow\pi\N$ as $n\to\infty$. A similar result hold for $(s^-_n)_{n\ge1}$.
\end{lemma}

If $(x_n)_{n\ge1}$ is a sequence and $X\subseteq\R$ then $x_n\to X$ simply means $\dist(x_n,X)\to0$.

\begin{proof}[Proof of Lemma \ref{lem:limdiffs+n}]
We have $\f'(s^+_n)=0$ for all $n$ and $\abs{\f(s^+_n)}^2=\Eng{\f}(s^+_n)\le\sigma^2(s^+_{n-1})\to0$ as $n\to\infty$. From \eqref{eq:sin2matsol} it follows that
\[
\sin^2(s^+_n)\longrightarrow\frac{\beta^2(1-\alpha^2)}{\beta^2-1}
\quad\text{and}\quad
\sin^2(\beta s^+_n)\longrightarrow\frac{1-\alpha^2}{\alpha^2(\beta^2-1)}
\]
as $n\to\infty$. Hence
\[
s^+_n\longrightarrow(\xi'+\pi\Z)\cup(-\xi'+\pi\Z)
\quad\text{and}\quad
\beta s^+_n\longrightarrow(\eta+\pi\Z)\cup(-\eta+\pi\Z)
\]
as $n\to\infty$ (recall \eqref{eq:alphabetadef} and \eqref{eq:alpha'beta'def} for the definitions of $\xi'$ and $\eta$). However, $s^+_n\in[0,\pi/2)+\pi\Z$ for all $n$, while $\cos(s^+_n)=-\alpha\cos(\beta s^+_n)+\f(s^+_n)$ and $\sin(s^+_n)=-\alpha\beta\sin(\beta s^+_n)$ so (modulo $2\pi$) $\beta s^+_n$ must tend to a quadrant diametrically opposite $s^+_n$. Hence $s^+_n,\,\beta s^+_n\longrightarrow[0,\pi/2]+\pi\Z$ as $n\to\infty$. Comparing with the previous expressions we then get
\[
s^+_n\longrightarrow\xi'+\pi\Z
\quad\text{and}\quad
\beta s^+_n\longrightarrow\eta+\pi\Z
\quad\Longrightarrow\quad
s^+_n-s^+_{n-1},\;\beta(s^+_n-s^+_{n-1})\longrightarrow\pi\Z
\]
as $n\to\infty$ (note that $\xi',\eta\in(0,\pi/2)$). Finally note that Lemma \ref{lem:minsep} gives $s^+_n-s^+_{n-1}\ge c>0$ for all $n$, so we can replace the final $\pi\Z$ with $\pi\N$.
\end{proof}

\begin{lemma}
\label{lem:bnddiffbirr}
If $\beta\notin\Q$ then $\#\{n:s^+_n-s^+_{n-1}\le C\}<\infty$ for all constants $C$.
A similar result holds for $(s^-_n)_{n\ge1}$.
\end{lemma}

\begin{proof}
Suppose $\#\{n:s^+_n-s^+_{n-1}\le C\}=\infty$ for some $C$. 
Then we can find $x\in[c,C]$ and a sub-sequence $s^+_{n_k}$ such that $s^+_{n_k}-s^+_{n_k-1}\to x$ as $k\to\infty$. Lemma \ref{lem:limdiffs+n} then gives $x=q\pi$ and $\beta x=p\pi$ for some $p,q\in\N$, so $\beta=p/q\in\Q$.
\end{proof}

\begin{lemma}
\label{lem:genunbnddiff}
If $(x_n)_{n\ge1}$ is an increasing positive sequence with $\#\{n:x_n-x_{n-1}\le C\}<\infty$ for all $C$ then $\#\{n:x_n\le R\}=o(R)$ as $R\to\infty$. 
\end{lemma}

\begin{proof}
Let $\epsilon>0$ and choose $N$ so that $x_n-x_{n-1}\le1/\epsilon$ for all $n>N$. Then, for all $m\ge1$, $x_{N+m}\ge x_N+m/\epsilon$. Hence
\begin{align*}
\#\{n\ge1:x_n\le R\}
&\le N+\#\{m\ge1:x_{N+m}\le R\}\\
&\le N+\#\{m\ge1:x_N+m/\epsilon\le R\}
\le N+\epsilon(R-x_N)
\end{align*}
whenever $R\ge x_N$. Thus
\[
\limsup_{R\to\infty}\frac1{R}\#\{n\ge1:x_n\le R\}
\le\limsup_{R\to\infty}\left(\epsilon+\frac{N-\epsilon x_N}{R}\right)
=\epsilon.
\]
Taking $\epsilon\to0^+$ completes the result.
\end{proof}

Since $\#(\stp\cap[0,R])=\#\{n:s^+_n\le R\}+\#\{n:s^-_n\le R\}$ Lemmas \ref{lem:bnddiffbirr} and \ref{lem:genunbnddiff} immediately lead to the following.

\begin{corollary}
\label{cor:diststp}
If $\beta\notin\Q$ then $\#(\stp\cap[0,R])=o(R)$ as $R\to\infty$.
\end{corollary}

\subsubsection*{Proof of Theorem \ref{thm:coszeros} when $\alpha\beta>1$, $\beta\notin\Q$}

If $\phi$ satisfies \eqref{eq:decayprop012} we can choose $m\ge m_0$ so that $\abs{\phi''(x)}<\kappa$ for all $x\ge t_{m-1}$. Lemma \ref{lem:propssEng}(iii) immediately gives the following (note that, if $\Istp{n}=\emptyset$ the result is trivial).

\begin{corollary}
\label{cor:dzIstp}
If $n\ge m$ then $\numz(\Istp{n})\le2$. 
\end{corollary}

Now suppose $\alpha\beta>1$ and $\beta\notin\Q$. Theorem \ref{thm:coszeros} for general $\phi$ then follows from the case $\phi\equiv0$ and the following result.

\begin{proposition}
Suppose $\phi$ satisfies \eqref{eq:decayprop012} and $\f[\phi]^{-1}(0)$ is a discrete subset of $\R$. Then
\[
\lim_{R\to\infty}\frac{\abs{\numz(R)-\numz[0](R)}}{R}=0.
\]
\end{proposition}

\begin{proof}
If $M>m$ then
\[
[t_m,t_M]=\bigcup_{m<n\le M}\Itp{n}'
=\bigcup_{m<n\le M}\rItp{n}\ \cup\bigcup_{m<n<M}\Istp{n}\ \cup\Istp{m}'\cup\Istp{M}',
\]
where $\Itp{n}'=\Itp{n}\setminus\{t_n\}=[t_{n-1},t_n)$ for $m<n<M$, $\Itp{M}'=\Itp{M}$, $\Istp{m}'=\Istp{m}\cap[t_m,\infty)$ and $\Istp{M}'=\Istp{M}\cap[0,t_M]$. Furthermore the intervals in the first covering are disjoint while those in the second covering can only overlap at points $t_n\in\tp\setminus\stp$; for such points
\[
\f^2(t_n)=\Eng{\f}(t_n)\ge\sigma^2(t_{n-1})\ge\sigma^2(t_n)>\Eng{\phi}(t_n)\ge\phi^2(t_n)\ge0
\]
so $t_n$ is not zero of $\f[\phi]$. Therefore
\[
\numz[0]([t_m,t_M])=\sum_{m<n\le M}\numz[0](\Itp{n}')
\]
and
\[
\numz([t_m,t_M])=\sum_{m<n\le M}\numz[0](\rItp{n})+\sum_{m<n<M}\numz(\Istp{n})
\ +\numz(\Istp{m}')+\numz(\Istp{M}').
\]
Now set
\[
\nstp=\#\{m\le n\le M:\Istp{n}\neq\emptyset\}=\#(\stp\cap[t_m,t_M]).
\]
By Corollary \ref{cor:dzIstp}
\[
\sum_{m<n<M}\numz(\Istp{n})\ +\numz(\Istp{m}')+\numz(\Istp{M}')
\le\sum_{m\le n\le M}\numz(\Istp{n})
\le2\nstp.
\]
By Corollary \ref{cor:easyintcomp} (and the discussion proceeding it) $\numz(\rItp{n})-\numz(\Itp{n}')=0$ if $\rItp{n}=\Itp{n}$, while $\abs{\numz(\rItp{n})-\numz(\Itp{n}')}\le 1$ in general. 
Since $\#\bigl\{m<n\le M:\rItp{n}\neq\Itp{n}\bigr\}\le2\nstp$ we then get
\[
\Biggl\lvert\sum_{m<n\le M}\numz[0](\rItp{n})-\sum_{m<n\le M}\numz[0](\Itp{n}')\Biggr\rvert
\le2\nstp.
\]
Combining the above estimates now gives
\[
\bigabs{\numz([t_m,t_M])-\numz[0]([t_m,t_M])}\le4\nstp.
\]
However $\nstp\le\#(\stp\cap[0,t_M])$ so Corollary \ref{cor:diststp} (together with Lemma \ref{lem:basictplem}) implies
\[
\bigabs{\numz([t_m,t_M])-\numz[0]([t_m,t_M])}=o(t_M)
\]
as $M\to\infty$. 
On the other hand, $\f[0]^{-1}(0)$ and $\f[\phi]^{-1}(0)$ are discrete subsets of $\R$ so $\numz[0]([0,t_m))$ and $\numz([0,t_m))$ are both finite. 
Thus $\bigabs{\numz(t_M)-\numz[0](t_M)}=o(t_M)$ as $M\to\infty$ and so
\[
\lim_{R\to\infty}\frac{\abs{\numz(R)-\numz[0](R)}}{R}
=\lim_{M\to\infty}\frac{\abs{\numz(t_M)-\numz[0](t_M)}}{t_M}
=0
\]
with the help of Lemmas \ref{lem:numzRRseqlim} and \ref{lem:basictplem}.
\end{proof}

\begin{remark} 
\label{sec:zlimit}
It is instructive to look at the limiting cases in Theorem \ref{thm:coszeros}.
When $\alpha=1$ we have
\[
\f(x)=\cos(x)+\cos(\beta x)=2\cos\Bigl(\frac{\beta+1}2\,x\Bigr)\cos\Bigl(\frac{\beta-1}2\,x\Bigr).
\]
The zeros of this function occur precisely when $(\beta\pm1)x\in(2\Z+1)\pi$. If $\beta\neq p/q$ for some $p,q\in\N$ with opposite parity then all zeros of $\f$ are simple and
\[
\lim_{R\to\infty}\frac{\numz[0](R)}{R}
=\frac1\pi\Bigl(\frac{\beta+1}2+\frac{|\beta-1|}2\Bigr)
=\frac1\pi\begin{cases}
1&\text{if $\beta<1$,}\\
\beta&\text{if $\beta>1.$}
\end{cases}
\]
The same formula holds for arbitrary $\beta$ if we count zeros with multiplicity. 
This agrees with Theorem \ref{thm:coszeros} and the limiting behaviour of $\nu_{\alpha,\beta}$ as $\alpha\to1^-$.

\smallskip

At the same time, Theorem \ref{thm:coszeros} does not extend in a straightforward manner to the case $\alpha\beta=1$.
This can be seen by taking $\alpha=1/3$, $\beta=3$; then
\[
\f(x)=\cos(x)+\frac13\,\cos(3x)=\frac43\,\cos^3(x).
\]
Although the zeros of this function are precisely the points $\bigl(n+\frac12\bigr)\pi$ for $n\in\Z$, we also have
\[
\f'\Bigl(\Bigl(n+\frac12\Bigr)\pi\Bigr)=\f''\Bigl(\Bigl(n+\frac12\Bigr)\pi\Bigr)=0.
\]
It is then straightforward to construct a perturbation $\phi$ satisfying \eqref{eq:decayprop012} so that $\f[\phi]=\f+\phi$ has arbitrarily many zeros close to $\bigl(n+\frac12\bigr)\pi$ for each $n\in\Z$.
\end{remark}


\begin{thebibliography}{GGHKSSV}

\bibitem[Ad]{Adams} R.~A.~Adams, \emph{Sobolev Spaces}, Pure and Applied Mathematics \textbf{65}, Academic Press, New York-London, 1975.

\bibitem[AMS]{AMS} J.~Avron, P.~H.~M.~v.~Mouche, B.~Simon, \emph{On the measure of the Spectrum for the almost Mathieu operator}, Commun.\ Math.\ Phys.\ \textbf{132}, 103--118 (1990).

\bibitem[BaTiBr]{BT} J. H.~Bardarson, M.~Titov, P. W. ~Brouwer, \emph{Electrostatic Confinement of Electrons in an Integrable Graphene Quantum Dot}, 
Phys.\ Rev.\ Lett.\ \textbf{102}, 226803 (2009).

\bibitem[BiLa]{BL} M.~S.~Birman, A.~Laptev, \emph{Discrete spectrum of the perturbed Dirac operator}, Ark.\ Mat.\ \textbf{32}, 13--32 (1994).

\bibitem[BiSo]{BS} M.~S.~Birman, M. Z.~Solomyak, \emph{Spectral asymptotics of pseudodifferential 
operators with  anisotropic homogeneous symbols}, I, Vestnik Leningrad.\  Univ.\ 
Mat.\  Mekh.\  Astronom.\  \textbf{13},  13--21  (1977);  II,  Vestnik  Leningrad.\  Univ.\ Mat.\ 
Mekh.\  Astronom.\  \textbf{13},  5--10 (1979) (Russian). 

\bibitem[BrFr]{BF} L.~Brey, H.~A.~Freitig, \emph{Emerging Zero Modes for Graphene in a Periodic Potential},  Phys.\ Rev.\ Lett.\ \textbf{103}, 046809 (2009).

\bibitem[Cw]{Cw} M.\ Cwikel, \emph{Weak type estimates for singular values and the number of bound states of Schr\"odinger operators}, Ann.\ of Math.\ \textbf{106}, 93--100 (1977).

\bibitem[GGL]{GGL} K.~Golden, S.~Goldstein, J.L.~Lebowitz, \emph{Classical transport in modulated structures}, Phys.\ Rev.\ Lett.\  \textbf{55}, no. 24, 
2629--2632 (1985).

\bibitem[ElTa]{EltonTa} D.~M.~Elton, N.~T.~Ta, \emph{Eigenvalue Counting Estimates for a Class of Linear Spectral Pencils with Applications to Zero Modes}, J.\ Math.\ Anal.\ Appl.\ \textbf{391}, 613--618 (2012) .

\bibitem[GGHKSSV]{GGHKSSV} F.~Gesztesy,  D.~Gurarie, H.~Holden, M.~Klaus, L.~Sadun, B.~Simon, P.~Vogl, \emph{Trapping and cascading eigenvalues in the large coupling limit}, Commun.\ Math.\ Phys.\ \textbf{118}, 597--634 (1988).

\bibitem[Ha]{Hartman} P.~Hartman, \emph{Ordinary Differential Equations}, John Wiley and Sons, New York, 1964.

\bibitem[HRP]{HRP} R.~R.~Hartmann, N.~J.~Robinson, M.~E.~Portnoi, \emph{Smooth electron waveguides in graphene}, Phys.\ Rev.\ B \textbf{81}, 245431 (2010) .

\bibitem[JM]{JM} S.~Jitomirskaya, C.~A.~Marx, \emph{Analytic quasi-periodic Schrödinger operators and rational frequency approximants}, Geom.\ Funct.\ Anal.\ \textbf{22}, no.\ 5, 1407--1443 (2012).

\bibitem[Kac]{Kac} M.~Kac,  \emph{On the distribution of values of trigonometric sums with linearly independent frequencies},  Amer.\ J.\ Math.\ \textbf{65}, no.\ 4, 609--615 (1943).

\bibitem[KKW]{KKW} M.~Kac, E.~R.~van Kampen, A.~Wintner, \emph{On the distribution of the values of real almost periodic functions}, 
Amer.\ J.\ Math.\ \textbf{61}, 985--991 (1939).

\bibitem[Kl]{K} M.~Klaus, \emph{On the point spectrum of Dirac operators}, Helv.\ Phys.\ Acta \textbf{53}, 453--462 (1980).

\bibitem[ReSi]{RSI} M.~Reed, B.~Simon, \emph{Methods of Modern Mathematical Physics IV: Analysis of Operators}, Academic Press, San Diego, 1978.

\bibitem[Ru]{R} W.~Rudin, \emph{Real and Complex Analysis}, 3rd Edition, McGraw-Hill, Singapore, 1987.

\bibitem[Sa]{Sa} O.~Safronov, \emph{The discrete spectrum of selfadjoint perturbations of variable sign},  Commun.\ PDE \textbf{26}, no.\ 3-4,  629--649 (2001).

\bibitem[Sch]{Schm} K.\ M.~Schmidt, \emph{Spectral properties of rotationally symmetric
massless Dirac operators}, Lett.\ Math.\ Phys.\ \textbf{92}, 231--241 (2010).

\bibitem[StSh]{StSh} E.~M.~Stein, R.~Shakarchi, \emph{Fourier Analysis: An Introduction}, Princeton University Press, Princeton, 2003.

\bibitem[St]{St} P.\ Stein, \emph{On the Real Zeros of a certain Trigonometric Function}, Math. Proc. Cambridge Philos. Soc.\ \textbf{31}, 455--467 (1935).

\bibitem[StDoPo]{SDP} D. A. Stone, C. A. Downing, M. E. Portnoi, \emph{ Searching for confined modes in graphene channels: The variable phase method}, 
Phys.\ Rev.\ B \textbf{86}, 075464 (2012).

\bibitem[WeSe]{WS} E.~Wegert, G.~Semmler, \emph{Phase plots of complex functions: a journey in illustration}, Notices AMS \textbf{58}, 768--780 (2010).

\bibitem[Wey]{Weyl} H.~Weyl, \emph{Inequalities between two kinds of eigenvalues of a linear transformation}, Proc.\ Natl.\ Acad.\ Sci.\ USA \textbf{35}, 408--411 (1949).

\end{thebibliography}
\end{document}